\documentclass[12pt]{article}
\usepackage[utf8]{inputenc}
\usepackage[english]{babel}
\usepackage[T1]{fontenc}
\usepackage{hyperref}
\usepackage{amsmath}
\usepackage{amsthm}
\usepackage{amsfonts}
\usepackage{amssymb}
\usepackage{graphicx}
\numberwithin{equation}{section}
\usepackage{graphicx}
\usepackage{tikz}
\usetikzlibrary{calc}
\usepackage{framed}
\usepackage{caption}
\usepackage{gensymb}
\usetikzlibrary{arrows}
\tikzset{>=latex}
\usetikzlibrary{patterns}
\usepackage{authblk}
\usepackage[top=2.5cm, bottom=2.5cm, left=2.5cm, right=2.5cm]{geometry}
\usepackage{fancyhdr}
\usepackage{cuted}
\usepackage{appendix}
\usepackage{chngcntr}
\newtheorem{theo}{Theorem}[section]
\newtheorem{prop}[theo]{Proposition}
\newtheorem{coro}[theo]{Corollary}
\newtheorem{lemm}[theo]{Lemma}

\newtheorem{rmk}[theo]{Remark}
\theoremstyle{definition}
\newtheorem*{question}{Question}

\newcommand{\N}{\mathbb{N}}

\newcommand{\R}{\mathbb{R}}
\newcommand{\C}{\mathcal{C}}
\newcommand{\Z}{\mathbb{Z}}

\newcommand{\Prob}{\mathbb{P}}
\newcommand{\E}{\mathbb{E}}
\newcommand{\e}{\mathcal{E}}

\newcommand{\T}{\mathbb{T}}

\newcommand{\s}{\mathcal{S}}

\renewcommand{\epsilon}{\varepsilon}

\newcommand\xqed[1]{%
  \leavevmode\unskip\penalty9999 \hbox{}\nobreak\hfill
  \quad\hbox{#1}}
\newcommand\demo{\xqed{$\blacksquare$}}

\title{Fluctuations and correlations in weakly asymmetric simple exclusion on a ring subject to an atypical current}

\author{Benoit Dagallier\footnote{Courant Institute of Mathematical Sciences, New York University. E-mail: {\tt bd2543@nyu.edu}}}

\date{}

\begin{document}
\maketitle

\begin{abstract}
We consider the weakly asymmetric simple exclusion process on a ring, 
driven out of equilibrium by tilting the dynamics so as to enforce a macroscopic current of particles on a large time interval. 
In this current-biased dynamics, the tilt by the current makes the dynamics non-local, non homogeneous and induces long-range correlations. 
We compute the correlation structure in the large time, large size limit for a certain range of asymmetry and current strength, 
recovering heuristic results of Bodineau et al.~\cite{Bodineau2008}. 
In addition, in this range of parameters, 
we characterise the full dynamics of fluctuations around the optimal density profile in the current-biased dynamics. 
The key ingredient at the microscopic scale is a precise relative entropy estimate at the level of correlations. 
We also discuss how to remove the (technical) restriction on the range of parameters. 
\end{abstract}

\section{Introduction}
Consider a large number $N$ of particles in a box, 
evolving in time and interacting locally. 
Assume that the dynamics of these particles can be modelled by a Markov chain $\Prob^N$ on a finite, but large state space $\Omega_N$ 
Consider also an observable $\mathcal O_T$ of the trajectory of the Markov chain on a time interval $[0,T]$, say of the following form:
if $(\eta_t)_{t\leq T}$ denotes a trajectory,
\begin{equation}
\mathcal O_T((\eta_t)_{t\leq T})
=
\int_0^T f(\eta_t)\, dt 
+
\sum_{t\leq T}g(\eta_{t_-},\eta_t)
.
\label{eq_def_observable_O_T}
\end{equation}
Above, $f:\Omega_N\to\R$ is a test function that could for instance be the density of particles in a box.  
The function $g:\Omega_N^2\to\R$ counts changes in the chain and could for instance correspond to the activity or current. 

Suppose also that we have some information on the observable $\mathcal O_T$ in the sense of a large deviation principle, 
written informally as follows:
\begin{equation}
\frac{1}{T}\log\Prob^N(\mathcal O_T \approx Tx) 
\underset{T\to\infty}{=} 
-I^N_{f,g}(x)
,\qquad 
x\in\R
,
\label{eq_large_devs_observable_intro}
\end{equation}
where $I^N_{f,g}\geq 0$ is a rate function that may be known e.g. through results of Donsker and Varadhan~\cite{Donsker1975}. 
The broad, informal question we wish to discuss is the following.\\
 
\begin{question}
Can one describe trajectories in the rare event $\{\mathcal O_T\approx Tx\}$ in the large size, long time limits?
\label{question}
\end{question} 
\subsection{Biased dynamics}
The description of trajectories in a rare dynamical event has received much attention in the physics community in recent years, see~\cite{BodineauDistributionCurrentNonequilibrium2005,
Bodineau2008,JackLargeDeviationsEnsembles2010,
Chetrite2015,Chetrite2015a,
BaekDynamicalSymmetryBreaking2017,
AgranovTricriticalBehaviorDynamical2023} for a non exhaustive list of references as well as the review~\cite{JackErgodicityLargeDeviations2020}. 
Studying trajectories in the rare event $\big\{\mathcal O_T\approx Tx\big\}$ when $N,T\gg 1$ boils down, 
following a standard large deviation idea, 
to studying typical trajectories under a tilted dynamics of the form (recall the expression~\eqref{eq_def_observable_O_T} of $\mathcal O_T$):
\begin{equation}
\mathrm{d}\Prob^{N,f,g}_{\lambda,T}\big((\eta_t)_{t\leq T}\big)
=
\frac{ \exp\big[ \lambda \mathcal O_T\big] }{\E^N\big[ \exp\big[ \lambda \mathcal O_T\big] \big]} \,\mathrm{d}\Prob^N\big((\eta(t))_{t\leq T}\big)
.
\label{eq_tilted_dynamics_general}
\end{equation}
If $\lambda\in\R$ is chosen appropriately as a function of $x$, 
then typical trajectories under~\eqref{eq_tilted_dynamics_general} when $N,T$ are large indeed correspond to trajectories in $\big\{\mathcal O_T\approx Tx\big\}$. 
We refer to~\cite{ChetriteNonequilibriumMicrocanonicalCanonical2013} for a discussion of the choice of $\lambda$. 
A schematic description of trajectories in~\eqref{eq_tilted_dynamics_general} is presented on Figure~\ref{fig_regions}, see also~\cite{Derrida_LD2019} for a more in-depth discussion.\\

Although the original dynamics $\Prob^N$ is a homogeneous Markov chain with local jump rates, 
the tilt by the observable $\mathcal O_T$ of~\eqref{eq_def_observable_O_T} 
makes the dynamics~\eqref{eq_tilted_dynamics_general} much more complicated: its jump rates are not explicit, now time-dependent and non-local. 
In addition, depending on the tilt strength $\lambda$, 
typical trajectories under the biased dynamics~\eqref{eq_tilted_dynamics_general} may have qualitatively different behaviour than under the original dynamics $\Prob^N$, 
for instance featuring dynamical phase transitions. 
Let us illustrate this point and more generally explain how the biased dynamics~\eqref{eq_tilted_dynamics_general} can be studied by considering the example of the Weakly Asymmetric Simple Exclusion Process (WASEP for short) on a ring, 
tilted by the current. 
This is the dynamics that we will focus on in this article. 
\subsection{Macroscopic fluctuation theory and driven process}
The WASEP on a ring with weak asymmetry $E\in\R$ is an interacting particle system where particles follow nearest-neighbour random walks on a discrete torus with $N$ sites, 
jumping to the right/left with rate proportional to $1+E/N$ and $1-E/N$ respectively.  
The only interaction between the particles comes from an exclusion rule: there can be at most one particle per site (see~\eqref{eq_def_generateur_SSEP} for a precise definition). 

Scaling limits of this model for both the density and current of particles have been studied in detail, see e.g.~\cite[Chapter 10]{Kipnis1999},~\cite{Bertini2007} and more recently~\cite{BertiniConcurrentDonskerVaradhan2023} for joint large deviations in the large time, large $N$ limits. 
In particular, it is known that there is typically a macroscopic current of particles on any time interval $[0,T]$. 
Also and interestingly, the dynamics is conjectured to undergo a dynamical phase transition.  
To state what this means, consider $\lfloor \rho N\rfloor$ particles in the system ($\rho\in(0,1)$) and the event:
\begin{equation}
\text{the time-integrated, space averaged current }Q_T\text{ is equal to }qT,\quad q\in\R
.
\label{eq_event_large_current}
\end{equation}
One may ask about the typical density profile in this event in the large $N$, large $T$ limit. 
The conjecture, formulated e.g. in~\cite{BertiniCurrentFluctuationsStochastic2005,
BodineauDistributionCurrentNonequilibrium2005} and supported by numerical simulations~\cite{EspigaresDynamicalPhaseTransition2013}, 
states that the result depends on the relationship between the tilt strength $\lambda$ and the asymmetry $E$:
\begin{align}	
	&\bullet\ \text{if }\frac{q^2}{\rho(1-\rho)}-E^2\rho(1-\rho)>-\pi^2\text{, then the optimal density profile is homogeneous in space}
	\nonumber\\
	&\hspace{0.58cm}\ \text{and time.}\label{eq_phiT}\\
	&\bullet \ \text{if }\frac{q^2}{\rho(1-\rho)}-E^2\rho(1-\rho)<-\pi^2\text{, then the optimal density profile is a travelling wave,}
	\nonumber\\
	&\quad \ \hspace{0.2cm}\text{inhomogeneous in time and/or space.}\nonumber
\end{align}
In~\cite{BodineauDistributionCurrentNonequilibrium2005} the conjecture~\eqref{eq_phiT} was formulated assuming no discontinuous phase transition take place.
Such transitions were later heuristically shown not to happen~\cite{BaekDynamicalSymmetryBreaking2017}. 
Note that in the symmetric simple exclusion process ($E=0$), 
~\eqref{eq_phiT} predicts that the optimal density profile is flat for any value of $\lambda$. 
This is known to be true, consistent with the fact that this model does not have a dynamical phase transition~\cite{Bertini2007}.
The conjecture~\eqref{eq_phiT} has not been established rigorously in either region. 
It is only known that far above the threshold the flat density profile is optimal, and that far below the threshold a travelling wave profile is better~\cite{Bertini2007} (but optimality of travelling waves is unknown). 

A second interesting property of trajectories in the event~\eqref{eq_event_large_current} is that one expects the current constraint to create a long-range spatial correlation structure. 
The two-point correlation structure in particular was computed heuristically in~\cite{Bodineau2008}, but nothing is known rigorously. \\

Let us now present different approaches employed in the literature to study trajectories in a rare event that can in particular be used to study the conjecture~\eqref{eq_phiT}. 
These approaches are based on the study of WASEP dynamics tilted by the time-integrated, space-averaged current $Q_T$; 
a particular example of~\eqref{eq_tilted_dynamics_general}.
This dynamics, referred to as the \emph{current-biased dynamics} throughout this article (see~
\eqref{eq_def_dynamique_conditionnee_courant} for a precise definition), reads:
\begin{equation}
\mathrm{d}\Prob^{curr,N}_{\lambda,E,T}
=
\frac{1}{\E^N_E\big[e^{\lambda NQ_T}\big]}\, e^{\lambda NQ_T}\mathrm{d}\Prob^N_E 
.
\label{eq_current_biased_dynamics_intro}
\end{equation}
Above, $\Prob^N_E,\E^N_E$ are the probability/expectations associated with the WASEP dynamics. 
In the exponential, the macroscopic current $Q_T$ is multiplied by a factor $N$. 
In view of large deviations for the current in the WASEP~\cite{Bertini2007}, 
this ensures that~\eqref{eq_current_biased_dynamics_intro} concentrates on trajectories in which the macroscopic current $Q_T$ takes different values from the one under $\Prob^N_E$.\\

A major difficulty in studying the current-biased dynamics~\eqref{eq_current_biased_dynamics_intro} comes from the fact that the state space is very large when $N\gg 1$. 
The first approach that can be used to study~\eqref{eq_current_biased_dynamics_intro} completely bypasses the large state space difficulty through the formalism of the Macroscopic Fluctuation Theory (MFT), 
see the review~\cite{Bertini2015}. 
There, one starts from a coarse grained, low-dimensional description of the microscopic dynamics, 
the details of which enter only through the choice of a diffusion coefficient and mobility. 
The long-time behaviour of the tilted dynamics is then directly studied at the level of the dynamical large deviation functional, 
see e.g.~\cite{BertiniCurrentFluctuationsStochastic2005,
BertiniNonEquilibriumCurrent2006,
ZarfatyStatisticsLargeCurrents2016} and the recent~\cite{AgranovTricriticalBehaviorDynamical2023}.  
Information on small fluctuations and correlations under~\eqref{eq_tilted_dynamics_general} can be deduced from formal expansions of the large deviation functional. 
This is in particular how two-point correlations were computed in the current-biased dynamics~\cite{Bodineau2008}. 
Note that in this approach, one first takes the large $N$ limit, then the large time limit.

The validity of the MFT has been established for a number of microscopic models (see the review~\cite{Bertini2015}, 
the book~\cite{Kipnis1999} as well as~\cite{Bertini2003,Bertini2007,Bertini2009LD_WASEP,
Farfan2011,
Bodineau2012}), 
by proving dynamical large deviations for the density and/or current of particles. 
Thus, at the level of density large deviations, the measure~\eqref{eq_current_biased_dynamics_intro} is well understood. 
In particular, for these microscopic models, 
dealing with the large size of the microscopic state seems to be a technical problem only, 
since the effective MFT description captures the physical features of the models. This should be the general picture.

On the other hand, the rigorous large deviation results do not give information on fluctuations, 
which require control of the dynamics on a finer scale. 
In particular and as already mentioned, 
there are no rigorous result on the correlation structure under the current-biased dynamics. 

\medskip

Another approach that has received a lot of attention in the physics literature consists in looking for an effective, but still microscopic description of the dynamics in the large time limit at fixed $N$. 
The macroscopic, $N\gg 1$ limit in this approach is thus taken second, after the long time limit.

Let us describe this approach in more detail. 
At each fixed $N$ but in the long time limit, 
typical trajectories in the rare event $\{Q_T\approx Tq\}$ are expected to feature an instantaneous current approximately equal to $q$ for all but a small portion of times in $[0,T]$. 
When the time $T$ is large, one then expects the dynamics~\eqref{eq_current_biased_dynamics_intro} to accurately be described in terms of a time-homogeneous Markovian dynamics with jump rates mimicking the effect of a constant current $q$. 
This homogeneous process, the \emph{driven process} in the language of~\cite{Chetrite2015} (see Figure~\ref{fig_regions}), 
can be built rigorously at each $N$, for general dynamics of the form~\eqref{eq_tilted_dynamics_general}, 
in terms of spectral elements of an explicit operator built from the initial dynamics and the observable by which the dynamics is biased.
\begin{figure}
\begin{center}
\includegraphics[width=16cm]{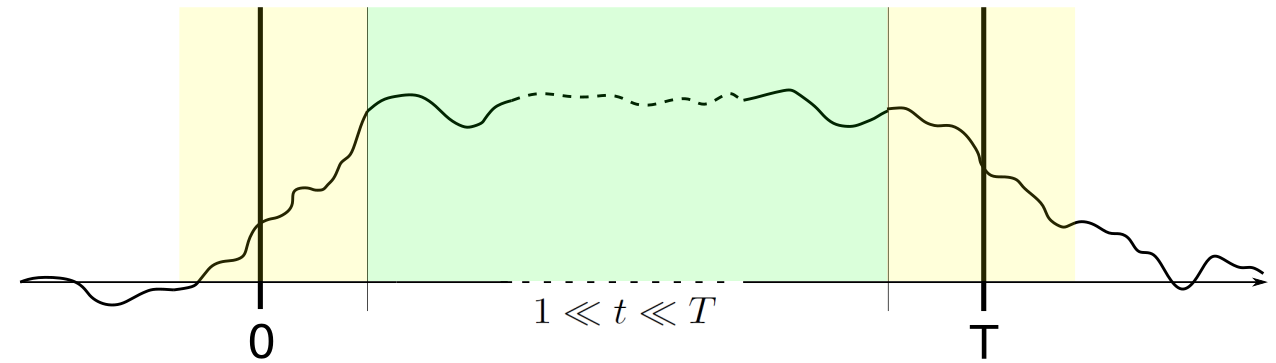} 
\caption{Schematic representation of a trajectory of the current-biased dynamics~\eqref{eq_current_biased_dynamics_intro}. 
At times far before $0$ and far after $T$, the influence of the current bias is not felt. 
In the yellow regions, the trajectory moves towards/away from a region of the state space with atypical current. It stays in this state (green region) for all but a negligible fraction of the time interval $[0,T]$ when $T\gg 1$.  \label{fig_regions}}
\end{center}
\end{figure}
The driven process approach was expounded in detail in~\cite{Chetrite2015,Chetrite2015a} and also~\cite{JackLargeDeviationsEnsembles2010,
ChetriteNonequilibriumMicrocanonicalCanonical2013,
JackEffectiveInteractionsLarge2015}, 
following earlier work, see e.g.~\cite{EvansRulesTransitionRates2004}.

One advantage of the driven process framework compared to the MFT is that it can be defined directly at the microscopic level.  
Taking the large $N$ limit is however difficult, 
as defining the driven process for $N$ large requires solving a high-dimensional spectral problem, 
the limit of which must then be studied. 
Results have mostly bypassed this difficulty by studying the driven process in low-dimensional systems (i.e. at the macroscopic level)~\cite{TsobgniLargeDeviationsCurrent2016,
TizonEffectiveDrivenDynamics2019,Derrida_fluct2019}. 
The recent paper~\cite{Derrida_LD2019} however managed to characterise the driven process in the large $N$ limit for two interacting particle systems: 
independent particles, 
and the symmetric simple exclusion process connected with boundary reservoirs. 
In this latter case, the reservoir densities are assumed to be small, 
which together with the symmetry of the original dynamics enabled the authors to solve the spectral problem defining the driven process in a perturbative fashion. 
In this perturbative regime, the scaling limit of the driven process is identified in~\cite{Derrida_LD2019} 
and it is shown that the macroscopic fluctuation theory and driven process approach are consistent. \\

While one expects long time and scaling limits to commute at both the large deviation and fluctuation level, consistent with the validity of the MFT picture;
proving the commutation is not an easy problem.  
It has been achieved only very recently, 
at the large deviation level and 
with a different motivation, 
for models such as the weakly asymmetric simple exclusion process on the on the torus in any dimension~\cite{BertiniConcurrentDonskerVaradhan2023} or the one-dimensional simple exclusion with boundary reservoirs~\cite{Correlations2022} 
(see also~\cite{BertiniLargeDeviationsDiffusions2022} for diffusions with small noise).  
It is not clear whether these commutation results are strong enough to apply also at the level of fluctuations.

\subsection{Approximate driven process}\label{sec_approx_driven_process_intro}
In this paper, 
we study small fluctuations around the typical density profile in the current-biased dynamics~\eqref{eq_current_biased_dynamics_intro}, 
which cannot be deduced from existing large deviation results.  
In a portion of the sub-critical regime~\eqref{eq_phiT} (i.e. before the dynamical phase transition), 
we obtain a full description of the non-local dynamics of density fluctuations at times $1\ll t\ll T$ (the green region of Figure~\ref{fig_regions}) and compute the two-point correlation structure in this same region. 
The expression of two-point correlations that we find agrees with the one derived in~\cite{Bodineau2008} through a heuristic expansion of the large deviation functional.  \\

Our approach builds on the driven process framework as in~\cite{Derrida_fluct2019,Derrida_LD2019}, 
but with a different flavour. 

The main difficulty in the driven process approach is to solve a high-dimensional spectral problem and study its large $N$ limit. 
We bypass this difficulty by building an \emph{approximate  driven process}, 
that is not the correct one for each fixed $N$ but has the correct macroscopic behaviour. 
This is done by using information on the macroscopic behaviour of the dynamics, so we take the large $N$ limit first before the large $T$ limit.  
The idea is that observables of interest, such as the density the current of particles or correlations, 
are determined at the macroscopic level by low-dimensional equations. 
These macroscopic equations are much easier to solve than the high-dimensional microscopic problem.  
One can then use the macroscopic data to build an informed microscopic approximation of the current-biased dynamics, 
the above mentioned approximate driven process (see Section~\ref{sec_driven_process} for more details).

\medskip

Compared to the genuine driven process, 
the downside is that our approximate driven process has less structure.  
We therefore need additional stability estimates at the microscopic level compared to~\cite{Derrida_LD2019} to ensure that our approximate driven process indeed describes the microscopic dynamics to the desired precision. 
These estimates turn out to be quite involved and unfortunately restrict the range of asymmetries $E$ and of the parameter $\lambda$ in~\eqref{eq_current_biased_dynamics_intro} that we can treat (extensions are discussed in Section~\ref{sec_perspectives}).

The main technical tool to prove that the approximate driven process is close to the microscopic dynamics is a very precise relative entropy estimate.   
The relative entropy method was initially introduced by Yau~\cite{Yau1991} to study hydrodynamic limits.  
It consists in comparing the law of the dynamics to a simple measure that only retains macroscopic information on the observable under study, 
usually the density of particles. 
Jara and Menezes~\cite{Jara2018,Jara2020} considerably improved the method, 
obtaining good enough estimates to study density fluctuations. 
In a recent paper~\cite{Correlations2022}, 
Bodineau and the author introduced a conceptual refinement of the relative entropy method, building on the techniques of Jara and Menezes. 
The main, though simple idea is that one can get better relative entropy bounds (meaning better control on the law of the dynamics) by adding more macroscopic information on the measure used for comparison, 
in particular by adding a correlation structure. 
A similar decomposition of the law of the dynamics was already suggested in~\cite{DemasiSmallDeviationsLocal1982}. 
The resulting better relative entropy bounds can be used in several contexts. 
They were for instance used in~\cite{Correlations2022} to obtain long time large deviations for two-point correlations, 
whereas in the present paper they are key to proving that our approximate driven process indeed accurately describes fluctuations under the current-biased dynamics. 
\\

The article is structured as follows. 
Section~\ref{sec_def} contains definitions and results, 
with Section~\ref{sec_structure_proof} detailing the structure of the proof of the main result, 
the characterisation of the density fluctuation process under the current-biased dynamics. 
Extensions of the result are discussed in Section~\ref{sec_perspectives}. 
Section~\ref{sec_fluc_correl_prob_h_lambda} contains a collection of preliminary results useful from Section~\ref{sec_driven_process_candidate} onwards where we start the proof of the main result. 
In particular, we study in Section~\ref{sec_fluc_correl_prob_h_lambda} fluctuations for a family of dynamics the approximate driven process is later shown to belong to.  
The main technical input, the relative entropy estimate, is also established in this section. 
Section~\ref{sec_driven_process_candidate} identifies a candidate for the approximate driven process and provides a stability estimate to ensure that this candidate indeed contains all information on fluctuations under the current biased dynamics. 
At this point the proof of our main can be reduced to estimates involving only the limiting ($N\gg 1$) fluctuations under the approximate driven process. 
Section~\ref{sec_time_decorrelation} contains the long-time analysis of these fluctuations and concludes the proof of the main result.
The two-point correlations structure is computed in Section~\ref{sec_comparison} and compared with~\cite{Bodineau2008}. 
Useful concentration of measure estimates and other technical results are gathered in the appendix.

\section{Definition and results}\label{sec_def}
\subsection{Model and notations}

\noindent \textbf{The WASEP.} For $N\in\N_{\geq 1}$, 
Let $\Omega_N = \{0,1\}^{\T_N}$ denote the state space of the system, with $\T_N:=\Z/N\Z$ the discrete torus with $N$ sites. 
A point $i\in\T_N$ is called a site, while elements $\eta = (\eta_i)_{i\in\T_N}\in\Omega_N$ are called (particle) configurations, 
with $\eta_i = 1$ if there is a particle at site $i\in\T_N$, and $\eta_i=0$ otherwise. 
The quantity $\eta_i$ is the occupation number at site $i\in\T_N$. 
For $\eta\in\Omega_N$ and $i,j\in \T_N$, define the configuration $\eta^{i,j}\in\Omega_N$ with exchanged occupation numbers at sites $i,j$:
\begin{equation}
\eta^{i,j}(k) = \begin{cases}
\eta(\ell)\quad &\text{if }\ell\notin\{i,j\},\\
\eta(j)\quad &\text{if }\ell=i,\\
\eta(i)\quad &\text{if }\ell=j.
\end{cases}
\end{equation}
The WASEP with asymmetry $E\in\R$ is the dynamics with generator $N^2L_E$, acting on $\phi:\Omega_N\rightarrow\R$ according to:
\begin{equation}
N^2L_E\phi(\eta) 
= 
N^2\sum_{i\in\T_N}c_E(\eta,\eta^{i,i+1})\big[\phi(\eta^{i,i+1}) - \phi(\eta)\big]
,
\label{eq_def_generateur_SSEP}
\end{equation}
with:
\begin{equation}
c_E(\eta,\eta^{i,i+1}) 
= 
c(\eta,\eta^{i,i+1})e^{-(\eta_{i+1}-\eta_i)E/N}
,\quad 
c(\eta,\eta^{i,i+1})
:=
\eta_i(1-\eta_{i+1}) + \eta_{i+1}(1-\eta_i)
.
\end{equation}
Note that the number of particles $\sum_{i\in\T_N}\eta_i$ is conserved by the dynamics~\eqref{eq_def_generateur_SSEP}. 
The factor $N^2$ in front of the generator \eqref{eq_def_generateur_SSEP} corresponds to a diffusive rescaling of time. 
Trajectories $\eta(\cdot) = (\eta(t))_{t\geq 0}$ belong to the Skorokhod space $\mathcal D(\R_+,\Omega_N)$ of left-continuous, right-limited $\Omega_N$-valued functions, 
and we write $\Prob^{\mu^N}_E,\E^{\mu^N}_E$ for the probability/expectation of the dynamics starting from an initial probability distribution $\mu^N$ on $\Omega_N$.\\

The invariant measures of the WASEP are well-understood (see e.g.~\cite{Kipnis1999}, Chapter 3-4). 
For $\rho\in[0,1]$, let $\nu^N_\rho$ denote the Bernoulli product measure with parameter $\rho$:
\begin{equation}
\nu^N_\rho 
:= 
\bigotimes_{i\in\T_N} \text{Ber}(\rho)
,\qquad 
\text{Ber}(\rho) : x\in\{0,1\} \mapsto \rho x + (1-\rho)(1-x) 
.
\label{eq_def_nu_N_rho}
\end{equation}
Then $\nu^N_\rho$ is invariant for the WASEP:
\begin{equation}
\forall i\in\T_N,\forall \eta\in\Omega_N,\qquad 
c_E(\eta,\eta^{i,i+1})\nu^N_\rho(\eta) 
= 
\sum_{j\in\T_N}c_E(\eta^{j,j+1},\eta )\nu^N_\rho(\eta^{j,j+1})
.
\end{equation}
In particular, for each $\rho\in[0,1]$, correlations under the product measure $\nu^N_\rho$ are trivial. \\

\noindent\textbf{The current.} 
Let $Q_T$ denote the space averaged, time-integrated current on $[0,T]$ ($T>0$): 
if $N_T(i\to j)$ is the number of jumps from site $i\in\T_N$ to site $j\in\T_N$ on $[0,T]$,
\begin{equation}
\forall \eta(\cdot)\in \mathcal  D(\R_+,\Omega_N),\qquad 
Q_T(\eta(\cdot)) 
:= 
\frac{1}{N^2}\sum_{i\in \T_N}\big[N_T(i\to i+1) - N_T(i+1\to i)\big]
.
\end{equation}
The additional factor of $N^{-1}$ reflects the diffusive scaling in time: if there is a macroscopic current of particles, i.e. if the net number of particles having gone through a bond $(i,i+1),i\in\T_N$ up to time $T$ is of order $N$, then $Q_T$ is of order $1$ in $N$. \\

\noindent\textbf{The current-biased dynamics.} 
We now define the dynamics of interest in this work. 
For a time $T>0$, $\lambda\in\R$ and an initial probability distribution $\mu^N$ of particles,  
define the \emph{current-biased dynamics} $\Prob^{curr,\mu^N}_{\lambda,E,T}$ on $\mathcal D(\R_+,\Omega_N)$ as:
\begin{equation}
\Prob^{curr,\mu^N}_{\lambda,E,T}(\cdot) 
= 
\frac{\E^{\mu^N}_E\big[{\bf 1}_{\cdot}e^{\lambda N Q_T}\big]}{\E^{\mu^N}_E\big[e^{\lambda NQ_T}\big]}.\label{eq_def_dynamique_conditionnee_courant}
\end{equation}
Under~\eqref{eq_def_dynamique_conditionnee_courant}, 
by~\cite{Bertini2007}, 
the probability of observing a macroscopic current $q=\rho(1-\rho)(\lambda+E)$ of particles on $[0,T]$ goes to $1$ when $N$ becomes large for a wide range of initial conditions enforcing a density $\rho\in(0,1)$ of particles.  
The sub-critical range~\eqref{eq_phiT} of parameters $\lambda,E$ for which there is no dynamical phase transition can then be rewritten as:
\begin{equation}
\rho(1-\rho)\lambda(\lambda+2E)
>
-\pi^2
.
\label{eq_phiT_rewritten}
\end{equation}
We stress again that, in contrast to the WASEP dynamics, 
the tilt by the current in~\eqref{eq_def_dynamique_conditionnee_courant} makes the dynamics non-local, inhomogeneous in time and induces long-range correlations (see~\cite{Bodineau2008} for heuristics and~\cite{Chetrite2015} for a general discussion of properties of tilted dynamics). 
Note also that we consider an "annealed" version of the dynamics, where both numerator and denominator in~\eqref{eq_def_dynamique_conditionnee_courant} are averaged on the initial condition. 
Consequences of this assumption and "quenched" dynamics are discussed in Section~\ref{sec_perspectives}. \\

\noindent\textbf{Initial condition, fluctuations and correlations.} 
In this article, we aim to describe the fluctuations and correlations of the dynamics \eqref{eq_def_dynamique_conditionnee_courant} when $N$, 
then $T$ are large, for sub-critical parameters $\lambda,E$ 
(i.e. satisfying~\eqref{eq_phiT_rewritten}) and in the intermediate regime where there is a macroscopic instantaneous current (the green region of Figure~\ref{fig_regions}). 
In the text, we refer to $0,T$ as \emph{time boundaries} and may refer to times $1\ll t \ll T$ as being far away from the time boundaries. 
For simplicity we will work at density $1/2$, with 
initial condition in~\eqref{eq_def_dynamique_conditionnee_courant} given by:
\begin{equation}
\mu^N(\eta) 
=
\nu^N_{1/2}(\eta)
:=
\bigotimes_{i\in\T_N}\text{Ber}(1/2)(\eta_i)
,\qquad 
\eta\in\Omega_N
.
\end{equation}
Other initial conditions are discussed in Section~\ref{sec_perspectives}. 
Since we focus on the sub-critical regime, 
the density profile under the current-biased dynamics is typically going to be constant equal to $1/2$. 
The fluctuation and correlation fields of interest are then defined as follows. 
Write: 
\begin{equation}
\bar\eta_i 
:=
\eta_i - \frac{1}{2}
,\qquad
\eta\in\Omega_N, i\in\T_N
.
\end{equation}
The fluctuation field $Y^N$ is a distribution, acting on test functions $f:\T\rightarrow\R$ according to:
\begin{equation}
Y^N(f) 
=  
Y^N(f)(\eta) 
= 
\frac{1}{N^{1/2}}\sum_{i\in\T_N}\bar\eta_i f(i/N)
\qquad
\eta\in\Omega_N
.
\label{eq_def_fluctuations}
\end{equation}
The correlation field $\Pi^N$ is defined as the "square" of $Y^N$ in the following sense: 
if $f_1,f_2:\T\rightarrow\R$, 
write $f_1\otimes f_2(x,y) = f_1(x)f_2(y)$ for $x,y\in\T$. 
Then:
\begin{equation}
\Pi^N(f_1\otimes f_2) 
= 
\frac{1}{4}Y^N(f_1)Y^N(f_2) - \frac{1}{4N}\sum_{i\in\T_N}(\bar\eta_i)^2 f_1(i/N)f_2(i/N)
.
\label{eq_lien_Pi_et_Y_dans_intro}
\end{equation}
The factor $1/4$ is a convenient normalisation. 
More generally, for $\phi:\T^2\rightarrow\R$, we define:
\begin{equation}
\Pi^N(\phi) 
= 
\frac{1}{4N}\sum_{i\neq j\in\T_N}\bar\eta_i\bar\eta_j \phi(i/N,j/N)
\qquad
\eta\in\Omega_N
.
\label{eq_def_correlations}
\end{equation}
In the text, 
"correlations" always refers to two-point correlations of the form $\bar\eta_i\bar\eta_j,i\neq j\in\T_N$ unless otherwise mentioned. 
More generally, $n$-point correlations ($n\in\N_{\geq 1}$) refer to any product of the form $\bar\eta_{i_1}...\bar\eta_{i_n}$. 
By convention, if $\phi:\T^n\rightarrow\R$, 
we write $\phi_{i_1,...,i_n}$ for $\phi(i_1/N,...,i_n/N)$, $i_1,...,i_n\in\T_N^n$. 
Letters $i,j,\ell$ are reserved for discrete indices, 
while $x,y,z$ denote continuous variables.\\

\noindent\textbf{Spaces of test functions.} 
Let $\s'(\T),\s'(\T^2)$ denote the space of distributions on $\T,\T^2$ respectively. 
The fluctuation field $Y^N$ is viewed as a random element in $\s'(\T)$. 
To study of the correlation field $\Pi^N$, 
we will have to consider test functions which are continuous, but have discontinuous normal derivatives across the diagonal $D = \{(x,x):x\in\T\}$ of $\T^2$. 
This is related to the discontinuity of out of equilibrium correlations across the diagonal $D$, 
see~\cite{Spohn1983,Correlations2022}, 
and in particular would not be needed to study correlations correlations at equilibrium~\cite{Goncalves2019}. 
Write for short $\{x\leq y\}:=\{(x,y)\in\T^2:x\leq y\}$, 
$\{x\geq y\}:=\{(x,y)\in\T^2:x\geq y\}$. 
The set of test functions $C^\infty_D(\T^2)$ is given by all continuous functions on $\T^2$ such that their restrictions on $\{x\leq y\}$ and $\{x\geq y\}$ are $C^\infty$ (the derivatives of the extensions on the diagonal $D$ need not coincide)
\begin{align}
C^\infty_D(\T^2) 
= 
\big\{\phi \in C^0(\T^2): \phi_{| \{x\leq y\}}\in C^\infty(\{x\leq y\}), 
\phi_{| \{x\geq y\}}\in C^\infty(\{x\geq y\}) 
\big\}
.
\label{eq_def_set_c_infty_D}
\end{align}
The correlation field $\Pi^N$ is seen as a random distribution in the set $\s'_D(\T^2)$ of bounded linear forms on $C^\infty_D(\T^2)$. 

A trajectory $\eta(\cdot)\in\mathcal D(\R_+,\Omega_N)$ induces a process $Y^N_\cdot=(Y^N_t)_{t\geq 0}$, which we refer to as the \emph{fluctuation process}. 
It is a random element of the Skorokhod space $\mathcal D(\R_+,\s'(\T))$, 
equipped with its usual topology, see e.g. \cite{Billingsley1999}. 
The \emph{correlation process} similarly denotes the random element $\Pi^N_\cdot=(\Pi^N_t)_{t\geq 0}$ of $\mathcal D(\R_+,\s'_D(\T^2))$. 
\subsection{Goal and approach: the approximate driven process}
\label{sec_driven_process}
In this section, we define what we mean by \emph{approximate driven process} as mentioned in Section~\ref{sec_approx_driven_process_intro}. 
We seek to characterise fluctuations and correlations for the current-biased dynamics~\eqref{eq_def_dynamique_conditionnee_courant}, 
in the large $N$, then large $T$ limits; 
for intermediate times far away from the time boundaries $0,T$ (corresponding to the green region in Figure~\ref{fig_regions}). 
In this regime, 
for sub-critical values~\eqref{eq_phiT_rewritten} of $\lambda,E$ for which the best way to realise the current constraint is by having a constant density profile, 
one expects the current-biased dynamics to be well approximated by a stationary Markov process as explained in the introduction. 
Our aim, then, is to find a Markov process $\tilde \Prob^{N}$   
which describes fluctuations and correlations under~\eqref{eq_def_dynamique_conditionnee_courant} in this intermediate time range. 
For this Markov process $\tilde \Prob^{N}$ to describe the fluctuations under the current biased-dynamics it should satisfy, 
for a good choice of initial condition $\mu^N$, 
any bounded continuous $F:\mathcal D(\R_+,\s'(\T))\rightarrow\R$ and each $t,T,\tau>0$ with $t+\tau\leq T$:
\begin{equation}
\int F\big((Y^N)_{t\leq s\leq t+\tau}\big)\, \mathrm{d}\Prob^{curr,\mu^N}_{\lambda,E,T} 
= 
\int F\big((Y^N_s)_{0\leq s\leq \tau}\big)\, \mathrm{d}\tilde \Prob^{N} +\epsilon^N_{\tau,t,T}(F)
,
\label{eq_def_driven_process}
\end{equation}
with $\epsilon^N_{\tau,t,T}(F)$ a term that must be small, 
in the large $N$ limit, 
for times away from the time boundaries:
\begin{equation}
\lim_{t\rightarrow\infty}\limsup_{T\rightarrow\infty}\limsup_{N\rightarrow\infty}\big|\epsilon^N_{\tau,t,T}(F)\big| 
= 0
.
\label{eq_error_in_def_driven_process}
\end{equation}
We call any $\tilde \Prob^N$ satisfying~\eqref{eq_def_driven_process}--\eqref{eq_error_in_def_driven_process} an \emph{approximate driven process} in analogy with the driven process of~\cite{Chetrite2015}. \\

Let us now give more details on the construction of an approximate driven process. 
To find such a process, 
we build the jump rates of $\tilde \Prob^N$ explicitly,
according to the following considerations. 
\begin{itemize}
	\item For any $T>0$ and $\lambda\neq 0$, 
	the typical current $Q_T$ under the current-biased dynamics~\eqref{eq_def_dynamique_conditionnee_courant} is different from the current under the WASEP $\Prob^N_E$.
	One expects the current-biased dynamics to be better approximated by a dynamics that has the correct macroscopic current. 
	The simplest way to modify $\Prob^N_E$ into a dynamics that has the correct current is to change the asymmetry from $E$ to $\lambda+E$, with jump rates:
	\begin{equation}
c_{\lambda+E}(\eta,\eta^{i,i+1}) 
:= 
c_{E}(\eta,\eta^{i,i+1})e^{-N^{-1}(\eta_{i+1}-\eta_i)\lambda},
\qquad i\in\T_N,\eta\in\Omega_N
.
\label{eq_def_jump_rates_lambda}
\end{equation}
	The resulting WASEP dynamics $\Prob^N_{\lambda+E}$ is, 
	by construction, a good approximation of the current-biased dynamics as far as the density and current of particles are concerned. 
	However, it does not have the correct correlation structure (recall that product Bernoulli measures with constant parameter are invariant for the WASEP). 
	It thus does not have the same density fluctuations and therefore cannot serve as an approximate driven process as defined in~\eqref{eq_def_driven_process}.
	\item We then tune observables on the next finer scale than the current, that is we tune the jump rates to obtain a dynamics that has the same two-point correlation structure as the current-biased dynamics. 
	We expect these correlations to be non-local, 
	so we need non-local jump rates. 
	An effective way to tune correlations without changing the density/current in the large $N$ limit~\cite{Correlations2022} is to consider, 
	for a symmetric function $h:\T^2\rightarrow\R$ that we call a bias below, the jump rates:
	\begin{equation}
c_{h,\lambda+E}(\eta,\eta^{i,i+1}) 
:= 
c_{\lambda+E}(\eta,\eta^{i,i+1})e^{\Pi^N(h)(\eta^{i,i+1}) - \Pi^N(h)(\eta)},
\qquad i\in\T_N,\eta\in\Omega_N
.
\label{eq_def_jump_rates_h_and_lambda}
\end{equation}
	Let $\Prob^N_{h,\lambda+E}$ be the corresponding dynamics. 
	We then optimise the bias $h$ so that $\Prob^N_{h,\lambda+E}$ and the current-biased dynamics are, loosely speaking, as close to each other as possible. 
\end{itemize}
As stated in the next section, tuning the current/density and two-point correlations is enough and $\Prob^N_{h,\lambda+E}$ is an approximate driven process in the sense of~\eqref{eq_def_driven_process} for a suitable bias $h$.
\subsection{Results}\label{sec_results}
A first result of this work is the computation of the correlation structure under the current-biased dynamics in the green region of Figure~\ref{fig_regions}, in Proposition~\ref{prop_correl_sec_results}.  
The correlation structure is related to the optimal bias $h$ that one must choose in order for $\Prob^N_{h,\lambda+E}$ to be an approximate driven process. 
This bias is characterised next. 
We start with some notations. 
Any $\phi\in C^\infty_D(\T^2)$ is identified with the kernel operator $\phi f = \int_{\T}\phi(x,\cdot)f(\cdot)\, dx$ for $f\in\mathbb L^2(\T)$. 
We say that $\phi$ is a positive kernel if:
\begin{equation}
\forall f\in \mathbb L^2(\T),
\qquad 
\int_{\T^2} f(x) \phi(x,y)f(y)\, dx\, dy 
\geq 
0
.
\end{equation}
A negative kernel corresponds to the opposite sign. 
Define also:
\begin{equation}\label{eq_def_sigma_intro}
\sigma 
:= 
\rho(1-\rho)|_{\rho=1/2}
= 
\frac{1}{4}
.
\end{equation}
For $\lambda,E\in\R$, 
consider the following  
differential equation with unknown 
$h:\T\rightarrow\R$:
\begin{equation}
\begin{cases}
h''(x) - \frac{\sigma}{2}\int_{\T} h'(x-z)h'(z)\, dz = 0\quad\text{for }x\in(0,1),\\
h'(0^+)-h'(1^-) = 2\lambda(\lambda+2E)
.
\end{cases}
\label{eq_ODE_sur_h}
\end{equation}
Solutions to~\eqref{eq_ODE_sur_h} are continuous periodic functions with derivative having a jump. 
\begin{prop}\label{prop_solving_PDE_h}
Let $\lambda,E\in\R$ be sub-critical parameters at density $1/2$ as in~\eqref{eq_phiT_rewritten}:
\begin{equation}
\sigma\lambda(\lambda+2E) 
>
-\pi^2
.
\label{eq_phiT_in_prop}
\end{equation}
There is then a unique family $(h_{\lambda,E})_{\lambda,E\text{ sub-critical}}$ of solutions in $C^0(\T)\cap C^\infty([0,1])$ of~\eqref{eq_ODE_sur_h} such that $\int_{\T}h_{\lambda,E}(x)\, dx =0$ and:
\begin{equation}
(\lambda,E)\mapsto \|h_{\lambda,E}\|_{2}\text{ is a continuous function vanishing on the line }\lambda=0
.
\label{eq_continuity_assumption_h}
\end{equation}
The function $h_{\lambda,E}$ is explicitly given by:
\begin{equation}
\forall x\in\T,\qquad 
h_{\lambda,E}(x) 
= 
\frac{\sqrt{2}}{\sigma} \sum_{\ell\geq 1}\Big(1-\Big[1+\frac{\sigma\lambda(\lambda+2E)}{\pi^2\ell^2}\Big]^{1/2}\Big)\sqrt{2}\cos(2\pi\ell x) 
.
\end{equation}
The kernel $(x,y)\in\T^2\mapsto h_{\lambda,E}(x-y)$, 
still denoted $h_{\lambda,E}$, 
is a symmetric kernel in $C^\infty_D(\T^2)$ with eigenvalues given by the Fourier coefficients of $h_{\lambda,E}$. 
In particular, $h_{\lambda,E}$ is a negative kernel if $\lambda(\lambda+2E)\geq 0$, 
a positive kernel if $\lambda(\lambda+2E)\leq 0$ 
and the critical line can be expressed in terms of the behaviour of the eigenvalues of $h_{\lambda,E}$:
\begin{equation}
\sigma\lambda(\lambda+2E)>-\pi^2
\quad \Leftrightarrow\quad
\sigma h_{\lambda,E}\text{ has leading eigenvalue strictly smaller than }1
.
\label{eq_critical_line_as_eigenvalue}
\end{equation}
%

%
%
\end{prop}
\begin{rmk}
The family $(h_{\lambda,E})$ is not unique without the continuity requirement~\eqref{eq_continuity_assumption_h}, see Appendix~\ref{app_BFP}.  
This requirement is however physically meaningful.  
Indeed, the current-biased dynamics with $\lambda=0$ case corresponds to the WASEP $\Prob^N_E$ for which there is no correlation in the steady state.  
Proposition~\ref{prop_correl_sec_results} shows that $h$ determines correlations under the current-biased dynamics which in particular depend continuously on $\lambda$.  \demo
\end{rmk}
In the following, the subscripts $\lambda,E$ will be dropped and we simply write $h$ for $h_{\lambda,E}$. 
\begin{prop}[Correlations under the current biased dynamics at times $1\ll t\ll T$]\label{prop_correl_sec_results}
Let $\lambda,E\in\R$. Assume that $\lambda(\lambda+2E)\geq 0$ and is small enough (in particular $\lambda,E$ are sub-critical~\eqref{eq_phiT_in_prop}). 
Then, for any $f_1,f_2\in C^\infty(\T)$:
\begin{equation}
\lim_{T\to\infty}\lim_{N\to\infty}\E^{curr,\,\nu^N_{1/2}}_{\lambda,E,T}[Y^N_{T/2}(f_1)Y^N_{T/2}(f_2)]
=
\frac{1}{4}\int_{\T^2} f_1(x) \big((\sigma\text{id} +k_{curr})f_2\big)(x)\, dx
,
\end{equation}
with $k_{curr}\in C^\infty([0,1])\cap C^0(\T)$ identified with the kernel $(x,y)\in\T^2\mapsto k_{curr}(x-y)$ and given explicitly by:
\begin{equation}
k_{curr}(x)
=
\sqrt{2}\sigma\sum_{\ell\geq 1}\bigg[\frac{1}{\sqrt{1+\frac{\lambda(\lambda+2E)\sigma}{\pi^2\ell^2}}} -1\bigg]\sqrt{2}\cos(2\pi\ell x)
.
\label{eq_kcurr_prop_sec2}
\end{equation}
In addition, $k_{curr}$ is related to $h$ through:
\begin{equation}
(\sigma^{-1}\text{id}-h)^{-1}
=
\sigma\text{id} + k_{curr}
.
\end{equation}
\end{prop}
The expression~\eqref{eq_kcurr_prop_sec2} of $k_{curr}$ agrees with the one derived in~\cite{Bodineau2008}. 
Yet another equivalent formulation of the sub-critical region~\eqref{eq_phiT_in_prop} is then as the region in which $k_{curr}$ has bounded largest eigenvalue.\\

Proposition~\ref{prop_correl_sec_results} is obtained as a corollary of the next theorem, 
the main result of this work, 
which characterises not just the spatial correlation structure but the full dynamics of fluctuations under the current-biased dynamics:  
we show that $\Prob^{N}_{h,\lambda+E}$ is an approximate driven process in the sense of \eqref{eq_def_driven_process}. 
To state the result precisely, 
introduce a discrete approximation $\nu^N_h$ of a Gaussian measure:
\begin{equation}
\nu^N_h
:=
\frac{e^{2\Pi^N(h)}}{\mathcal Z^N_h}\ \nu^N_{1/2}
\label{eq_def_nu_rho_g}
.
\end{equation}
Above, $\mathcal Z^N_h$ is a normalisation factor. Properties of this measure are analysed in Appendix~\ref{app_BFP}, 
in particular it converges to a Gaussian field with covariance $\sigma{\text{id}} + k_{curr}$.  
It is technical but convenient to work with the dynamics $\Prob^N_{h,\lambda+E}$ started from $\nu^N_h$ as it is a good approximation of its invariant measure, see Theorems~\ref{theo_fluct_correl_driven_process}--\ref{theo_size_entropy} below. 
\begin{theo}\label{theo_main_result}
Let $\lambda,E\in\R$. 
Assume that $\lambda(\lambda+2E)\geq 0$ and is small enough (in particular~\eqref{eq_phiT_in_prop} holds). 
Then $\Prob^{\nu^N_h}_{h,\lambda+E}$ is an approximate driven process for the current-biased dynamics $\Prob^{curr,\nu_{1/2}^N}_{\lambda,E,T}$. 
That is, for any bounded continuous $F:\mathcal D(\R_+,\s'(\T))\rightarrow\R$ and $t,T, \tau>0$ with $t+\tau\leq T$:
\begin{equation}
\E^{curr,\nu^N_{1/2}}_{\lambda,E,T}\Big[F\big((Y^N)_{t\leq s\leq t+\tau}\big)\Big]
= 
\E^{\nu^N_h}_{h,\lambda+E}\Big[F\big((Y^N_s)_{0\leq s\leq \tau}\big)\Big]+\epsilon^{N,T}_{[t,t+\tau]}(F)
,
\label{eq_driven_process_in_theorem}
\end{equation}
with:
\begin{equation}
\limsup_{t\rightarrow\infty}\limsup_{T\rightarrow\infty}\limsup_{N\rightarrow\infty}\big|\epsilon^{N,T}_{[t,t+\tau]}(F)\big|
=
0
.
\end{equation}
In~\eqref{eq_driven_process_in_theorem}, 
$t,\tau$ can also be taken to diverge with $T$ provided $T-\tau$ diverges as well. 
\end{theo}
We expect the claim of Theorem~\ref{theo_main_result} to be valid throughout the sub-critical regime~\eqref{eq_phiT_in_prop}, 
with some care required about the initial condition. 
This and extensions of Theorem~\ref{theo_main_result} are discussed in Section~\ref{sec_perspectives}. 
Let us however already mention that Theorem~\ref{theo_main_result} is the only theorem for which we need to restrict to small $\lambda(\lambda+2E)\geq 0$. 
All other Theorems are valid in the entire sub-critical regime~\eqref{eq_phiT_in_prop} as we shall see.

Equation~\eqref{eq_driven_process_in_theorem} implies that the fluctuation process under the current biased dynamics converges to the limiting fluctuations under $\Prob^N_{h,\lambda+E}$. 
These fluctuations can be characterised precisely, as stated next in Theorem~\ref{theo_fluct_correl_driven_process} 
where a more general study of dynamics of the form $\Prob^N_{\tilde h,\tilde\lambda}$ is carried out for $\tilde\lambda\in\R$ and $\tilde h\in C^\infty_D(\T^2)$ 
(i.e. not just for $\tilde h=h$ or $\tilde\lambda=\lambda+E$). 
\begin{theo}\label{theo_fluct_correl_driven_process}
Let $\tilde\lambda\in\R$ and 
let $\tilde h \in C^\infty_D(\T^2)$ be symmetric and such that $\sigma\tilde h$ has leading eigenvalue strictly below $1$. Define $\nu^N_{\tilde h}$ as in~\eqref{eq_def_nu_rho_g}.

The law of $(Y^N_\cdot,\Pi^N_\cdot)$ under $\Prob^{\nu^N_{\tilde h}}_{\tilde h,\tilde \lambda}$ then converges weakly in $\mathcal D(\R_+,\s'(\T)\times \s'_D(\T^2))$ to the law of a couple $(Y_\cdot,\Pi_\cdot)$ satisfying:
\begin{enumerate}
	\item The law $\nu^\infty_{\tilde h}$ of $Y_0$ is a Gaussian field on $\s'(\T)$, 
	with covariance $C_{\tilde h} := (\sigma^{-1}\text{id}-\tilde h)^{-1}$. Moreover, $Y_\cdot$ and $\Pi_\cdot$ are stationary processes.
	\item For any $\phi \in C^\infty_D(\T^2)$, 
	if $\phi^s$ denotes the symmetric part of $\phi$, 
	then $\Pi_\cdot(\phi) = \Pi_\cdot(\phi^s)$.
	\item $Y_\cdot$ is uniquely characterised by the following martingale problem. 
	For each test function $f\in C^{1}(\R_+,C^\infty(\T))$, 
	the processes $M_\cdot(f),\mathfrak M_\cdot(f)$ are continuous martingales with respect to the canonical filtration $(\mathcal F_t)_{t\geq0}$ generated by $Y_\cdot$: 
	\begin{align}
\forall t\geq 0,\qquad 
M_t(f) 
&= 
Y_t(f_t)-Y_0(f_t) -\int_0^t Y_s\Big(\partial_s f_s+ \mathcal L_{\tilde h}f_s\Big)ds
\nonumber\\
\mathfrak M_t(f) 
&= 
\big(M_t(f)\big)^2- 2\sigma\int_0^t\|\partial_x f_s\|^2_{2}\,  dx
,
\label{eq_martingale_problem_theo}
\end{align}
with $\mathcal L_{\tilde h}$ the operator acting on $f\in C^\infty(\T)$ according to:
\begin{equation}
\mathcal L_{\tilde h}f(x) 
=
f''(x) +\sigma \int_{\T} \partial_2 \tilde h(x,z)f'(z)\, dz
=
(\text{id}-\sigma\tilde h)\Delta f(x) ,\qquad 
x\in\T
.
\label{eq_def_L_star_theo_fluct}
\end{equation}
\item The processes $Y_\cdot,\Pi_\cdot$ are related as follows:
\begin{equation}
\forall (f_1,f_2)\in C^\infty(\T)^2,
\qquad \Pi_\cdot(f_1\otimes f_2) 
=
 \frac{1}{4}Y_\cdot(f_1)Y_\cdot(f_2) - \frac{\sigma}{4}\int_{\T}f_1(x)f_2(x)\, dx
.
\label{eq_lien_Pi_et_Y_dans_theo}
\end{equation}
\item (Non-smooth test functions and bound on correlations) 
The process $\Pi_\cdot$ admits a unique extension to a process in $\mathcal D([0,+\infty),\s'_D(\T^2))$ (i.e. also defined on test functions that may have discontinuities on the diagonal).   
If $p\in\N$ and $t_1,...,t_p\geq 0$, 
$\phi_1,...,\phi_p\in\C^\infty_D(\T^2)$, 
then the vector $\big(\Pi^N_{t_1}(\phi_1),...,\Pi^N_{t_p}(\phi_p)\big)$ converges weakly to $\big(\Pi_{t_1}(\phi_1),...,\Pi_{t_p}(\phi_p)\big)$. 
In addition, 
there is $C(\tilde h,\tilde\lambda)>0$ such that, for each $\epsilon>0$:
\begin{equation}
{\bf P}^{\nu^\infty_{\tilde h}}_{\tilde h,\tilde\lambda}\Big(|\Pi_t(\phi_1)\big|>\epsilon\Big)
\leq 
\frac{C(\tilde h,\tilde \lambda)\|\phi_1\|_\infty}{\epsilon}
,\qquad 
t\geq 0
.
\end{equation}
\end{enumerate}
\end{theo}
\begin{rmk}
\begin{itemize}
\item In the special case where $\tilde h$ commutes with the Laplacian, the operator $\mathcal L_{\tilde h}$ is self-adjoint and the associated fluctuations are reversible. 
This is the case for translation-invariant $\tilde h$, 
in particular the fluctuations under the current-biased dynamics are reversible as in that case $\tilde h=h$.
\item Note that the parameter $\tilde\lambda$ does not appear in the martingale problem of Theorem~\eqref{theo_fluct_correl_driven_process}. 
This is due to the fact that we consider fluctuations around density $1/2$. 
At density $\rho\neq 1/2$, there would be a term proportional to $\tilde\lambda \sigma'(\rho)\int_0^t Y_s(\partial_x f_s)\, ds$ in $M_t(f)$ in~\eqref{eq_martingale_problem_theo}, 
with $\sigma(\rho) = \rho(1-\rho)$ (and in particular even for translation-invariant $\tilde h$ the resulting fluctuations would not be reversible). 

\end{itemize}
\end{rmk}
Item 3 of Theorem \ref{theo_fluct_correl_driven_process} is essentially proven in~\cite{Jara2018}. 
The key ingredient there is a refinement of the relative entropy method of Yau~\cite{Yau1991}. 
This refinement has been used in \cite{Jara2018}--\cite{Jara2020} to characterise out of equilibrium fluctuations in dimension $d<4$. 
Putting Theorem~\ref{theo_main_result} and Theorem~\ref{theo_fluct_correl_driven_process} together, 
one obtains that the fluctuations under the current-biased dynamics in the intermediate regime $1\ll t\ll T$ (the green region of Figure~\ref{fig_regions}) are given by the only solution of the infinite-dimensional Ornstein-Uhlenbeck process:
\begin{equation}
\mathrm{d}Y_t 
=
\mathcal L^*_{h}Y_t\, \mathrm{d}t + \sqrt{2\sigma}\, \nabla {\mathrm d}W_t
,
\end{equation}
with $\mathrm{d}W_\cdot$ a space-time white noise and $\mathcal L^*_h=\mathcal L_h$ the non-local, self-adjoint operator defined in~\eqref{eq_def_L_star_theo_fluct}, 
acting on $Y_t$ by $\mathcal L_hY_t(f) = Y_t(\mathcal L_hf)$ ($f\in C^\infty(\T)$). 
The non-locality of the drift term $\mathcal L_h$ is responsible for the non-local correlation structure of the current-biased dynamics. \\

Although Theorem~\ref{theo_fluct_correl_driven_process} can be proven using the relative entropy bounds of~\cite{Jara2018}, 
proving that $\Prob^{\nu^N_h}_{h,\lambda+E}$ is an approximate driven process for the current biased dynamics as in~\eqref{eq_def_driven_process} requires yet finer bounds. 
These bounds, stated next, are the main technical ingredient. 
Like Theorem~\ref{theo_fluct_correl_driven_process}, 
they hold in the whole sub-critical regime~\eqref{eq_phiT_in_prop}. 
\begin{theo}\label{theo_size_entropy}
Let $\lambda,E\in\R$ be sub-critical, i.e. satisfy~\eqref{eq_phiT_in_prop}.  
Let $h$ be given by Proposition \ref{prop_solving_PDE_h}. 
For $t\geq 0$, 
let ${\bf H}^N(f_t)$ denote the relative entropy of the law of the dynamics $f_t \nu^N_h$ with respect to $\nu^N_h$:
\begin{equation}
{\bf H}^N(f_t)
:=
\nu^N_h(f_t\log f_t)
=
\sum_{\eta\in\Omega_N}f_t(\eta)\log f_t(\eta)\nu^N_h(\eta)
.
\end{equation}
There are then $\gamma,C>0$ depending only on $\lambda,E$ such that:
\begin{align}
\forall t> 0,\qquad 
&\partial_t{\bf H}^N(f_t)
\leq 
\frac{{\bf H}^N(f_t)}{\gamma} + \frac{C}{N^{1/2}}
\nonumber\\
&\hspace{2cm}\Rightarrow\quad 
{\bf H}^N(f_t)\leq \frac{\gamma C e^{t/\gamma}}{N^{1/2}},
\quad t\geq 0
.
\end{align}
\end{theo}
Theorem~\ref{theo_size_entropy} is proven in Section~\ref{sec_entropy_estimate}. 
Such a relative entropy estimate provides a precise control on the dynamics that can be used in other contexts. 
In~\cite{Correlations2022}, 
they allowed for the study of the probability of observing long time anomalous correlations in the symmetric simple exclusion process in contact with reservoirs at different densities.
\subsection{Heuristics and structure of the proof}\label{sec_structure_proof}
Let us now sketch the proof of Theorem~\ref{theo_main_result}, 
postponing technical estimates to later sections. \\

\textbf{Step 1: Construction of a candidate for the approximate driven process.} To obtain a candidate approximate driven process, 
we carry out the procedure outlined in Section~\ref{sec_driven_process}, 
successively modifying the jump rates to first obtain a dynamics with the right macroscopic current, 
then with the right two-point correlations. 
This is the content of Proposition~\ref{prop_choice_of_h_sec4}, 
which states that, for any symmetric, translation invariant bias $\tilde h\in C^\infty_D(\T^2)$, 
the current-biased dynamics can be rewritten in terms of a tilted dynamics:
\begin{align}
\Prob^{curr,\nu^N_{1/2}}_{\lambda,E,T}(\cdot)
=
\frac{\E^{\nu^N_{\tilde h}}_{\tilde h,\lambda+E}\Big[{\bf 1}_{\cdot} \exp\Big(-\Pi^N_T(\tilde h)-\Pi^N_0(\tilde h) + \int_0^T \Pi^N_t( F_{\lambda,E}(\tilde h))\, dt + \int_0^T \epsilon^N_{\tilde h,\lambda,E}(\eta_t)\, dt\Big)\Big]}{\E^{\nu^N_{\tilde h}}_{\tilde h,\lambda+E}\Big[ \exp\Big(-\Pi^N_T(\tilde h)-\Pi^N_0(\tilde h) + \int_0^T \Pi^N_t( F_{\lambda,E}(\tilde h))\, dt + \int_0^T \epsilon^N_{\tilde h,\lambda,E}(\eta_t)\, dt\Big)\Big]}
.
\label{eq_heuristique_radon_nikodym}
\end{align}
Above, $F_{\lambda,E}$ is a differential operator while the random variable $\epsilon^N_{\tilde h,\lambda,E}$ involves three-point correlations and higher that are expected to be negligible compared to two-point correlations, which one expects to be of order $1$ in $N$.  

In order for $\Prob^N_{\tilde h,\lambda+E}$ to be an approximate driven process in the sense of~\eqref{eq_def_driven_process}, 
we need the exponential in~\eqref{eq_heuristique_radon_nikodym} to be bounded with $N,T$ when $N$, then $T$ large. 
This suggests to take $\tilde h$ with $F_{\lambda,E}(\tilde h)=0$, 
which is equivalent to $\tilde h$ being the bias $h$ of Proposition~\ref{prop_solving_PDE_h}.

The error terms $\int_0^T\epsilon^N_{h,\lambda,E}(\eta_t)\, dt,\Pi^N_T(h),\Pi^N_0(h)$, however, appear inside exponentials. 
We check in Section~\ref{sec_fluc_correl_prob_h_lambda} that they have are well-controlled in expectation under the dynamics $\Prob^{\nu^N_{h}}_{h,\lambda+E}$, 
using the relative entropy estimate of Theorem~\ref{theo_size_entropy} that is also proven there. 
Still, it could be that their exponential moment blows up with $N,T$ on a rare event. 
This would mean that $\Prob^{\nu^N_{h}}_{h,\lambda+E}$ is in fact not an approximate driven process. \\

\textbf{Step 2: estimate of error terms. }
The key Proposition~\ref{prop_domination_by_P_h} provides a control of the exponential terms in~\eqref{eq_heuristique_radon_nikodym}, 
proving that they indeed remain bounded when $N$, then $T$ are large.  
This is achieved through a domination bound of the following form: 
for any sequence $(\mathcal O_N)$ of measurable events that involve the dynamics up to time $T$ at most,
\begin{equation}
\limsup_{N\rightarrow\infty}\Prob^{curr,\nu^N_{1/2}}_{\lambda,E,T}(\mathcal O_N)
\leq 
C(\lambda,E)\limsup_{N\to\infty}\Prob^{\nu^N_h}_{h,\lambda+E}(\mathcal O_N)^{1/2}
.
\label{eq_domination_bound_sec2}
\end{equation}
This bound is the only claim for which we need the (technical) restriction that $\lambda(\lambda+2E)\geq 0$ is small enough. 

Equipped with the domination bound~\eqref{eq_domination_bound_sec2}, 
it is not hard to prove (see Proposition~\ref{prop_prob_curr_as_prob_driven_with_bounded_terms}) that the current-biased dynamics has the same fluctuations at times $1\ll t\ll T$ when $N$, then $T$ are large as the dynamics:
\begin{equation}
\frac{\E^{\nu^N_h}_{h,\lambda+E}\big[{\bf 1}_{\cdot }H_1(Y^N_0)H_2(Y^N_T)\big]}{\E^{\nu^N_h}_{h,\lambda+E}\big[H_1(Y^N_0) H_2(Y^N_T)\big]}
.
\label{eq_approx_dyn_sec_heuristics}
\end{equation}
Above, $H_1,H_2$ are suitable bounded functions of the fluctuation field. \\

\textbf{Step 3: fluctuations under $\E^{\nu^N_{h}}_{h,\lambda+E}$ and time decorrelation.} 
Fluctuations under~\eqref{eq_approx_dyn_sec_heuristics} are then analysed in two stages. 
The limiting law ${\bf P}^{\nu^\infty_h}_{h,\lambda+E}$ of the fluctuation process under $\Prob^{\nu^N_{h}}_{h,\lambda+E}$ is first characterised in Section~\ref{sec_fluc_correl_prob_h_lambda}, in particular proving Theorem~\ref{theo_fluct_correl_driven_process}. 
We prove that,
for any bounded $F:\mathcal D(\R_+,\s(\T))\rightarrow\R$ and suitable bounded functions $H_1,H_2$,  
\begin{equation}
\lim_{N\rightarrow\infty}
\frac{\E^{\nu^N_h}_{h,\lambda+E}\big[F\big((Y^N_s)_{s\leq T}\big)H_1(Y^N_0) H_2(Y^N_T)\big]}{\E^{\nu^N_h}_{h,\lambda+E}\big[H_1(Y^N_0) H_2(Y^N_T)\big]}
=
\frac{{\bf E}^{\nu^\infty_h}_{h,\lambda+E}\big[F\big((Y_s)_{s\leq T}\big)H_1(Y_0) H_2(Y_T)\big]}{{\bf E}^{\nu^\infty_h}_{h,\lambda+E}\big[H_1(Y_0) H_2(Y_T)\big]}
.
\label{eq_structure_proof_limiting_fluctuations}
\end{equation}
We conclude the proof of Theorem~\ref{theo_main_result} in Section~\ref{sec_time_decorrelation}, 
showing that the limiting fluctuation process decorrelates in time in the sense that, for each $t\geq 0$ and each $A>0$ with $t+A\leq T$:
\begin{equation}
\frac{{\bf E}^{\nu^\infty_h}_{h,\lambda+E}\big[F\big((Y_s)_{t\leq s\leq t+A}\big)H_1(Y_0) H_2(Y_T)\big]}{{\bf E}_{h,\lambda+E}\big[H_1(Y_0) H_2(Y_T)\big]}
= 
{\bf E}^{\nu^\infty_h}_{h,\lambda+E}\big[F\big((Y_s)_{s\leq T}\big)\big] 
+C_{t,T}(F,H_1,H_2)
,
\label{eq_decorrelation_bound_sec_heuristics}
\end{equation}
with:
\begin{equation}
\lim_{t\rightarrow\infty}\limsup_{T\rightarrow\infty}|C_{t,T}(F,H_1,H_2)|
=
0
.
\end{equation}
This decorrelation in time boils down to Gaussian computations on a family of finite-dimensional diffusions approximating the limiting fluctuation process, 
constructed as in~\cite{Holley1978}.
\subsection{Discussion of the results and perspectives}\label{sec_perspectives}
In this article, the large $N$ density fluctuations of the current-biased dynamics at intermediate times $1\ll t \ll T$ are characterised. 
In this time regime,  
the current biased dynamics are shown to be well-approximated by the explicit, non-local and homogeneous Markov dynamics $\Prob^{\nu^N_h}_{h,\lambda+E}$. 
Moreover, correlations at times $1\ll t\ll T$ under the current-biased dynamics are shown to agree with the expression predicted in~\cite{Bodineau2008}.
Let us discuss a few limitations and possible extensions.

\paragraph{Dynamics for times around $0,T$.} 
It is also possible to investigate the behaviour of the current-biased dynamics~\eqref{eq_def_dynamique_conditionnee_courant} at times $t,T-t$ with $t$ of order $1$. 
This would require no change to the microscopic study, 
only work on the limiting fluctuation process. 
Indeed, for such times the limiting fluctuations can still be expressed in terms of ${\bf P}^{\nu_h}_{h,\lambda+E}$ but the time decorrelation~\eqref{eq_decorrelation_bound_sec_heuristics} does not happen, reflecting the fact that initial/final conditions still influence the dynamics.  
The corresponding fluctuation process therefore becomes inhomogeneous in time.

\paragraph{Annealed and quenched dynamics.}
The current-biased dynamics~\eqref{eq_current_biased_dynamics_intro} is defined in an "annealed" way, in the sense that it is a ratio of quantities averaged on the initial condition. 
One could also ask about results on a "quenched" version of the dynamics. 
Several variants of quenched dynamics can be envisioned.  
A fully quenched version, so to speak, would be:
\begin{equation}
\E^{q,\mu^N}_{\lambda,E,T}(\cdot)
=
\sum_{\eta\in\Omega_N}\mu^N(\eta)\frac{\E^\eta_E[{\bf 1}_{\cdot}e^{\lambda NQ_T}]}{\E^{\eta}_E[e^{\lambda NQ_T}]}
\qquad\text{for an initial condition }\mu^N
.
\label{eq_def_quenched_dynamics}
\end{equation}
If e.g. $\mu^N=\nu^N_{1/2}$, 
the only missing ingredient to prove Theorem~\ref{theo_main_result} for the dynamics~\eqref{eq_def_quenched_dynamics} is an equivalent of the domination bound~\eqref{eq_domination_bound_sec2}. 
It is unclear whether such a bound should hold. 
If it does, then the above quenched dynamics would have the same fluctuations as the "annealed" dynamics~\eqref{eq_def_dynamique_conditionnee_courant}.

Another possibility is to average on the initial condition while keeping the number of particles fixed, 
as in the following dynamics started from the uniform distribution $U_m$ on configurations with $m\in\{0,...,N\}$ particles:
\begin{equation}
\E^{curr,U_{m}}_{\lambda,E,T}(\cdot)
=
\frac{\E^{U_m}_E[{\bf 1}_{\cdot}e^{\lambda NQ_T}]}{\E^{U_m}_E[e^{\lambda NQ_T}]}
.
\label{eq_current_biased_dynamics_with_unif_measure}
\end{equation}
For this dynamics we expect Theorem~\ref{theo_main_result} to hold without much change to the proof. 
This would be very interesting to prove as the initial condition $U_m$ has an advantage over $\nu^N_{1/2}$ as discussed in the next paragraph.

\paragraph{Influence of the initial condition.} 
We now come back to the annealed current-biased dynamics~\eqref{eq_def_dynamique_conditionnee_courant} considered in the paper and discuss initial conditions more in depth.

Our main result, Theorem~\ref{theo_main_result}, 
is a statement about the current-biased dynamics in the intermediate time interval $1\ll t\ll T$. 
Part of its proof consists in showing a time decorrelation result according to which the effect of the initial condition is not relevant. 
One therefore expects the statement to be independent of the specific choice $\nu^N_{1/2}$ of initial condition (and for instance Theorem~\ref{theo_main_result} also holds starting from a class of discrete Gaussian measures of the form~\eqref{eq_def_nu_rho_g}, with the same proof).

On one aspect, however, the choice of $\nu^N_{1/2}$ does matter. 
That is because this measure allows all possible number of particles, including $0$ or $N$ particles where there is no current. 
Since the number of particles is conserved by the dynamics, this has important consequences. 
Indeed, start the WASEP dynamics from the uniform measure $U_m$ on configurations with $m=\lfloor \rho N\rfloor$ particles ($\rho\in(0,1)$). 
A direct computation using~\cite{BertiniConcurrentDonskerVaradhan2023} shows that the denominator in~\eqref{eq_dyn_courant_en_terme_dyn_lambda} is informally given to leading order in $N,T$ by:
\begin{equation}
\E^{U_m}_E\Big[e^{\lambda NQ_T}\Big] 
\approx 
e^{NT\rho(1-\rho)\lambda(\lambda+2E)}
.
\label{eq_value_denom_phiT}
\end{equation}
As all densities are allowed under $\nu^N_{1/2}$, 
as soon as $\lambda(\lambda+2E)<0$ the denominator $\E^{\nu^N_{1/2}}_{E}[e^{\lambda NQ_T}]$ in the definition~\eqref{eq_def_dynamique_conditionnee_courant} is therefore dominated by the contributions at $0,N$ particles when $T\gg 1$ even though the typical density under $\nu^N_{1/2}$ is $1/2$:
\begin{align}
\E^{\nu^N_{1/2}}_E\Big[e^{\lambda NQ_T}\Big] 
&=
\frac{1}{2^N}\sum_{m=0}^N\binom{N}{m} \E^{U_m}_E\Big[e^{\lambda NQ_T}\Big] 
\nonumber\\
&\approx 
\frac{1}{2^{N-1}}\Big(\E^{U_0}_E\Big[e^{\lambda NQ_T}\Big]+ o_T(1) \Big)
=
\frac{1}{2^{N-1}}\big(1+o_T(1)\big)
,
\end{align}
where the $o_T(1)$ is uniform in $N$. 
Thus for $\lambda(\lambda+2E)<0$ the "annealed" current-biased dynamics~\eqref{eq_def_dynamique_conditionnee_courant} started from $\nu^N_{1/2}$ does not capture any interesting dynamical phenomenon when $N,T$ are large. 
This is the reason for the restriction $\lambda(\lambda+2E)\geq 0$ in Theorem~\ref{theo_main_result}. 
In particular, recalling the definition~\eqref{eq_phiT_in_prop} of the conjectured critical line, 
the restriction $\lambda(\lambda+2E)\geq 0$ means that only $\lambda,E$ far from the critical line can be considered. 

A natural way to bypass this problem would be to consider an initial condition with a fixed number of particles as in~\eqref{eq_current_biased_dynamics_with_unif_measure}. 

Extending Theorem~\ref{theo_main_result} to that case should require no new ideas and would allow for $\lambda,E$ with $\lambda(\lambda+2E)\leq 0$. 
Indeed, the structure of the proof of Theorem~\ref{theo_main_result} 
for $|\lambda(\lambda+2E)|$ small enough would be identical: 
the candidate for the driven process and the time-decorrelation result of the limiting fluctuation fields are the same. 
The key microscopic ingredient, the relative entropy estimate, 
can also be set up at fixed density without change. 
To get bounds similar to those of Theorem~\ref{theo_size_entropy}, 
however, 
we would need an additional input: 
an equivalent of Appendix~\ref{app_concentration} for the canonical measure:
\begin{equation}
U_{\lfloor N/2\rfloor,h} 
:=
\nu^N_h\Big(\cdot\Big| \sum_i \eta_i = \lfloor N/2\rfloor\Big)
.
\label{eq_def_Umh}
\end{equation}
In particular exponential concentration bounds for $n$-point correlations under this measure when $\lambda(\lambda+2E)\leq 0$, 
corresponding to $h$ being a positive kernel (recall Proposition~\ref{prop_solving_PDE_h}) would be required. 
We would also need to prove a central limit theorem for density fluctuations, 
which should be accessible, perhaps using the classical results on mean-field models of~\cite{kusuokaGibbsMeasuresMean1984}. 
As these estimates are fairly technical and would still not allow to cover the full range of sub-critical $\lambda,E$, we did not pursue that generalisation.

\paragraph{The non-perturbative case.}
Let us now discuss how to remove the smallness assumption on $\lambda(\lambda+2E)$ in Theorem~\ref{theo_main_result}. 

The requirement that $\lambda(\lambda+2E)\geq 0$ be small comes from a single, but central technical estimate: 
the domination bound~\eqref{eq_domination_bound_sec2}. 
We claim that this bound can be proven for any $\lambda,E$ with $\lambda(\lambda+2E)\geq 0$ provided the following estimate holds. 

Write $U_{m,h}$ for the measure $\nu^N_h$ conditioned to having $0\leq m\leq N$ particles.  
Assume that, for any $\lambda,E$ in the sub-critical regime of Proposition~\ref{prop_solving_PDE_h}:
\begin{align}
&U_{m,h} \text{ satisfies a logarihmic Sobolev inequality }
\nonumber\\
&\hspace{3cm}\text{with constant }c(\lambda,E)N^2 \text{ uniform in }m\in \{0,...,N\}
.
\end{align}
In other words, assume that there is $c(\lambda,E)$ independent of $N$ and $m\in\{0,...,N\}$ such that, 
for any density $f:\Omega_{N,m}\to\R_+$ for $U_{m,h}$:
\begin{equation}
U_{m,h}(f\log f)
\leq 
c(\lambda,E)N^2 
U_{m,h}\Big(\sum_{i\in\T_N} \big[\sqrt{f(\eta^{i,i+1})} -\sqrt{f(\eta)}\big]^2\Big)
.
\label{eq_LSI_sec_perspectives}
\end{equation}
Then Theorem~\ref{theo_main_result} with the current-biased dynamics started from $\nu^N_{1/2}$ holds for any $\lambda(\lambda+2E)\geq 0$. 
Alternatively, assuming the above log-Sobolev inequality~\eqref{eq_LSI_sec_perspectives} for $m=\lfloor N/2\rfloor$ only as well as a CLT under the measure $U_{\lfloor N/2\rfloor, h}$ of~\eqref{eq_def_Umh}, 
the result of Theorem~\ref{theo_main_result} for 
the current-biased dynamics started from $U_{\lfloor N/2\rfloor}$ would similarly hold in the whole sub-critical regime~\eqref{eq_phiT_in_prop}.

The proof of such a logarithmic Sobolev inequality is currently under investigation.

\section{Preliminaries: fluctuations and correlations under $\Prob^{N}_{h,\lambda+E}$ and relative entropy estimate}\label{sec_fluc_correl_prob_h_lambda}
In this section we establish preliminary results needed in the proof of the main result, Theorem~\ref{theo_main_result}. 
We prove Theorem~\ref{theo_fluct_correl_driven_process} characterising fluctuations and correlations under the approximate driven process $\Prob^N_{h,\lambda+E}$ 
(Sections~\ref{sec_CLT_fluctuations} to~\ref{sec_stationarity}) 
and the relative entropy estimate of Theorem~\ref{theo_size_entropy} (Section~\ref{sec_entropy_estimate}).  

Throughout, we consider the following more general case: rather than $\Prob^N_{h,\lambda+E}$, 
we study fluctuation and correlation processes for a dynamics of the form $\Prob^{N}_{\tilde h,\tilde \lambda}$, 
where $\tilde \lambda\in\R$ and $\tilde h\in C^\infty_D(\T^2)$ is symmetric 
(this set is defined in~\eqref{eq_def_set_c_infty_D}). 
We assume throughout that $\sigma\tilde h$ has leading eigenvalue strictly below $1$: 
for some parameter $\delta\in(0,1)$:
\begin{equation}
\sup_{f\in\mathbb L^2(\T):\|f\|_2=1}\int_{\T^2}f(x) \sigma \tilde h(x,y)f(y)\, dx\, dy
\leq 
(1-\delta)
,\qquad
\sigma := (1/2)(1-1/2) = 1/4
.
\label{eq_assumption_size_tilde_h}
\end{equation}
\begin{rmk}
In the case $\tilde h=h$ with $\tilde\lambda = \lambda+E$, 
the assumption~\eqref{eq_assumption_size_tilde_h} covers the full sub-critical regime, recall Proposition~\ref{prop_solving_PDE_h}. 
In particular and as mentioned above, no limitation on $\lambda,E$ come from Theorems~\ref{theo_fluct_correl_driven_process}--\ref{theo_size_entropy}.
\demo
\end{rmk}
The dynamics $\Prob^N_{\tilde h,\tilde \lambda}$ is started from the discrete Gaussian measure $\nu^N_{\tilde h}$ defined as in~\eqref{eq_def_nu_rho_g} with $\tilde h$ instead of $h$ there. 
The fluctuations are studied as follows.
\begin{itemize}
	\item Limit theorems for the law of the fluctuation process $Y^N_\cdot$ under a local dynamics ($\tilde h=0$) have been obtained in~\cite{Jara2018} using $O_N(1)$ estimates on the relative entropy of the law of the dynamics at each time compared to a suitable product measure. 
	The extension to $\tilde h\neq 0$ does not require new arguments provided a similar relative entropy estimate holds.   
	We therefore only provide a sketch of proof, in Section~\ref{sec_CLT_fluctuations} and refer to~\cite{Jara2018} for details.
	\item The correlation process $\Pi^N_\cdot$ is then expressed in terms of $Y^N_\cdot$ when acting on smooth test functions. 
	This and an approximation argument characterise its limit uniquely when acting also on test functions in $C^\infty_D(\T^2)$ (recall~\eqref{eq_def_set_c_infty_D}) that are smooth up to singularities on the diagonal. 
	Stationarity of the limiting processes $Y_\cdot,\Pi_\cdot$ follows from an explicit formula for the law of $Y_t$ ($t\geq 0$).
\end{itemize}
The following proposition contains the entropy estimate, 
and includes Theorem~\ref{theo_size_entropy} as a special case. 
The $o_N(1)$ bound stated below is stronger than required for the proof of Theorem~\ref{theo_fluct_correl_driven_process}. 
Note that the law of the dynamics is compared to the correlated measure $\nu^N_{\tilde h}$ rather than a product measure as in~\cite{Jara2018}. 
This is essential for the $o_N(1)$ bound and does not change induce any change to the proof of Theorem~\ref{theo_fluct_correl_driven_process} (where an $O_N(1)$ bound is enough) compared to the proof in~\cite{Jara2018}. 
\begin{prop}\label{prop_entropy_estimate_general}
Let $f_t\nu^N_{\tilde h}$ denote the law of the dynamics $\Prob^{\nu^{N}_{\tilde h}}_{\tilde h,\tilde \lambda}$ at time $t\geq 0$. 
There are $\gamma,C>0$ independent of $\tilde h$ such that, for each $T>0$:
\begin{equation}
\forall N\in\N_{\geq 1},\qquad 
\sup_{t\leq T}{\bf H}^N(f_t) 
= 
\sup_{t\leq T}
\leq 
\frac{\gamma C e^{T/\gamma}}{N^{1/2}}
,\qquad\quad
{\bf H}^N(f_t) 
:= 
\nu^N_{\tilde h}(f_t\log f_t)
.
\label{eq_entropy_estimate_prop_sec3}
\end{equation}
This bound implies that, 
for each bounded $\phi:\T^2\rightarrow\R$:
\begin{equation}
\sup_N \sup_{t\leq T} \E^{\nu^N_{\tilde h}}_{\tilde h,\tilde \lambda}\Big[\big|\Pi^N_t(\phi)\big|^{3/2}\Big] 
\leq 
C(T,\phi)
.
\label{eq_moment_bound}
\end{equation}
Moreover, Pinsker's inequalities yields the following useful consequence.
For any function $F_N:\Omega_N\rightarrow\R$ and any $T\geq 0$,
\begin{equation}
\sup_{t\leq T}
\Big|\E^{\nu^N_{\tilde h}}_{\tilde h,\tilde \lambda}\big[F_N(\eta(t))\big] 
- \nu^N_{\tilde h}(F_N)\Big|
\leq 
\|F_N\|_\infty \sup_{t\leq T}\int |f_t-1|\, d\nu^N_{\tilde h}
\leq 
\frac{\|F_N\|_\infty\sqrt{2\gamma Ce^{T/\gamma}}}{N^{1/4}}
.\label{eq_pinsker}
\end{equation}
\end{prop}
The proof of Proposition~\ref{prop_entropy_estimate_general} is quite technical, 
and therefore postponed to Section~\ref{sec_entropy_estimate} where we also discuss practical applications of this bound. 
\subsection{The limiting fluctuation process}
\label{sec_CLT_fluctuations}
Here we sketch the proof of item $3$ in Theorem~\ref{theo_fluct_correl_driven_process} and characterise the law of the fluctuations at time $0$ (part of item $1$). 
The proof of item 3 is standard: we prove tightness of $(Y^N_{t})_{t\geq 0}$ in $\mathcal D([0,\infty),\s'(\T))$, 
then prove that limit points satisfy a martingale problem that has a unique solution.
\subsubsection{Tightness}\label{sec_tightness}
By Mitoma's criterion~\cite[Theorem 3.1]{Mitoma1983} (formulated for trajectories on bounded time intervals, 
but immediately extending to unbounded intervals as in~\cite[Theorem 16.4]{Billingsley1999}), 
tightness of $(Y^N_{t})_{t\geq 0}$ in $\mathcal D([0,\infty),\s'(\T))$ follows from tightness of all projections $(Y^N_{t}(f))_{t\geq 0}$ in $\mathcal D([0,\infty),\R)$ ($f\in C^\infty(\T)$). 
By \cite[Theorems 16.8--16.10]{Billingsley1999}, 
tightness of $(Y^N_{t}(f))_{t\geq 0}$ follows from the following two estimates:
\begin{enumerate}
	\item[(i)] For each $t\geq 0$, 
	\begin{equation}
	\lim_{a\to\infty}\limsup_{N\to\infty}\Prob^{\nu^N_{\tilde h}}_{\tilde h,\tilde\lambda}\Big( \big|Y^N_t(f)\big|\geq a\Big)
	=
	0
	.
	\end{equation}
	\item[(ii)] (Aldous criterion) For $t\geq 0$, let $\mathcal I_{t}$ denote the set of stopping times for $(Y^N_s(f))_{s\geq 0}$ bounded by $t\geq 0$. 
	For each $t\geq 0$ and $\epsilon>0$, 
	\begin{equation}
	\lim_{\delta\to0}\sup_{\tau\in\mathcal I_t}\limsup_{N\to\infty}\Prob^{\nu^N_{\tilde h}}_{\tilde h,\tilde\lambda}\Big( \big|Y^N_{\tau+\delta}(f)-Y^N_{\tau}(f)\big|\geq \epsilon\Big)
	=
	0
	.
	\end{equation}
\end{enumerate}
Both items follow from $O_N(1)$ bounds on the relative entropy ${\bf H}^N(f_t)$ uniform on compact time intervals as we now show. 
Recall the entropy inequality:
\begin{equation}
\E^{\nu^N_{\tilde h}}_{\tilde h,\tilde\lambda}[F(\eta_t)]
\leq 
\frac{{\bf H}^N(f_t)}{\gamma} 
+
\frac{1}{\gamma}\log \nu^N_{\tilde h}\Big[e^{\gamma F}\Big]
,\qquad
\gamma>0, F:\Omega_N\to\R
.
\label{eq_entropy_inequality_sec3}
\end{equation}
Let $t\geq0$ and $a>0$. 
For point (i), Markov- and the entropy inequalities give:
\begin{equation}
\Prob^{\nu^N_{\tilde h}}_{\tilde h,\tilde\lambda}\Big(\big|Y^N_t(f)\big|\geq a\Big)
\leq 
\frac{1}{a}\E^{\nu^N_{\tilde h}}_{\tilde h,\tilde\lambda}\Big[\big|Y^N_t(f)\big|\Big]
\leq
\frac{{\bf H}^N(f_t)}{a} 
+
\frac{1}{a}\log \nu^N_{\tilde h}\Big[\exp\big|Y^N(f)\big|\Big]
.
\label{eq_using_entropy_inequality}
\end{equation}
Proposition~
\ref{prop_concentration_exponentielle} shows that the last exponential moment is bounded uniformly in $N$, 
which yields the claim as the entropy ${\bf H}^N(f_t)$ is also bounded.

For point (ii), the starting point is the semi-martingale decomposition:
\begin{equation}
Y^N_t(f)
=
Y^N_{0}(f) 
+\int_0^T N^2L_{\tilde h,\tilde\lambda}Y^N_s(f)\, ds 
+M^{N}_t(f)
,
\end{equation}
with $(M^{N}_t(f))_{t\geq 0}$ a martingale independent from $Y^N_0$ and starting at $0$.  
Let $t\geq 0$, $\epsilon,\delta>0$ and $\tau\in\mathcal I_t$ denote a stopping time. 
Then Markov property gives, 
recalling that $f_\tau\nu^N_{\tilde h}$ stands for the law of the dynamics at time $\tau$:
\begin{align}
\E^{\nu^N_{\tilde h}}_{\tilde h,\tilde\lambda}\Big[ \big|Y^N_{\tau+\delta}(f)-Y^N_{\tau}(f)\big|\Big]
&=
\E^{\nu^N_{\tilde h}}_{\tilde h,\tilde \lambda}\bigg[ \E^{f_\tau \nu^N_{\tilde h}}_{\tilde h,\tilde \lambda}\Big[\big|Y^N_{\delta}(f)-Y^N_{0}(f)\big|\Big]\bigg]
\nonumber\\
&\hspace{-1cm}\leq 
\E^{\nu^N_{\tilde h}}_{\tilde h,\tilde \lambda}\bigg[ \E^{f_\tau \nu^N_{\tilde h}}_{\tilde h,\tilde \lambda}\Big[\Big|\int_0^\delta N^2L_{\tilde h,\tilde \lambda}Y^N_s(f)\, ds\Big|\Big] + \E^{f_\tau \nu^N_{\tilde h}}_{\tilde h,\tilde \lambda}\Big[\big|M^{N}_{\delta}(f)\big|\Big]\bigg]
.
\label{eq_Aldous}
\end{align}
Note that, by Cauchy-Schwarz inequality, 
the estimate of the martingale term boils down to an estimate on its quadratic variation:
\begin{equation}
\big<M^{N}_{\delta}(f)\big>
=
\int_0^\delta N^2\sum_{i\in\T_N}c_{\tilde h,\tilde\lambda}(\eta_{t'},\eta_{t'}^{i,i+1}) \Big[\Pi^N(f)(\eta_{t'}^{i,i+1}) - \Pi^N(f)(\eta_{t'})\Big]^2\, d{t'}
.
\label{eq_quadratic_variation}
\end{equation}
Using the entropy inequality as in~\eqref{eq_using_entropy_inequality} and noticing that $\tau\leq t$ and that the relative entropy is bounded with $N$ uniformly on $[0,t+\delta]$, 
item (ii) is therefore proven if we can prove that the integrands in~\eqref{eq_Aldous}--\eqref{eq_quadratic_variation} have bounded exponential moment under $\nu^N_{\tilde h}$.  
This is indeed the case. 
As computations similar to that of the action of the generator on $Y^N(f)$ and to~\eqref{eq_quadratic_variation} are carried out to prove Proposition~\ref{prop_entropy_estimate_general}, we omit them here and conclude the proof of tightness. 
\subsubsection{Martingale problem}
\label{sec_martingale_problem}
We now sketch the proof that limit points of $(\Prob^{\nu^N_{\tilde h}}_{\tilde h,\tilde\lambda})$ satisfy the martingale problem of Theorem~\ref{theo_fluct_correl_driven_process}. 

The key argument is the following:
\begin{equation}
Y^N_t(f_t) 
=
Y^N_0(f_0) + \int_0^t Y^N_s(\partial_s f_s +\mathcal L_{\tilde h} f_s)\, ds 
+ 
M^{N}_t(f) + \int_0^t\epsilon^{N}_s(f)\, ds
,
\label{eq_Ito_time_dependent}
\end{equation}
where $f\in C^1(\R_+,C^\infty(\T))$, 
$\mathcal L_{\tilde h}$ is the operator appearing in Theorem~\ref{theo_fluct_correl_driven_process} and $\epsilon^{N}(f)$ is an error term in the sense:
\begin{equation}
\lim_{N\to\infty}\E^{\tilde \nu^N_h}_{\tilde h,\tilde \lambda}\Big[\Big|\int_0^t\epsilon^{N}_s(f)\, ds\Big|\Big]
=
0
.
\label{eq_Boltzmann_Gibbs_sketch}
\end{equation}
The above estimate is an instance of the so-called Boltzmann-Gibbs principle.   
The expression of $\epsilon^N(f)$ is obtained through computations similar to those carried out in the proof of Proposition~\ref{prop_entropy_estimate_general}, 
so we give no more detail. 
Estimates such as~\eqref{eq_Boltzmann_Gibbs_sketch} are proven in Corollary~\ref{coro_bound_error_terms} below, 
see also~\cite{Jara2018} for a different presentation. \\

Equation~\eqref{eq_Ito_time_dependent} defines a limiting process $(M_t(f))_{t\geq 0}$ for each limit point $(Y_t)_{t\geq 0}$ of $(Y^N_t)_{t\geq 0}$. 
This process is $(Y_s)_{s\leq t}$-measurable, 
independent from $Y_0$ and starts at $0$ by construction. 
It is also continuous as $(Y_t)_{t\geq 0}$ itself is a continuous process due to $(Y^N_t(f))_{t\geq 0}$ having jumps of amplitude bounded by $\|f\|_\infty N^{-1/2}$.

The continuous process $(\mathfrak M_t(f))_{t\geq 0} = \big( M_t(f)^2-\int_0^t\|\partial_x f_s\|_2\, ds\big)_{t\geq 0}$ can be constructed similarly.
To show that $(M_t(f))_{t\geq 0}, (\mathfrak M_t(f))_{t\geq 0}$ are martingales, 
it is enough to show that $(M^{N}_t(f)^{2+\epsilon})_{N\in\N_{\geq 1}}$ is uniformly integrable for some $\epsilon>0$~\cite[Propositions IX.1.4 and IX.1.12]{Jacod2003}. 
This again follows from explicit computations of the quadratic variation of $M^N_t(f)$ and $O_N(1)$ relative entropy bounds, see~\cite[(6.7) and Corollary 2.3]{Jara2018} for a similar argument. \\

The fact that $(M_t(f))_{t\geq 0},(\mathfrak M_t(f))_{t\geq 0}$ are martingales characterises $(M_t(f))_{t\geq 0}$ by a result of Dubins-Schwarz: for each $t\geq 0$,
\begin{equation}
M_t(f)
=
B^f_{2\sigma\int_0^t\|\partial_x f_s\|_2^2\, ds}
,\qquad
(B^f_t)_{t\geq 0}\text{ a standard Brownian motion}. 
.
\label{eq_M_as_Brownian_motion}
\end{equation}
For $t\geq 0$, 
taking as test function $f_{s,t} = e^{(t-s)\mathcal L_{\tilde h}}f\in C^1([0,t],C^\infty(\T))$ for $f\in C^\infty(\T)$ (the regularity is proven in Appendix~\ref{app_BFP}), 
we find that any limit point $(Y_t)_{t\geq 0}$ of the $(Y^N_t)_{t\geq 0}$ satisfies:
\begin{equation}
Y_t(f) 
=
Y_0\big(f_t\big) 
+
B^{f_{\cdot,t}}_{2\sigma\int_0^t\|\partial_x f_{s,t}\|_2^2\, ds}
.
\label{eq_Y_t_as_Y_0_plus_mart}
\end{equation}
It follows that the law of $(Y_t)_{t\geq 0}$ is uniquely determined as soon as the law of the initial condition $Y^N_0$ converges, which we check next. 
\subsubsection{Initial condition}
\label{sec_initial_condition}
\begin{prop}[Law at the initial time]\label{prop_limit_nu_g}
The weak limit $\nu^\infty_{\tilde h}$ of the law of $Y^N_0$ under $\Prob^{\nu^N_{\tilde h}}_{\tilde h}$ is a centred Gaussian field on $\s'(\T)$, 
with covariance given for $f_1,f_2\in C^\infty(\T)$ by:
\begin{equation}
\E_{\nu^\infty_{\tilde h}}\Big[Y_0(f_1)Y_0(f_2)\Big] 
= 
\int_{\T} f_1(x)\big(C_{\tilde h}f_2\big)(x)dx,
\quad 
C_{\tilde h} 
:= 
\big(\sigma^{-1}\text{id} - \tilde h\big)^{-1}
. 
\end{equation}
Above, the function $\tilde h$ is identified with the kernel operator $\tilde h f_2(\cdot) = \int_{\T}\tilde h(\cdot,y)f_2(y)dy$.
\end{prop}
\begin{proof}
Let us prove that the law $\nu^N_{\tilde h}\circ (Y^N)^{-1}$ of $Y_0^N$ under $\Prob^{\nu^N_{\tilde h}}_{\tilde h,\tilde \lambda}$ converges to a Gaussian field with covariance $C_{\tilde h}$. 
To do so, we compute moments of all order. 
Let $n\in\N$ and $f^1,...,f^{2n+1} \in C^\infty(\T)$. 
It is an immediate consequence of the invariance of $\nu_{\tilde h}$ under the mapping $\eta\mapsto (1-\eta_i)_{i\in\T_N}$ that odd moments of $Y^N$ vanish:
\begin{equation}
\nu^N_{\tilde h}\Big[\prod_{i=1}^{2n+1} Y^N(f^i)\Big] 
= 
0
.
\end{equation}
Above and in the following, we use $\nu^N_{\tilde h}(F)$ or $\nu^N_{\tilde h}[F]$ to denote expectation of $F:\Omega_N\to\R$ under $\nu^N_{\tilde h}$. 
Proposition \ref{prop_Wick_theorem} provides the following estimate of even moments. 
Let $\tilde h^N$ denote the matrix:
\begin{equation}
\tilde h^N(i,j)
:=
\tilde h_{i,j}
=
\tilde h\Big(\frac{i}{N},\frac{j}{N}\Big)
,\qquad
\tilde h^N(i,i)
=
0,
\qquad 
i\neq j\in\T_N
.
\end{equation}
The following approximate Wick formula then holds: 
there is $e^n_N = O(N^{-n-1})$ such that, 
with $P_{2n}$ the set of pairings of $2n$ elements and $I_N$ the $N\times N$ identity matrix:
\begin{equation}
\sup_{J\subset \Lambda_N : |J|=2n}
\Big|\nu^N_{\tilde h}\Big(\prod_{j\in J}\bar\eta_j\Big) 
- \sum_{p\in P_{2n}} \prod_{j=1}^n \big(\sigma^{-1}I_N-N^{-1}\tilde h^N)^{-1}(j,p(j))\Big|
\leq 
e^{n}_N
.
\label{eq_local_Wick_sec3}
\end{equation}
The inverse matrix above is well defined for large enough $N$, see the discussion around~\eqref{eq_matrix_inverse_exists}. 
Using~\eqref{eq_local_Wick_sec3} for each set $J$ with even cardinal $|J|\leq 2n$ as well as:
\begin{equation}
\bar\eta^2_i
=
\sigma 
\quad (i\in\T_N),
\qquad
\sup_{j\in\T_N}\Big|\big(\sigma^{-1}I_N-N^{-1}\tilde h^N)^{-1}(j,j)
-\sigma\Big|
= 
O(N^{-1})
,
\end{equation} 
we find:
\begin{align}
\nu^N_{\tilde h}\Big[\prod_{i=1}^{2n} Y^N(f^i)\Big] 
&= 
\frac{1}{N^n}\sum_{j_1,...,j_{2n}\in\T_N}\nu^N_{\tilde h}\Big[\prod_{i=1}^{2n} \bar\eta_{j_i}f^{i}_{j_i}\Big] 
\\
&= 
\sum_{p\in P_{2n}} \prod_{i=1}^n\bigg(\frac{1}{N}\sum_{a,b\in\T_N} f^i_{a} f^{p(i)}_{b} \big(\sigma^{-1}I_N-N^{-1}\tilde h^N\big)^{-1}(a,b)\bigg)
+e'_N(f^1,...,f^p)
,
\nonumber
\end{align}
where $e'_N(f_1,...,f_p) = O(N^{-1})$. 
To conclude, it remains to prove that, for any $f^1,f^2\in C^\infty(\T)$:
\begin{equation}
\lim_{N\to\infty}\frac{1}{N}\sum_{a,b\in\T_N} f^1_{a} f^{2}_{b} \big(\sigma^{-1}I_N-N^{-1}\tilde h^N\big)^{-1}(a,b) 
=
\int_{\T} f^1(x)C_{\tilde h}f^2(x)\, dx
.
\label{eq_to_prove_WIck_thm_sec3}
\end{equation}
Let $f\in C^\infty(\T)$ and, 
for an $N\times N$ matrix $M$ define, define for each $i\in\T_N$:
\begin{equation}
Mf(i)
:=
\sum_{j\in\T_N} M(i,j) f_j
.
\end{equation}
From $\tilde h\in C^\infty_D(\T^2)\subset C^0(\T^2)$ we get:
\begin{equation}
\sup_{x\in\T}\Big|(\sigma^{-1}\text{id}-\tilde h)f(x) - (\sigma^{-1}I_N-N^{-1}\tilde h^N)f\big(\lfloor x/N\rfloor/N\big)\Big|
\underset{N\to\infty}{\longrightarrow}
0
.
\label{eq_convergence_discrete_to_C_inverse}
\end{equation}
In addition, we prove in Lemma~\ref{lemm_LU} that $(\sigma^{-1}I_N-N^{-1}\tilde h^N)^{-1}$ satisfies, for some $c>0$:
\begin{equation}
\sup_{N\geq 1}\max_{i\in\T_N}\big|(\sigma^{-1}I_N-N^{-1}\tilde h^N)^{-1}(i,i)\big| 
\leq 
c,
\qquad
\sup_{N\geq 1}\max_{i\neq j\in\T_N}\big|(\sigma^{-1}I_N-N^{-1}\tilde h^N)^{-1}(i,j)\big| 
\leq 
\frac{c}{N}
.
\end{equation}
Using these bounds with~\eqref{eq_convergence_discrete_to_C_inverse} applied to $C_{\tilde h}f$, 
identified with the vector $\big((C_{\tilde h}f)_i\big)_{i\in\T_N}$, we find that the vector:
\begin{align}
(\sigma^{-1}I_N-N^{-1}\tilde h^N)^{-1}
-C_{\tilde h}
=
(\sigma^{-1}I_N-N^{-1}\tilde h^N)^{-1}\Big[f -(\sigma^{-1}I_N-N^{-1}\tilde h^N)C_{\tilde h}f\Big]
\end{align}
has largest entry in absolute value that vanishes with $N$. 
This implies~\eqref{eq_to_prove_WIck_thm_sec3} and concludes the proof.
\end{proof}
Putting together the results of Sections~\ref{sec_tightness} to~\ref{sec_initial_condition} proves item 3 and part of item 1 of Theorem~\ref{theo_fluct_correl_driven_process}. 
In the next section, we move on to characterise limiting correlations, proving items 2, 4 and 5. 
The stationarity claim of item 1 is proven in Section~\ref{sec_stationarity}.  
\subsection{The limiting correlation process}
Here, we prove that the correlation process $\Pi^N_\cdot\in \mathcal D([0,+\infty),\s'_D(\T^2))$ under $\Prob^{\nu^N_{\tilde h}}_{\tilde h,\tilde\lambda}$ has a unique weak limit, 
as written out in Theorem~\ref{theo_fluct_correl_driven_process}, and characterise its properties. 
This is first done by expressing $\Pi_\cdot$ in terms of $Y_\cdot$ for smooth test functions. 
An approximation argument then gives a characterisation of the correlation process also including test functions that are not smooth on the diagonal.  
Let $(e_\ell)_{\ell\geq 0}$ be an orthonormal basis of eigenvectors of the Laplacian on the torus: $e_0 =1$, and:
\begin{equation}
\forall \ell\in\N_{\geq 1},\forall x\in\T,\qquad e_{2\ell-1}(x) = \sqrt{2}\sin(2\pi \ell x),\quad e_{2\ell}(x) = \sqrt{2}\cos(2\pi \ell x).\label{eq_eigenbasis_laplacien}
\end{equation}
\begin{prop}\label{prop_lien_Y_Pi}
\begin{enumerate}
\item (Smooth test functions). 
Consider $\Pi^N_\cdot$ as an element of the set $\mathcal D([0,+\infty),\s'(\T^2))$ of distributions tested against smooth test functions only. 
Then the couple $(Y^N_\cdot,\Pi^N_\cdot)$ converges in law to a measure ${\bf P}^{\nu^\infty_{\tilde h}}_{\tilde h,\tilde\lambda}$ on $\mathcal D([0,+\infty),\s'(\T)\times\s'(\T^2))$. 
Under this measure, for each $(f_1,f_2)\in C^\infty(\T)^2$, with probability $1$:
\begin{equation}
\Pi_t(f_1\otimes f_2) 
= 
\frac{1}{4}Y_t(f_1)Y_t(f_2) - \frac{\sigma}{4}\int_{\T}f_1(x)f_2(x)\, dx
,\qquad 
t\geq 0
.
\label{eq_lien_Y_Pi_continu}
\end{equation}
\item (Non-smooth test functions). 
For each $t\geq 0$, $\Pi_t$ admits a unique continuous extension to $\s'_D(\T^2)$ and each vector $(\Pi^N_{t_1}(\phi_1),...,\Pi^N_{t_p}(\phi_p))$ converges in law to $(\Pi_{t_1}(\phi_1),...,\Pi_{t_p}(\phi_p))$ for $p\in\N$ and $t_1,...,t_p\geq 0$, $\phi_1,...,\phi_p\in C^\infty_D(\T^2)$.
Moreover, there is $C(\tilde h,\tilde\lambda)>0$ such that, for each $\epsilon>0$:
\begin{equation}
{\bf P}^{\nu^\infty_{\tilde h}}_{\tilde h,\tilde\lambda}\Big(|\Pi_t(\phi_1)\big|>\epsilon\Big)
\leq 
\frac{C(\tilde h,\tilde \lambda)\|\phi_1\|_\infty}{\epsilon}
,\qquad 
t\geq 0
.
\label{eq_continuity_Pi}
\end{equation}
\end{enumerate}
\end{prop}
\begin{proof}
1. Take $f_1,f_2\in C^\infty(\T)$ and $N\in\N_{\geq 1}$. 
From the definition of $\Pi^N$ in~\eqref{eq_def_correlations}, 
one has with probability one:
\begin{equation}
\forall t\geq 0,\qquad 
\Pi^N_t(f_1\otimes f_2) = \frac{1}{4}Y^N_t(f_1)Y^N_t(f_2)- \frac{\sigma}{4 N}\sum_{i\in\T_N}(f_1)_i(f_2)_i.\label{eq_Pi_as_fct_Y_limit}
\end{equation}
In view of the characterisation of tightness of a process in Section~\ref{sec_tightness} and since all test functions in $C^\infty(\T^2)$ can be approximated uniformly by linear combinations of tensor products of functions of one variable, 
it follows that the sequence of laws of $(Y^N_\cdot,\Pi^N_\cdot)$ is tight as probability measures on $\mathcal D([0,+\infty),\s'(\T)\times\s'(\T^2))$. 
As the sequence of laws of $Y^N_\cdot$ converges and as $\Pi^N_t(f_1\otimes f_2)$ is a continuous function of $Y^N_t(f_1),Y^N_t(f_2)$ for each $t\geq 0$, 
the sequence of laws of $(Y^N_\cdot,\Pi^N_\cdot)$ converges to a measure ${\bf P}^{\nu^\infty_{\tilde h}}_{\tilde h,\tilde\lambda}$. 
The validity of~\eqref{eq_Pi_as_fct_Y_limit} for each $N$ and the fact that the second term in the right-hand side of~\eqref{eq_Pi_as_fct_Y_limit} is a (non-random) Riemann sum for $\sigma\int_{\T}f_1(x)f_2(x)\, dx$ conclude the proof of 1.

\noindent 2. Let $t\geq 0$. 
As $\Pi_t$ is a bounded linear form, it is a uniformly continuous function on the complete metric space $C^\infty(\T^2)$ that is dense in $C^0(\T^2)\supset C^\infty_D(\T^2)$ for the topology of uniform convergence. 
It therefore admits a unique bounded extension on $C^0(\T^2)'\supset \s'_D(\T^2)$. 
Fix now $p\in\N$, $t_1,...,t_p\geq 0$, $\phi_1,...,\phi_p\in C^\infty_D(\T^2)$ and sequences $(\phi^n_i)_{n}\in C^\infty(\T^2)^{\N}$ converging uniformly to $\phi_i$ ($1\leq i\leq p$). 
We do the proof for $p=2$, the general case being the same. 
Let $F:\R^2\to\R$ be a bounded Lipschitz function and write: 
\begin{align}
F\big(\Pi_{t_1}(\phi_1)&,\Pi_{t_2}(\phi_2)\big) 
- F\big(\Pi^N_{t_1}(\phi_1),\Pi^N_{t_2}(\phi_2)\big) 
\nonumber\\
&=
F\big(\Pi_{t_1}(\phi_1),\Pi_{t_2}(\phi_2)\big)
- F\big(\Pi_{t_1}(\phi^n_1),\Pi_{t_2}(\phi^n_2)\big) 
\nonumber\\
&\quad 
+ F\big(\Pi_{t_1}(\phi^n_1),\Pi_{t_2}(\phi^n_2) \big)
-F\big(\Pi^N_{t_1}(\phi^n_1),\Pi^N_{t_2}(\phi^n_2) \big)
\nonumber\\
&\quad 
+ F\big(\Pi^N_{t_1}(\phi^n_1),\Pi^N_{t_2}(\phi^n_2) \big)
- F\big(\Pi^N_{t_1}(\phi_1),\Pi^N_{t_2}(\phi_2) \big)
.
\end{align}
The middle line vanishes with $N$ at fixed $n$ when integrating on the laws of $\Pi_\cdot,\Pi^N_\cdot$ by weak convergence for smooth test functions. 
The average of the first line under ${\bf P}^{\nu^\infty_{\tilde h}}_{\tilde h,\tilde \lambda}$ vanishes with $n$ by definition of the extension of $\Pi$ and the dominated convergence theorem. 
Lastly, the average of the third line vanishes with $n$ uniformly in $N$. 
Indeed, as $F$ is Lipschitz (say with respect to the norm $\|(x,y)\|_1= |x|+|y|$), it is enough to control times $t_1,t_2$ separately. 
E.g. for $t_1$, the entropy inequality gives, 
for each $\gamma>0$:
\begin{equation}
\E^{\nu^N_{\tilde h}}_{\tilde h,\tilde \lambda}\Big[\big|\Pi^N_{t_1}(\phi^n_1 - \phi_1)\big|\Big]
\leq 
\frac{\|\phi^n_1-\phi_1\|_\infty}{\gamma}\Big[{\bf H}^N(f_{t_1})
+ \log \nu^N_{\tilde h}\Big(\exp\Big|\frac{\gamma \Pi^N(\phi^n_1 - \phi_1)}{\|\phi^n_1-\phi_1\|_\infty}\Big|\Big)\Big]
.
\label{eq_interm_proof_continuity_Pi}
\end{equation}
The entropy is bounded uniformly in $N$ by Proposition~\ref{prop_entropy_estimate_general} and Proposition~\ref{prop_concentration_exponentielle} yields the existence of $\gamma,C>0$ independent of $N,\phi^n_1,\phi_1$ such that the exponential moment is bounded by $C$. 
This concludes the proof of 2. except for~\eqref{eq_continuity_Pi} that we now prove. 
Let $t\geq 0$ and $\epsilon>0$. 
Since $(\Pi^N_t(\phi_1))_N$ converges weakly to $\Pi_t(\phi_1)$, 
\begin{equation}
{\bf P}^{\nu^\infty_{\tilde h}}_{\tilde h,\tilde\lambda}\big(|\Pi_t(\phi_1)|>\epsilon\big)
\leq 
\liminf_{N\to\infty}\Prob^{\nu^N_{\tilde h}}_{\tilde h,\tilde\lambda}\big(|\Pi^N_t(\phi_1)|>\epsilon\big)
.
\end{equation}
Markov- and the entropy inequality give the claim as in~\eqref{eq_interm_proof_continuity_Pi}.
\end{proof}
\subsection{Stationarity}
\label{sec_stationarity}
The last unproven claim of Theorem~\ref{theo_fluct_correl_driven_process} is the stationarity of the processes $(Y_t)_{t\geq 0}$, $(\Pi_t)_{t\geq 0}$. \\

In view of Proposition~\ref{prop_lien_Y_Pi}, 
it is enough to prove that $(Y_t)_{t\geq 0}$ is stationary, 
i.e. that the Gaussian field $\nu^\infty_{\tilde h}$ on $\s'(\T)$ with covariance $C_{\tilde h} = (\sigma^{-1}\text{id} - \tilde h)^{-1}$ is invariant. 
It is enough to check that the law of $(Y_t(f_1),...,Y_t(f_n))$ ($t\geq 0$) does not depend on time for each $f_1,...,f_n\in C^\infty(\T)$ and $n\in\N_{\geq 1}$. 
By Lévy's characteristic function theorem and the linearity of $Y_t$, 
establishing the $n=1$ case is sufficient. 

Fix $f\in C^\infty(\T)$. 
Equations~\eqref{eq_M_as_Brownian_motion}--\eqref{eq_Y_t_as_Y_0_plus_mart}, 
the fact that $Y_0$ is a centred Gaussian field with covariance $C_{\tilde h}:= (\sigma^{-1}\text{id}-h)^{-1}$ and the independence of $Y_0$ and $M_t(f)$ imply that 
$Y_t(f)$ is a centred Gaussian random variable with variance:
\begin{equation}
{\bf E}^{\nu^\infty_{\tilde h}}_{\tilde h,\tilde \lambda}\big[Y_t(f)^2\big]
=
\int_{\T}f_t(x) (C_{\tilde h} f_t)(x)\, dx + 2\sigma\int_0^t\|\partial_x f_s\|_2^2\, ds
.
\label{eq_sum_variances_stationarity}
\end{equation}
Let us prove that this variance is just $\int_{\T}f(x)(C_{\tilde h}f)(x)\, dx$, which will suffice. 
Recall the expression~\eqref{eq_def_L_star_theo_fluct} of $\mathcal L_{\tilde h}$ and
let $t\mapsto f_t = e^{t\mathcal L_{\tilde h}}f\in C^\infty(\R_+\times\T)$ (see Appendix~\ref{app_BFP} for the regularity). 
Then:
\begin{equation}
\partial_t f_t 
=
\mathcal L_{\tilde h}f_t 
=
(\text{id}-\sigma\tilde h)\Delta f_t 
=
\sigma C_{\tilde h}^{-1}\Delta f_t
,
\end{equation}
where we used $C_{\tilde h}=\sigma(\text{id}-\sigma\tilde h)^{-1}$ in the last line.  
Integrating the second term in the right-hand side of~\eqref{eq_sum_variances_stationarity} by parts and writing the first term as its $t=0$ value plus the integral of its time derivative, 
we find that the variance of $Y_t(f)$ is indeed $\int_{\T}f(x)(C_{\tilde h}f)(x)\, dx$, concluding the proof:
\begin{align}
\E^{\nu^\infty_{\tilde h}}_{\tilde h,\tilde \lambda}\big[Y_t(f)^2\big]
&=
\int_{\T}f(x)(C_{\tilde h}f)(x)\, dx
+2\int_0^t\int_{\T}\big[f_sC_{\tilde h}\partial_s f_s -\sigma f_s\Delta f_s\big](x)\, dx\, ds
\nonumber\\
&=
\int_{\T}f(x)(C_{\tilde h}f)(x)\, dx
.
\end{align}

\subsection{The relative entropy estimate}\label{sec_entropy_estimate}
In this section, we prove the $o_N(1)$ relative entropy estimate of Proposition~\ref{prop_entropy_estimate_general}, 
which includes Theorem~\ref{theo_size_entropy} as a special case. 

The main idea is the following: the more information we put on the measure approximating the law of the dynamics $\Prob^N_{\tilde h,\tilde\lambda}$ at each time, 
the better the resulting estimate. 
To obtain information on the two-point correlation structure of the law of the dynamics, 
one needs to tune the density profile as well as two-point correlations of the approximation measure.  
One way to do so is to compare the dynamics with a measure of the form $\nu^N_g$ defined as in~\eqref{eq_def_nu_rho_g} for a symmetric $g\in C^\infty_D(\T^2)$, 
and optimise on $g$ to get the correct correlation structure of the dynamics.

For the dynamics we study here the optimal choice $g=\tilde h$. 
This is not true in general and only valid here due to working at density $1/2$. 
We choose to nonetheless work with a general $g\in C^\infty_D(\T^2)$ in the following. 
This is to showcase that the optimal $g$ is in general the solution of a partial differential equation involving $\tilde h$ (see~\cite{Correlations2022}). 
Moreover, the computations for general $g$ are very similar to those needed later in Section~\ref{sec_Radon_Nikodym_derivative}.\\

\noindent\textbf{Notations:} 
A function $g\in C^\infty_D(\T^2)$ is henceforth fixed such that $\sigma g$ has leading eigenvalue eigenvalue strictly below $1$. 
We write $\nu^N_g$ for the associated measure as in~\eqref{eq_def_nu_rho_g}. 
If $f$ is a density for $\nu^N_g$, the relative entropy of $f\nu^N_g$ with respect to $\nu^N_g$ is denoted by:
\begin{equation}
{\bf H}^N_g(f) 
:=
\nu^N_g(f\log f)
,
\end{equation}
where $\nu^N_g(\cdot)$ (or $\nu^N_g[\cdot]$) denotes expectation under $\nu^N_g$. 
In the special case $g=\tilde h$ we simply write ${\bf H}^N_{\tilde h}(f) = {\bf H}^N(f)$. \\

To prove Proposition~\ref{prop_entropy_estimate_general}, 
the following lemma is the key starting point.
\begin{lemm}[Lemmas A.1-A.2 in~\cite{Jara2018}]\label{lemm_derivative_entropy}
For $t\geq 0$, 
recall that ${\bf H}^N_g(f_t) = \nu^N_{g}(f_t\log f_t)$ denotes the relative entropy of the law of the dynamics $f_t \nu^N_{g}$ at time $t$ with respect to $\nu^N_{g}$. 
Then: 
\begin{equation}
\partial_t{\bf H}^N_g(f_t)
\leq 
- N^2\nu^N_{g}\big(\Gamma_{\tilde h,\tilde \lambda}(\sqrt{f_t})\big)
+ N^2\nu^N_{g}\big(f_t L^*_{g,\tilde h,\tilde\lambda}{\bf 1}\big)
.
\label{eq_derivative_entropy}
\end{equation}
Above, $\Gamma_{\tilde h,\tilde\lambda}$ is the carre du champ associated with the dynamics (recall the definition~\eqref{eq_def_jump_rates_lambda}--\eqref{eq_def_jump_rates_h_and_lambda} of the jump rates): for each function $F:\Omega_N\rightarrow\R$,
\begin{equation}
\Gamma_{\tilde h,\tilde\lambda}(F)(\eta) 
:=
\sum_{i\in\R_N} c_{\tilde h,\tilde\lambda}(\eta,\eta^{i,i+1}) \big[F(\eta^{i,i+1})- F(\eta)\big]^2
.
\end{equation}
The operator $N^2L^*_{g,\tilde h,\tilde\lambda}$ is the adjoint of the generator $N^2L_{\tilde h,\tilde\lambda}$ of the dynamics in $\mathbb L^2(d\nu^N_{\tilde h})$: 
\begin{equation}
N^2L^*_{g,\tilde h,\tilde\lambda}F (\eta)
:=
N^2\sum_{i\in\T_N} \Big[c_{\tilde h,\tilde \lambda}(\eta,\eta^{i,i+1})F(\eta^{i,i+1})\frac{\nu^N_{ g}(\eta^{i,i+1})}{\nu^N_{ g}(\eta)} - c_{\tilde h,\tilde \lambda}(\eta,\eta^{i,i+1})F(\eta)\Big]
.
\end{equation}
A similar expression arises in the bound of exponential moments of time-integrated observables using the Feynman-Kac inequality:
\begin{align}
&\frac{1}{T}\log \E^{\nu^N_{g}}_{\tilde h,\tilde\lambda}\Big[\exp\Big[\int_0^T F(\eta_t)\, dt\Big]\Big] 
\nonumber\\
&\hspace{2cm}\leq
\sup_{f\geq 0:\nu^N_{g}(f) = 1}\Big\{\nu^N_{g}(fF) + \frac{N^2}{2}\nu^N_{g}(fL^*_{g,\tilde h,\tilde\lambda}{\bf 1}) - \frac{N^2}{2}\nu^N_{g}\big(\Gamma_{\tilde h,\tilde\lambda}(\sqrt{f})\big)
\Big\}
.
\end{align}
\end{lemm}
Using Lemma~\ref{lemm_derivative_entropy}, 
the entropy estimate~\eqref{eq_entropy_estimate_prop_sec3} in Proposition~\ref{prop_entropy_estimate_general} follows from the following claim as proven next.
\begin{lemm}\label{lemm_bound_adjoint}
Take $g=\tilde h$ and write ${\bf H}^N := {\bf H}^N_{\tilde h}$ for the associated relative entropy. 
For each $N\in\N_{\geq 1}$, 
there is a function $\mathcal R_{\tilde h,\tilde\lambda}:\Omega_N\rightarrow\R$ such that, 
for any density $f$ for $\nu^N_{\tilde h}$:
\begin{equation}
\nu^N_{\tilde h}\big(fN^2L^*_{\tilde h, \tilde h,\tilde\lambda}{\bf 1}\big) 
\leq 
\frac{N^2}{2}\nu^N_{\tilde h}\big(\Gamma_{\tilde h,\tilde\lambda}(\sqrt{f})\big) + \nu^N_{\tilde h}\big(f\mathcal R_{\tilde h,\tilde\lambda}\big)
.
\end{equation}
Moreover, there are $\gamma,C>0$ such that:
\begin{equation}
\forall N\in\N_{\geq 1},\qquad
\frac{1}{\gamma}\log\nu^N_{\tilde h}\Big(\exp\big[\gamma|\mathcal R_{\tilde h,\tilde\lambda}|\big]\Big)
\leq
\frac{C}{N^{1/2}}
.
\label{eq_exp_moment_R_N_h}
\end{equation}
\end{lemm}
Let us prove the entropy estimate in Proposition~\ref{prop_entropy_estimate_general} assuming Lemma~\ref{lemm_bound_adjoint}.  
The remaining part of Proposition~\ref{prop_entropy_estimate_general}, 
the moment bound on $\Pi^N_t$, 
is proven in Corollary~\ref{coro_moment_bound} below. 
Recall that the entropy inequality (see Appendix A.8 in~\cite{Kipnis1999}) states: 
for any density $f$ for $\nu^N_{\tilde h}$, 
any $F:\Omega_N\rightarrow\R$ and $\gamma>0$,
\begin{equation}
\nu^N_{\tilde h}\big(f F\big)
\leq 
\frac{{\bf H}^N(f)}{\gamma} + \frac{1}{\gamma} \log \nu^N_{\tilde h}\Big(e^{\gamma F}\Big)
.
\label{eq_entropy_ineq}
\end{equation}
Using Lemma~\ref{lemm_bound_adjoint} and the entropy inequality, 
the bound~\eqref{eq_derivative_entropy} on $\partial_t{\bf H}^N(f_t)$ becomes:
\begin{align}
\partial_t {\bf H}^N(f_t)
&\leq 
- \frac{N^2}{2}\nu^N_{\tilde h}\big(\Gamma_{\tilde h,\tilde \lambda}(\sqrt{f_t})\big)
+ \nu^N_{\tilde h}\big(\mathcal R_{\tilde h,\tilde \lambda}\big)
\nonumber\\
&\leq 
\frac{{\bf H}^N(f_t)}{\gamma} +\frac{1}{\gamma}\log \nu^N_{\tilde h}\Big[ e^{\gamma \mathcal R_{\tilde h,\tilde\lambda} }\Big]
\leq 
\frac{{\bf H}^N(f_t)}{\gamma} + \frac{C}{N^{1/2}}
.
\end{align}
Gronwall inequality and ${\bf H}^N(f_0)=0$ then give ${\bf H}^N(f_t)\leq \gamma C e^{t/\gamma} N^{-1/2}$ as claimed..\\

The proof of Lemma~\ref{lemm_bound_adjoint} spans the next two sections, 
with useful concentration estimates (such as~\eqref{eq_exp_moment_R_N_h}) established in Appendix~\ref{app_concentration}. 

Before we start proving Lemma~\ref{lemm_bound_adjoint}, 
we give a useful corollary. 
It states that the relative entropy estimate implies the so-called Boltzmann-Gibbs principle, 
or replacement lemma at the level of fluctuations,
as already observed in~\cite{Jara2018}. 
Corollary~\ref{coro_bound_error_terms} was already used in Section~\ref{sec_martingale_problem} to identify the limiting martingale process. 
\begin{coro}[Bounding error terms]\label{coro_bound_error_terms}
Assume $\theta^N:\Omega_N\rightarrow\R$, $N\in\N_{\geq 1}$ satisfies the following conditions: 
there are functions $\zeta^N_{\delta}$ ($N\in\N_{\geq 1}$, $\delta\in(0,1/2)$) such that,
for any $f$ for $\nu^N_{\tilde h}$:
\begin{equation}
\nu^N_{\tilde h}\big(f(\pm \theta^N)\big) 
\leq 
\delta N^2 \nu^N_{\tilde h}\big(\Gamma(\sqrt{f})\big) + \nu^N_{\tilde h}\big(f\zeta^N_\delta\big)
, 
\label{eq_def_zeta_N_corollary}
\end{equation}
and for some $\gamma>0$ depending on $\delta$:
\begin{equation}
\limsup_{N\rightarrow\infty}\log \nu^N_{\tilde h}\Big[\exp\big(\gamma|\zeta^N_\delta|\big)\Big] 
=
0
.
\end{equation}
Then:
\begin{equation}
\limsup_{N\rightarrow\infty}
\E^{\nu^N_{\tilde h}}_{\tilde h,\tilde \lambda}\Big[\Big|
\int_0^T\theta^N(\eta_t)\, dt\Big|\Big]
=
0
.
\end{equation}
\end{coro}
If $\theta^N = \zeta^N_\delta$ for some $\delta\in(0,1/2)$, 
the assumptions in the corollary are just saying that $\gamma|\theta^N|$ has exponential moments under $\nu^N_{\tilde h}$ equal to $1+o_N(1)$. 
The result is then a direct consequence of the entropy estimate of Proposition~\ref{prop_entropy_estimate_general} applied at each time $t\in[0,T]$. 
The more subtle $\theta^N\neq \zeta^N_\delta$ case will be proven at the end of Section~\ref{sec_bound_adjoint}. 
\subsubsection{Computation of the adjoint}
Here, we compute the adjoint $N^2L^*_{g,\tilde h,\tilde\lambda}{\bf 1}(\eta)$ of the generator in $\mathbb L^2(\nu^N_{g})$ for a fixed configuration $\eta\in\Omega_N$. 
For any $f:\Z^2\rightarrow\R$ and any $(i,j)\in\Z^2$, 
write for short:
\begin{align}
\partial^{N}_1 f(i,j) 
&:= 
N\big[f(i+1,j)-f(i,j)\big]
,\nonumber\\
\Delta^N_1 f(i,j) 
&:= 
N^2\big[f(i+1,j) + f(i-1,j) -2f(i,j)\big]
.
\end{align}
Recall that $\bar\eta_i = \eta_i-1/2$ for $i\in\T_N$. 
Elementary computations yield, for each $i\in\T_N$:
\begin{align}
\frac{\nu^N_g(\eta^{i,i+1})}{\nu^N_g(\eta)} 
&= 
\exp\big[-2(\eta_{i+1}-\eta_i)B^g_i/N\big]
,
\end{align}
with:
\begin{equation}
B^g_i(\eta) 
=: 
\frac{1}{2N}\sum_{j\notin\{i,i+1\}} \bar\eta_j\partial^N_1 g_{i,j}
.
\label{eq_def_B_g_x}
\end{equation}
Similarly,
\begin{align}
c_{\tilde h,\tilde\lambda}(\eta,\eta^{i,i+1}) 
&= 
c_{\tilde\lambda}(\eta,\eta^{i,i+1})\exp\big[-(\eta_{i+1}-\eta_i)B^{\tilde h}_i/N\big]
\nonumber\\
&=
c(\eta,\eta^{i,i+1})\exp\Big[-(\eta_{i+1}-\eta_i)\frac{\big(\tilde\lambda + B^{\tilde h}_i\big)}{N}\Big]
.
\label{eq_variation_jump_rate_c_h_lambda}
\end{align}
With these relations, one has for the adjoint $N^2L^*{\bf 1}$:
\begin{align}
N^2L^*_{g,\tilde h,\tilde\lambda}{\bf 1}(\eta) 
&= 
N^2\sum_{i\in\T_N}\Big[c_{\tilde h,\tilde\lambda}(\eta^{i,i+1},\eta)\frac{\nu^N_g(\eta^{i,i+1})}{\nu_g(\eta)} - c_{\tilde h,\tilde\lambda}(\eta,\eta^{i,i+1})\Big]
\nonumber\\
&=
N^2\sum_{i\in\T_N}\eta_i(1-\eta_{i+1})\Big(\exp\Big[\frac{-\tilde\lambda + B^{2g-\tilde h}_i}{N}\Big] -\exp\Big[\frac{\tilde\lambda+B^{\tilde h}_i}{N}\Big]\Big)
\nonumber\\
&\quad + 
N^2\sum_{i\in\T_N}\eta_{i+1}(1-\eta_i)\Big(\exp\Big[\frac{\tilde\lambda - B^{2g-\tilde h}_i}{N}\Big] -\exp\Big[-\frac{\tilde\lambda+B^{\tilde h}_i}{N}\Big]\Big)
. 
\end{align}
Developing the exponentials using $e^x = 1+x + \frac{x^2}{2} + \frac{x^3}{6}+ \int_0^1 \frac{(1-t)^3}{6}x^4e^{tx}\, dt$ yields:
\begin{align}
N^2L^*_{g,\tilde h,\tilde\lambda}{\bf 1}(\eta) 
&= 
2N\sum_{i\in\T_N}\big(\eta_{i+1}-\eta_i\big)\big(\tilde\lambda-B^{g-\tilde h}_i\big) + 2\sum_{i\in\T_N}c(\eta,\eta^{i,i+1}) \Big[\big(B^{g-\tilde h}_i - \tilde\lambda\big)B^{g}_i\Big] 
\nonumber\\
&\quad + 
\frac{1}{6N}\sum_{i\in\T_N}\big(\eta_{i+1}-\eta_i\big) \Big[\big(\tilde\lambda-B^{2g-\tilde h}_i\big)^3 - \big(\tilde\lambda + B^{\tilde h}_i\big)^3\Big] 
+ \delta^N_{\text{order}\geq 4}(\eta)
\label{eq_L_star_is_L_1_plus_L_2}\\
&=: 
\sum_{k=1}^{3}\big(N^2L^*_{g,\tilde h,\tilde\lambda}{\bf 1}(\eta)\big)_{|\text{order k}}  + \delta^N_{\text{order}\geq 4}(\eta)\nonumber
,
\end{align}
with:
\begin{equation}
\sup_{\eta\in\Omega_N}\big|\delta^N_{\text{order}\geq 4}(\eta)\big|
=
O(N^{-1})
.
\end{equation}
The bound on $\delta^N_{\text{order}\geq 4}$ comes from the fact that $B^g_i, B^{\tilde h}_i$ are bounded uniformly in $i,N,\eta$.  
We now focus on computing $\big(N^2L^*_{g,\tilde h,\tilde \lambda}{\bf 1}(\eta)\big)_{|\text{order k}}$ for $k\in\{1,2,3\}$. \\

\noindent\textbf{Computing} $\big(N^2L^*_{g,\tilde h,\tilde\lambda}{\bf 1}(\eta)\big)_{|\text{order 1}}${\bf .} 
Since $(\eta_{i+1}-\eta_i) = (\bar\eta_{i+1}-\bar\eta_i)$, 
upon integrating $\bar\eta_i$ by parts in the first sum in \eqref{eq_L_star_is_L_1_plus_L_2} and re-indexing, one finds:
\begin{align}
\big(N^2L^*_{g,\tilde h,\tilde\lambda}{\bf 1}(\eta)\big)_{|\text{order 1}} 
&= 
\frac{1}{N}\sum_{i\in\T_N}\sum_{j : |j-i|>1}\bar\eta_i\bar\eta_j \Delta^N_1 (g-h)_{i,j}
\nonumber\\
&\quad
+ \sum_{i\in\T_N}\bar\eta_i\bar\eta_{i+1}\big[\partial^{N}_1 (g-\tilde h)_{i+1,i} - \partial^{N}_1 (g-\tilde h)_{i-1,i+1}\big]
.
\label{eq_expr_finale_L_star_order_1}
\end{align}
\textbf{Computing }$\big(N^2L^*_{g,\tilde h,\tilde\lambda}{\bf 1}(\eta)\big)_{|\text{order 2}}${\bf .} 
Recall definition \eqref{eq_def_B_g_x} of $B^g$. 
One has:
\begin{align}
\big(N^2L^*_{g,\tilde h,\tilde\lambda}{\bf 1}(\eta)\big)_{|\text{order 2}} 
&= 
\frac{1}{2N^2}\sum_{j,\ell\in\T_N}\bar\eta_j\bar\eta_\ell \Big(\frac{1}{N}\sum_{\substack{i\in \T_N \\ i\notin\{j-1,j,\ell-1,\ell\}}} c(\eta,\eta^{i,i+1})\partial^{N}_1 (g-\tilde h)_{i,j}\partial^{N}_1 g_{i,\ell}\Big) 
\nonumber\\
&\quad
- \frac{\tilde\lambda}{N}\sum_{i\in\T_N}c(\eta,\eta^{i,i+1}) \sum_{j\notin\{i,i+1\}}\bar\eta_j\partial^{N}_1 g_{i,j}
.
\label{eq_order_2_Lstar_0}
\end{align}
To proceed, rewrite the jump rates $c(\eta,\eta^{i,i+1})$ in terms of $\bar\eta_i,\bar\eta_{i+1}$ as follows:
\begin{equation}
c(\eta,\eta^{i,i+1})
=
\eta_i(1-\eta_{i+1}) + \eta_{i+1}(1-\eta_{i}) 
=
\frac{1}{2}- 2\bar\eta_i\bar\eta_{i+1}
=
2\big(\sigma - \bar\eta_i\bar\eta_{i+1}\big)
.
\end{equation}
Splitting also the first sum in~\eqref{eq_order_2_Lstar_0} between $j=\ell$ diagonal terms and non-diagonal ones and recalling $\bar\eta_i^2 = \sigma$, 
one finds:
\begin{align}
\big(N^2L^*_{g,\tilde h,\tilde\lambda}{\bf 1}(\eta)\big)_{|\text{order 2}} 
&= 
\frac{\sigma}{N}\sum_{\substack{j,\ell\in\T_N\\ j\neq\ell}}\bar\eta_j\bar\eta_\ell \Big(\frac{1}{N}\sum_{\substack{i\in \T_N \\ i\notin\{j-1,j,\ell-1,\ell\}}} \partial^{N}_1( g-\tilde h)_{i,j}\partial^{N}_1 g_{i,\ell}\Big) 
\nonumber\\
&\quad 
+ \frac{\sigma^2}{N^2}\sum_{j\in\T_N}\sum_{\substack{i\in \T_N \\ i\notin\{j-1,j\}}} \partial^{N}_1 (g-\tilde h)_{i,j} \partial^{N}_1 g_{i,j}+ \epsilon_{g,\tilde h,\tilde\lambda} + \e_{g}
.
\label{eq_expr_finale_L_star_order_2}
\end{align}
The term $\e_{g}$ contains third-order correlations that will require more work to estimate:
\begin{align}
\e_{g}(\eta) 
&= 
\frac{2\tilde \lambda}{N}\sum_{i\in\T_N} \bar\eta_i\bar\eta_{i+1}\sum_{j\notin\{i,i+1\}}\bar\eta_j \partial^{N}_1 g_{i,j}
.
\label{eq_def_W_3}
\end{align}
The term $\epsilon_{g,\tilde h}$ in \eqref{eq_expr_finale_L_star_order_2} contains all other terms that will be shown to be small:
\begin{align}
\epsilon_{g,\tilde h,\tilde\lambda} 
&= 
-\frac{1}{N}\sum_{i\in\T_N}\bar\eta_{i}\bar\eta_{i+1}\Big(\frac{1}{N}\sum_{j\notin\{i,i+1\}}\sigma\partial^{N}_1 (g-\tilde h)_{i,j}\partial^{N}_1 g_{i,j}\Big)
\nonumber\\
&\quad 
-\frac{1}{N}\sum_{i\in\T_N}\bar\eta_{i}\bar\eta_{i+1}\Big(\frac{1}{N}\sum_{j\neq \ell\notin\{i,i+1\}}\bar\eta_j\bar\eta_\ell\partial^{N}_1 g_{i,j}\partial^{N}_1 g_{i,\ell}\Big)
\nonumber\\
&\quad 
-\frac{2\tilde \lambda\sigma}{N}\sum_{i\in\T_N}\sum_{j\notin \{i,i+1\}}\bar\eta_j \partial^{N}_1 g_{i,j}
.
\label{eq_def_epsilon_N}
\end{align}
Putting together the initial expression~\eqref{eq_L_star_is_L_1_plus_L_2} of the adjoint and the first and second order estimates~\eqref{eq_expr_finale_L_star_order_1}--\eqref{eq_expr_finale_L_star_order_2}, 
we have so far obtained:
\begin{align}
N^2L^*_{g,\tilde h,\tilde\lambda}{\bf 1}(\eta) 
&=
\frac{1}{N}\sum_{i\in\T_N}\sum_{j : |j-i|>1}\bar\eta_i\bar\eta_j \Delta^N_1 (g-\tilde h)_{i,j}
\nonumber\\
&\quad
+ \sum_{i\in\T_N}\bar\eta_i\bar\eta_{i+1}\big[\partial^{N}_1 (g-\tilde h)_{i+1,i} - \partial^{N}_1 (g-\tilde h)_{i-1,i+1}\big]
\nonumber\\
&\quad 
+ \frac{\sigma}{N}\sum_{\substack{j,\ell\in\T_N\\ j\neq\ell}}\bar\eta_j\bar\eta_\ell \Big(\frac{1}{N}\sum_{\substack{i\in \T_N \\ i\notin\{j-1,j,\ell-1,\ell\}}} \partial^{N}_1( g-\tilde h)_{i,j}\partial^{N}_1 g_{i,\ell}\Big) 
\nonumber\\
&\quad 
+ \frac{\sigma^2}{N^2}\sum_{j\in\T_N}\sum_{\substack{i\in \T_N \\ i\notin\{j-1,j\}}} \partial^{N}_1 (g-\tilde h)_{i,j} \partial^{N}_1 g_{i,j}+ \delta_{g,\tilde h,\tilde\lambda}(\eta)
,
\label{eq_final_estimate_L_star_before_choice_g}
\end{align}
where $\delta_{g,\tilde h}$ is given by:
\begin{equation}
\delta_{g,\tilde h,\tilde\lambda} 
= 
\epsilon_{g,\tilde h,\tilde\lambda} + \e_{g} + \big(N^2L^*_{g,\tilde h,\tilde\lambda}{\bf 1}(\eta)\big)_{|\text{order 3}} + O(N^{-1})
,
\label{eq_def_delta_g_h}
\end{equation}
with the $O(N^{-1})$ uniform on the configuration.

Postponing the precise estimate of $\delta_{g,\tilde h,\tilde\lambda}$ to Section~\ref{sec_bound_adjoint}, 
let us informally explain why the choice $g=\tilde h$ is optimal.  
The contribution of two-point correlations
i.e. the first four lines of~\eqref{eq_final_estimate_L_star_before_choice_g}, 
are expected to typically be of order $1$ in $N$ since $\Pi^N$ is, informally speaking, the square of the approximately Gaussian random variable $Y^N$. 
On the other hand $\delta_{g,\tilde h,\tilde\lambda}$ involves higher order correlations or other terms that we expect to typically be of order $o_N(1)$.   
Thus, in order to make the adjoint as small as possible as stated in Lemma~\ref{lemm_derivative_entropy}, 
we need to choose $g$ so that all two-point correlations vanish. 
This is the case when $g=\tilde h$ and we obtain:
\begin{equation}
N^2L^*_{\tilde h,\tilde h,\tilde\lambda}{\bf 1}(\eta) 
= 
\delta_{\tilde h,\tilde h,\tilde\lambda}
.
\label{eq_adjoint_equals_delta}
\end{equation}
Therefore bounding $\delta_{\tilde h,\tilde h,\tilde\lambda}$ is the same as bounding the adjoint.

In Section~\ref{sec_driven_process_candidate} we will make use of a bound of $\delta_{g,\tilde h,\tilde\lambda}$ for $g\neq \tilde h$, so we keep a general $g$ for now.  
In the next section, 
we conclude the proof of Lemma~\ref{lemm_bound_adjoint} by constructing, 
for each $\alpha\in(0,1/2]$, 
a function $\mathcal R_{g,\tilde h,\tilde\lambda,\alpha}:\Omega_N\rightarrow\R$ such that, 
for each density $f$ for $\nu^N_{\tilde h}$:
\begin{equation}
\nu^N_{g}\big(f\delta_{g,\tilde h,\tilde\lambda}\big) 
\leq 
\alpha N^2 \nu^N_{g}\big(\Gamma(\sqrt{f})\big) 
+ \nu^N_{g}\big(f\mathcal R_{g,\tilde h,\tilde\lambda,\alpha}\big)
,
\label{eq_estimate_error_terms}
\end{equation}
with, for some $\gamma_\alpha,C_\alpha>0$ independent of $N$ and all $N\in\N_{\geq 1}$:
\begin{equation}
\frac{1}{\gamma_\alpha}\log \nu^N_{g}\Big(\exp\big|\gamma_\alpha\mathcal R_{g,\tilde h,\tilde\lambda,\alpha}\big|\Big)
\leq 
\frac{C_\alpha}{N^{1/2}}
.
\end{equation}
The function $\mathcal R_{\tilde h,\tilde\lambda}$ appearing in Lemma~\ref{lemm_derivative_entropy} is simply $\mathcal R_{g=\tilde h,\tilde h,\tilde\lambda,\alpha=1/2}$. 
\subsubsection{Proof of~\eqref{eq_estimate_error_terms} and estimate of error terms}\label{sec_bound_adjoint}
Here, we complete the proof of Lemma~\ref{lemm_bound_adjoint} by proving~\eqref{eq_estimate_error_terms}.  
At the end of the section, 
we also establish the moment bound in Proposition~\ref{prop_entropy_estimate_general} and prove Corollary~\ref{coro_bound_error_terms}, 
which in particular implies the estimate~\eqref{eq_Boltzmann_Gibbs_sketch}. 
Recall that we still work with a general $g\in C^\infty_D(\T^2)$ such that $\sigma g$ has leading eigenvalue strictly below 1.

The proof of~\eqref{eq_estimate_error_terms} is split into two parts: 
the estimate of $\e_{\tilde h}$ (recall~\eqref{eq_def_W_3}), 
and the estimate of the rest, 
i.e. $\epsilon_{g,\tilde h,\tilde\lambda}+(N^2L^*_{g,\tilde h,\tilde\lambda}{\bf 1})_{|\text{order }3}$ (recall that $\epsilon_{g,\tilde h,\tilde\lambda}$ is defined in~\eqref{eq_def_epsilon_N}).

In Proposition~\ref{prop_concentration_exponentielle}, 
the following concentration estimates under $\nu^N_{\tilde h}$ are established. 
Let $n\geq 1$ be an integer and $\phi_n:\T^n_N\rightarrow\R$ ($N\in\N_\geq 1$) be a sequence of functions satisfying $\sup_N \|\phi_n\|_\infty <\infty$. 
For $J\subset\T_N$ and $n\geq 2$, define:
\begin{equation}
W^{n,J}_{\phi_n}(\eta)
:=
\frac{1}{N^{n-1}}\sum_{i_1,...,i_n\in\T_N}\phi_n(i_1,...,i_n) \Big(\prod_{j\in J}\bar\eta_{i_1+j}\Big)\bar\eta_{i_2}...\bar\eta_{i_n},
\qquad \eta\in\Omega_N
\label{eq_def_W_n_J_phi}
.
\end{equation}
If $n=1$, define instead:
\begin{equation}
W^{1,J}_{\phi_1}(\eta) 
:=
\frac{1}{N^{1/2}}\sum_{i\in\T_N}\phi_1(i) \Big(\prod_{j\in J}\bar\eta_{i+j}\Big)
,
\qquad \eta\in\Omega_N
.
\end{equation}
There is then $\gamma_n,C_n>0$ depending only $\tilde h,J$ but not on $\phi_n$ such that, 
for each $N\in\N$:
\begin{equation}\label{eq_bound_exp_moments_sec_2}
\frac{1}{\gamma_n}\log \nu^N_{g}\Big(\exp\Big[\frac{\gamma_n |W^{n,K}_{\phi_n}|}{\sup_N\|\phi_n\|_\infty}\Big]\Big)
\leq 
\begin{cases}
C_1\quad &\text{if }n\in\{1,2\},\\
\frac{C_n}{N^{\frac{n-2}{2}}}&\text{if }n\geq 3.
\end{cases}
\end{equation}
Let us use~\eqref{eq_bound_exp_moments_sec_2} to bound $\epsilon_{g,\tilde h,\tilde \lambda}$ and $\big(N^2L^*_{g,\tilde h,\tilde\lambda}{\bf 1}(\eta)\big)_{|\text{order }3}$. 
The first line of $\epsilon_{g,\tilde h,\tilde\lambda}$ in~\eqref{eq_def_epsilon_N} is of the form $N^{-1/2} W^{1,\{0,1\}}_{\phi_1}$ for a bounded $\phi_1$ depending on $g,\tilde h$ and given by:
\begin{equation}
\phi_1(i)
:=
-\frac{1}{N}\sum_{j\notin\{i,i+1\}}\sigma\partial^{N}_1 (g-\tilde h)_{i,j}\partial^{N}_1 g_{i,j}
,\qquad
i\in\T_N
.
\end{equation}
On the other hand, the second line of $\epsilon_{g,\tilde h,\tilde\lambda}$ is of the form $W^{3,\{0,1\}}_{\phi_3}$ with:
\begin{equation}
\phi_3(i,j,\ell)
=
-{\bf 1}_{j\neq \ell\notin\{i,i+1\}}\partial^{N}_1 g_{i,j}\partial^{N}_1 g_{i,\ell}
,\qquad
(i,j,\ell)\in\T^3_N
.
\end{equation}
Lastly, the last line of $\epsilon_{g,\tilde h,\tilde\lambda}$ reads, 
exchanging sums and integrating by parts:
\begin{align}
-\frac{2\tilde\lambda\sigma}{N}\sum_{i\in\T_N}\sum_{j\notin \{i,i+1\}}\bar\eta_j \partial^{N}_1 g_{i,j}
&=
-2\tilde\lambda\sigma\sum_{j\in\T_N}\bar\eta_j\sum_{i\notin \{j-1,j\}}\big[g_{i+1,j}-g_{i,j}\big]
\nonumber\\
&=
\frac{2\tilde\lambda\sigma}{N}\sum_{j\in\T_N}\bar\eta_j N \big[g_{j+1,j}-g_{j-1,j}\big]
.
\label{eq_bound_IPP_epislon_term}
\end{align}
Since $g\in C^\infty_D(\T^2)$, $N$ times the bracket is bounded uniformly in $N,j$:
\begin{align}
\psi_2(j)
:=&\,
N \big[g_{j+1,j}-g_{j-1,j}\big] 
= 
N \big[g_{j+1,j}-g_{j_+,j}\big] + N\big[g_{j_-,j} - g_{j-1,j}\big]
\nonumber\\
=&\,
\partial_1 g_{j_+,j} - \partial_1 g_{j_-,j} 
+ 
o_N(1)
,
\end{align}
with the $o_N(1)$ uniform in $j$ and where we used the continuity of $g$ on the diagonal.  
The last line of~\eqref{eq_bound_IPP_epislon_term} is thus of the form $N^{-1/2}W^{1,\{0\}}_{\psi_1}$. 
From~\eqref{eq_bound_exp_moments_sec_2} and the entropy inequality applied to each three terms in $\epsilon_{g,\tilde h,\tilde\lambda}$, 
it follows that there are constants $\gamma_1,\gamma'_1,\gamma_3>0$ and $C_1,C'_1,C_3>0$ such that, 
for any density $f$ for $\nu^N_{g}$ (recall ${\bf H}^N_g(f):= \nu^N_{g}(f\log f)$):
\begin{align}
\nu^N_{g}\big(f\epsilon_{g,\tilde h,\tilde\lambda}\big) 
&\leq 
\frac{{\bf H}^N_g(f)}{\gamma_1 N^{1/2}}
+\frac{1}{\gamma_1}\log \nu^N_g\Big( \exp\big[\gamma_1 W^{1,\{0,1\}}_{\phi_1}\big]\Big)
+\frac{ {\bf H}^N_g(f)}{\gamma_3N^{1/2}} 
\nonumber\\
&\quad 
+\frac{1}{\gamma_3}\log \nu^N_{g}\Big( \exp\big[\gamma_3 W^{3,\{0,1\}}_{\phi_3}\big]\Big)
+\frac{ {\bf H}^N_g(f)}{\gamma'_1N^{1/2}} 
+\frac{1}{\gamma'_1 N^{1/2}}\log \nu^N_{g}\Big( \exp\big[\gamma'_1 W^{1,\{0\}}_{\psi_1}\big]\Big)
\nonumber\\
&\leq 
{\bf H}^N_g(f)\Big[\frac{1}{\gamma_3} +\frac{1}{\gamma_1 N^{1/2}} + \frac{1}{\gamma'_1 N^{1/2}}\Big]
+\frac{C_1+C'_1+C_3}{N^{1/2}}
.
\label{eq_bound_epsilon}
\end{align}
Consider now $\big(N^2L^*_{g,\tilde h,\tilde\lambda}{\bf 1}(\eta)\big)_{|\text{order 3}}$. 
Recalling the expression~\eqref{eq_L_star_is_L_1_plus_L_2} of the third order term,
\begin{align}\label{eq_development_L_star_order_3}
\big(N^2L^*_{g,\tilde h,\tilde\lambda}{\bf 1}(\eta)\big)_{|\text{order 3}}  
&=
\frac{1}{6N}\sum_{i\in\T_N}\big(\eta_{i+1}-\eta_i\big) \Big[-6\tilde\lambda^2 B^{g}_i + 12\tilde\lambda B_i^{g}B_i^{g-\tilde h} -\big(B_i^{2g-\tilde h}\big)^3 - \big(B_i^{\tilde h}\big)^3\Big]
.
\end{align}
Since $B^g_i,B^{\tilde h}_i$ are bounded uniformly in $i,N$ 
and $\eta_{i+1}-\eta_i = \bar\eta_{i+1}-\bar\eta_i$,
~\eqref{eq_development_L_star_order_3} can be bounded as follows for some $C(g,\tilde h,\tilde \lambda)>0$:
\begin{equation}\label{eq_development_L_star_order_3_bis}
\big(N^2L^*_{g,\tilde h,\tilde\lambda}{\bf 1}(\eta)\big)_{|\text{order 3}}
\leq 
-\frac{\tilde\lambda^2}{N}\sum_{i\in\T_N}(\bar\eta_{i+1}-\bar\eta_i) B^g_i
+
\frac{C(g,\tilde h,\tilde\lambda)}{N}\sum_{i\in\T_N}\Big[ (B_i^g)^2 + (B_i^{\tilde h})^2\Big]
.
\end{equation}
The first term in the right-hand side of~\eqref{eq_development_L_star_order_3_bis} can be integrated by parts, 
after which it becomes bounded by $O(N^{-1})$ uniformly in the configuration. 
On the other hand, 
for each $i\in\T_N$, 
$(B^g_i)^2,(B^{\tilde h}_i)^2$ are of the form $N^{-1}W^{2,\{0\}}_{\psi_2^i}$ for a family $\psi^i_2:\T_N^2\to\R$ 
satisfying $\sup_{N,i,\T_N^2}|\psi_2^i|<\infty$. 
Equation~\eqref{eq_bound_exp_moments_sec_2} and the entropy inequality thus yield similarly to~\eqref{eq_bound_epsilon}, 
for some $C_2,\gamma_2>0$ depending only on $g,\tilde h,\tilde\lambda$:
\begin{equation}\label{eq_bound_average_Lstar_order3}
\nu^N_g\Big(f\big(N^2L^*_{g,\tilde h,\tilde\lambda}{\bf 1}(\eta)\big)_{|\text{order 3}}\Big)
\leq 
\frac{{\bf H}^N_g(f)}{\gamma_2 N} + \frac{C_2}{N}
.
\end{equation}
Putting together~\eqref{eq_bound_epsilon}--\eqref{eq_bound_average_Lstar_order3}, 
we have shown the existence of $\gamma,C>0$, 
depending only on $g,\tilde h,\tilde \lambda$, 
such that (recall the definition~\eqref{eq_def_W_3} of $\e_g$ and~\eqref{eq_def_delta_g_h} of $\delta_{g,\tilde h,\tilde\lambda}$):
\begin{equation}
\nu^N_{g}\big(f\delta_{g,\tilde h,\tilde\lambda}\big)
\leq 
\nu^N_{g}\big(f\e_{g}\big) + \frac{{\bf H}^N_g(f)}{\gamma} + \frac{C}{N^{1/2}}
.
\end{equation}
Lemma~\ref{lemm_estimate_threepoint} below provides an estimate of $\e_g$, 
concluding the proof of~\eqref{eq_estimate_error_terms} with the function $\mathcal R_{g,\tilde h,\tilde\lambda,\alpha}$ ($\alpha\in(0,1/2]$) appearing there given by: 
\begin{equation}
\mathcal R_{g,\tilde h,\tilde \lambda,\alpha} 
:= 
\big(N^2L^*_{g,\tilde h,\tilde\lambda}{\bf 1}\big)_{\text{order }3} + \epsilon_{g,\tilde h,\tilde\lambda} + \mathcal W_{g,\tilde h,\tilde\lambda,\alpha}
,
\end{equation}
and $\mathcal W_{g,\tilde h,\tilde\lambda,\alpha}$ defined in the next lemma.
\begin{lemm}\label{lemm_estimate_threepoint}
Let $\e_{g}$ be given by~\eqref{eq_def_W_3} and $\alpha\in(0,1/2]$. 
There is $\mathcal W_{g,\tilde h,\tilde\lambda,\alpha}:\Omega_N\rightarrow\R$ such that, 
for any density $f$ for $\nu^N_{g}$:
\begin{equation}
\nu^N_{g}\big(f \e_{g}\big)
\leq 
\alpha N^2 \nu^N_{g}\big(\Gamma_{\tilde h,\tilde\lambda}(\sqrt{f})\big)
+ \nu^N_{g}\big(f \mathcal W_{\tilde h,\tilde \lambda}\big)
,
\end{equation}
and, for some $\gamma,C>0$ depending on $g,\tilde h,\tilde\lambda,\alpha$ but not on $N$:
\begin{equation}
\frac{1}{\gamma}\log \nu^N_{g}\Big(\exp\big[\gamma |\mathcal W_{\tilde h,\tilde\lambda,\alpha}|\big]\Big)
\leq 
\frac{C}{N^{1/2}}
.
\end{equation}
\end{lemm}
\begin{proof}[Proof of Lemma~\ref{lemm_estimate_threepoint}]
Fix a density $f$ for $\nu^N_g$ throughout the proof.
Consider more generally a family of functions $\phi_2:\T^2_N\rightarrow\R$ that is bounded uniformly in $N$ and let us prove Lemma~\ref{lemm_estimate_threepoint} for $W^{2,\{0,1\}}_{\phi_2}$ (recall~\eqref{eq_def_W_n_J_phi}) 
rather than $\e_{\tilde h}$ (defined in~\eqref{eq_def_W_3}). 
A direct application of the entropy inequality only yields a bound $O_N(1)$ on the exponential moment. 
To improve this bound to $O(N^{-1/2})$, 
we make use of the long, diffusive time-scale, 
which appears in the bound~\eqref{eq_derivative_entropy} on the entropy production $\partial_t {\bf H}^N(f_t)$ of the relative entropy through the carré du champ.  
Following the technique of~\cite{Jara2018}, 
we estimate the cost of replacing
the sum on $\bar\eta_i\bar\eta_{i+1}$ by a sum on local averages of $\bar\eta$'s. 
To do so, let $a\in(0,1/2)$. 
For ease of notation, we write $aN$ instead of $\lfloor a N\rfloor$. 
Split $W^{2,\{0,1\}}_{\phi_2}$ as follows: 
\begin{equation}
W^{2,\{0,1\}}_{\phi_2} 
=
W^+ + W^- 
+W^{=}
,
\end{equation}
with, for $\eta\in\Omega_N$:
\begin{align}
W^+(\eta) 
&:=
\frac{1}{N}\sum_{i\in\T_N} \sum_{j\notin\{i,...,i+aN\}}\bar\eta_i\bar\eta_{i+1}\bar\eta_j \phi_2(i,j),
\nonumber\\
W^-(\eta)
&:= 
\frac{1}{N}\sum_{i\in\T_N}\sum_{j\in \{i+2,...,i+aN\}} \bar\eta_i\bar\eta_{i+1}\bar\eta_j \phi_2(i,j)
\nonumber\\
W^=(\eta)
&:=
\frac{1}{N}\sum_{i\in\T_N}\sum_{j\in\{i,i+1\}} \bar\eta_i\bar\eta_{i+1}\bar\eta_j \phi_2(i,j)
.
\end{align}
Note that, 
recalling $(\bar\eta_i)^2=\sigma$ ($i\in\T_N$), 
$W^=$ is of the form $N^{-1/2}W^{1,\{0,1\}}_{\psi_1}$ for $\psi_1 = \sigma[\phi_2(i,i)+\phi_2(i,i+1)]$. 
One thus directly has by~\eqref{eq_bound_exp_moments_sec_2} and the entropy inequality, 
for some $\gamma,C>0$:
\begin{equation}
\nu^N_{\tilde h}\big(f W^=\big)
\leq 
\frac{{\bf H}^N_g(f)}{\gamma N^{1/2}} + \frac{1}{\gamma N^{1/2}}\log\nu^N_g\Big(\exp\big[\gamma W^=\big]\Big)
\leq 
\frac{{\bf H}^N_g(f)}{\gamma N^{1/2}}
+ \frac{C}{N^{1/2}}
.
\end{equation}
We now estimate $W^{+}$ and $W^-$.  
For $i\in\T_N$, define the local averages:
\begin{equation}
\bar\eta^{+,a N}_i 
:= 
\frac{1}{a N}\sum_{\ell =i}^{i+aN-1}\bar\eta_\ell,
\qquad
\bar\eta^{-,a N}_i 
:= 
\frac{1}{a N}\sum_{\ell =i-aN+1}^{i}\bar\eta_\ell
,
\end{equation}
with $i\pm (aN -1)$ understood modulo $N$. 
The only difference in the estimate of $W^{+},W^-$ is that we replace each $\eta_{i+1}$ (respectively $\eta_i$) by $\eta^{+,aN}_{i+1}$ (respectively $\eta^{-,aN}$) ($i\in\T_N$), 
so we only treat $W^+$. 
Introduce $W^{+,aN}$, defined for $\eta\in\Omega_N$ by:
\begin{align}
W^{+,aN}(\eta)
:&= 
\frac{1}{N}\sum_{i\in\T_N}\sum_{j\notin\{i,..,i+aN\}}\bar\eta_i\bar\eta_{i+1}^{+,aN}\bar\eta_j \phi_2(i,j)
\nonumber\\
&=
\frac{1}{aN^2}\sum_{i\in\T_N}\sum_{j\notin\{i,..,i+aN\}}\sum_{\ell\in\{i+1,...,i+aN\}}\bar\eta_i\bar\eta_{\ell}\bar\eta_j \phi_2(i,j)
.
\end{align}
As $a<1/2$ implies that the sums on $j,\ell$ do not overlap, 
$W^{+,aN}$ is of the form $W^{3,\{0\}}_{\phi_3}$ for a bounded $\phi_3$, 
so its log-exponential moments are bounded by $O(N^{-1/2})$ according to~\eqref{eq_bound_exp_moments_sec_2}.  
We now show that the difference $W^+-W^{+,aN}$ can be estimated in terms of the carré du champ. 
To do so, notice first the following elementary identity:
\begin{equation}
\bar\eta_{i+1} - \bar\eta_{i+1}^{+,aN}
=
\sum_{p\in\{i+1,...,i+aN-1\}} \varphi_a(p-i)\big(\bar\eta_{p} - \bar\eta_{ p+1}\big),
\qquad 
\varphi_a(n) = {\bf 1}_{n\in\{1,...,aN-1\}}\Big(1-\frac{n}{aN}\Big)
.
\end{equation}
Define then, for each $p\in\T_N$:
\begin{equation}
u_p(\eta)
:=
\frac{1}{N}\sum_{\substack{i\in\T_N \\ i\in\{ p-(aN-1),...,p-1\}}}\sum_{j\notin \{i,...,i+aN\}}\varphi_a(|p-i|_{\T_N})\bar\eta_i\bar\eta_j \phi_2(i,j)
\qquad \eta\in\Omega_N
,
\end{equation}
with $|p-i|_{\T_N}$ the torus distance between $i,p$. 
The earlier choice of $a\in(0,1/2)$ removes any ambiguity in the definition of this distance and:
\begin{equation}
W^{+}(\eta) - W^{+,aN}(\eta)
=
\sum_{p\in\T_N}(\bar\eta_{p}-\bar\eta_{p+1})u_p(\eta),
\qquad
\eta\in\Omega_N
.
\end{equation}
Applying the integration by parts formula of Lemma~\ref{lemm_IPP} for each $p\in\T_N$ to $u = u_p$, 
then summing over $p$
gives the existence of $C= C(g,\tilde h,\lambda)>0$ such that:
\begin{align}
\nu^N_{g}\big(f\big[W^{+}- W^{+,aN}\big]\big)
&\leq
\alpha N^2\nu^N_{g}\big(\Gamma_{\tilde h,\tilde\lambda}(\sqrt{f})\big)
+ \frac{C}{\alpha N^2}\sum_{p\in\T_N}\nu^N_{g}\big(f |u_p|^2\big)
\nonumber\\
&\hspace{-1.5cm}
+ \sum_{p\in\T_N}\nu^N_{g}\bigg( fu_p(\bar\eta_{p}-\bar\eta_{p+1}) \Big(\exp \Big[-\frac{2(\bar\eta_{p+1}-\bar\eta_{p})}{N} B^{\tilde h}_p\Big]-1\Big)\bigg)
.
\label{eq_bound_W_3_final_0}
\end{align}
Let us estimate the different terms in the right-hand side of ~\eqref{eq_bound_W_3_final_0}. 
The quantity:
\begin{equation}
\mathcal R^+_\alpha 
:=
\frac{C}{\alpha N}\sum_{p\in\T_N} |u_p|^2
\end{equation}
is of the form $W^{4,\{0\}}_{\phi_4}$ for some bounded $\phi_4$. 
By~\eqref{eq_bound_exp_moments_sec_2}, 
it thus satisfies, for some $\gamma_4,C_4>0$ (depending on $\alpha$):
\begin{equation}
\nu^N_g\big(f\mathcal R^+_\alpha\big)
\leq 
\frac{{\bf H}^N_g(f)}{\gamma_4} + \frac{C}{\gamma_4 N}
.
\end{equation}
On the other hand, using $|e^x - 1 - x| \leq cx^2$ for some $c>0$ and all $x$ smaller than $2\sup_{p,N}N^{-1}|B^{\tilde h}_p|$, 
we find that the second line in~\eqref{eq_bound_W_3_final_0} is bounded by:
\begin{align}
\nu^N_{g}\big(f\mathcal R^{+}\big)
:= 
-\frac{2}{N}\sum_{p\in\T_N}\nu^N_{g}\big( fu_pB^{\tilde h}_p\big) +\frac{c}{N}\sum_{p\in\T_N} \Big(\sup_{p,N}\frac{|u_p|}{N}\Big)\nu^N_{g}\big( f\big(B^{\tilde h}_p\big)^2\big)
.
\label{eq_def_R_+}
\end{align}
The first term is an average of functions of the form $W^{3,\{0\}}_{\phi^p_3}$, 
the second one of the form $N^{-1}W^{2,\{0\}}_{\phi^p_2}$ for functions $\phi^p_2,\phi^p_3$ bounded uniformly in $p\in\T_N$ and $N$.  
The log exponential moments under $\nu^N_{g}$ of both terms in the right-hand side of~\eqref{eq_def_R_+} are thus bounded by $O(N^{-1/2})$ in a neighbourhood of $0$ by~\eqref{eq_bound_exp_moments_sec_2}.
This concludes the estimate of the cost of replacing $W^+$ by $W^{+,aN}$: 
for some $\gamma,C>0$,
\begin{align}
\nu^N_{g}\big(f\big[W^+ - W^{+,aN}\big]\big)
&\leq 
\alpha N^2\nu^N_{g}\big(\Gamma_{\tilde h,\tilde\lambda}(\sqrt{f})\big)
+ \nu^N_g\big(f\mathcal R^+_\alpha\big) + \nu^N_{g}\big(f\mathcal R^+\big)
\nonumber\\
&\leq
\alpha N^2\nu^N_{g}\big(\Gamma_{\tilde h,\tilde\lambda}(\sqrt{f})\big)
+ \frac{{\bf H}^N_g(f)}{\gamma} + \frac{C}{N^{1/2}}
.
\end{align}
Similarly building $\mathcal R^{-},\mathcal R^-_\alpha$ and defining $\mathcal W_{g,\tilde h,\tilde \lambda,\alpha}$ as follows concludes the proof:
\begin{equation}
\mathcal W_{g, \tilde h,\tilde \lambda,\alpha}
:= 
W^{+,aN}+\mathcal R^+ + \mathcal R^+_\alpha +  W^{-,aN} + \mathcal R^- + \mathcal R^-_\alpha
+ W^=
.
\label{eq_def_W_g_h}
\end{equation}
\end{proof}
Lemma~\ref{lemm_derivative_entropy} provides the entropy bound in Proposition~\ref{prop_entropy_estimate_general}. 
We now prove the moment bound on the correlation field at each time.
\begin{coro}[Moment bound for the correlation field]\label{coro_moment_bound}
Let $\phi:\T^2\rightarrow\R$ be bounded. 
Then, for any $T\geq 0$:
\begin{equation}
\sup_N\sup_{t\leq T}\E^{\nu^N_{\tilde h}}_{\tilde h,\tilde\lambda}\Big[\big|\Pi^N_t(\phi)\big|^{3/2}\Big]
\leq 
C(T,\phi)
.
\end{equation}
\end{coro}
\begin{proof}
Let $t\geq 0$. 
Recall the identity $\E[X^{3/2}]\leq 1+ \frac{3}{2}\int_{1}^\infty r^{1/2}\Prob(X>r)\, dr$ valid for any random variable $X\geq 0$. 
This together with the bound $|\Pi^N(\phi)|\leq C(\phi)N$ for each $N$ gives:
\begin{align}
\E^{\nu^N_{\tilde h}}_{\tilde h,\tilde\lambda}&\Big[\big|\Pi^N_t(\phi)\big|^{3/2}\Big]
\leq 
\frac{3}{2}\int_0^{C(\phi)N}r^{1/2}\,  \Prob^{\nu^N_{\tilde h}}_{\tilde h,\tilde\lambda}\Big(\big|\Pi^N_t(\phi)\big|>r\Big)\,dr
.
\end{align}
Using the entropy inequality~\eqref{eq_entropy_inequality_sec3} with a constant $\gamma = cr$ for $c>0$ to be chosen later and the test function ${\bf 1}_{|\Pi^N(\phi)|>r}$ ($r>0$) yields:
\begin{align}
\E^{\nu^N_{\tilde h}}_{\tilde h,\tilde\lambda}&\Big[\big|\Pi^N_t(\phi)\big|^{3/2}\Big]
\leq 
\frac{3}{2}\int_0^{C(\phi)N} \frac{1}{c \sqrt{r}}\Big( {\bf H}^N(f_t)+ \log \nu^N_g\Big[e^{cr{\bf 1}_{|\Pi^N(\phi)|>r}}\Big]\Big) \,dr
\nonumber\\
&=
\frac{3}{2}\int_0^{C(\phi)N} \frac{1}{c \sqrt{r}}\Big( {\bf H}^N(f_t)+ \log \Big( 1+ (e^{cr}-1)\nu^N_{\tilde h}\big(|\Pi^N(\phi)|> r\big)\Big)\Big)\,dr
.
\end{align}
Proposition~\ref{prop_concentration_exponentielle} ensures that one can choose $c=c(\phi)>0$ such that:
\begin{equation}
(e^{cr}-1)\nu^N_{\tilde h}\big(|\Pi^N(\phi)|> r\big)
\leq 
e^{-cr/2},\qquad r>0
.
\end{equation}
Using the entropy bound of Proposition~\ref{prop_entropy_estimate_general}, 
the bound $\log(1+x)\leq x$ for $x\geq 0$ and the fact that $r\mapsto r^{-1/2}e^{-cr/2}$ is integrable on $\R_+$ concludes the proof.
\end{proof}
We now prove Corollary~\ref{coro_bound_error_terms} on time averages of error terms.
\begin{proof}[Proof of Corollary~\ref{coro_bound_error_terms}]
Only the case $\theta^N\neq \zeta^N_\alpha$ ($\alpha\in (0,1/2]$) requires a proof as explained below the statement of the corollary. 
This case corresponds to situations where a renormalisation procedure such as the one performed for $\e_{\tilde h}$ in the proof of Lemma~\ref{lemm_estimate_threepoint} is necessary. 
The proof below is basically the one in~\cite[Theorem 5.1]{Jara2018} adapted to get bounds on expectations of time-integrated observables rather than just tail probabilities. 
Recall the definition of the function $\mathcal R_{\tilde h,\tilde \lambda,\alpha}$ of Lemma~\ref{lemm_bound_adjoint}, 
let $a>1$ be a parameter that will eventually become large, 
define $\alpha_a := (4a)^{-1}$ and write:
\begin{align}
\Big|\int_0^T\theta^N(\eta_t)\, dt\Big| 
&\leq
\Big|\int_0^T\big[\theta^N(\eta_t) - \zeta^N_{\alpha_a}(\eta_t)\big]\, dt\Big| 
- \frac{1}{2a}\int_0^T \mathcal R_{\tilde h,\tilde\lambda}(\eta_t)\, dt 
\nonumber\\
&\quad 
+ \int_0^T\big|\zeta^N_{\alpha_a}(\eta_t)\big|\, dt
+ \frac{1}{2a}\int_0^T\mathcal R_{\tilde h,\tilde\lambda}(\eta_t)\, dt
. 
\label{eq_decomp_theta_Boltzmann_Gibbs}
\end{align}
By assumption on $\zeta^N_{\alpha_a}$ and by Lemma~\ref{lemm_bound_adjoint} for $\mathcal R_{\tilde h,\tilde\lambda}$, 
the second line is controlled by the entropy inequality:
\begin{equation}
\limsup_{N\rightarrow\infty}\E^{\nu^N_{\tilde h}}_{\tilde h,\tilde\lambda}\bigg[\int_0^T\big|\zeta^N_{\alpha_a}(\eta_t) + \mathcal R_{\tilde h,\tilde\lambda}(\eta_t)\big|\, dt\bigg]
=
0
.
\label{eq_bound_zeta_N_R_h}
\end{equation}
On the other hand, subtracting $\zeta^N_{\alpha_a}$ and $\mathcal R_{\tilde h,\tilde\lambda}$ gives us good control on the expectation of the first line in~\eqref{eq_decomp_theta_Boltzmann_Gibbs} as we show below:
\begin{align}
&\E^{\nu^N_{\tilde h}}_{\tilde h,\tilde\lambda}\bigg[\Big|\int_0^T\big[\theta^N(\eta_t)-\zeta^N_{\alpha_a}(\eta_t)\big]\, dt\Big| - \frac{1}{2a}\int_0^T\mathcal R_{\tilde h,\tilde\lambda}(\eta_t)\, dt\bigg]
\\
&\hspace{3cm}\leq 
\frac{1}{a}\log \E^{\nu^N_{\tilde h}}_{\tilde h,\tilde\lambda}\bigg[\exp\Big[ a\Big|\int_0^T\big[\theta^N(\eta_t)-\zeta^N_{\alpha_a}(\eta_t)\big]\, dt\Big| - \frac{1}{2}\int_0^T\mathcal R_{\tilde h,\tilde\lambda}(\eta_t)\, dt \Big]\bigg]
\nonumber
.
\end{align}
Inside the exponential, the absolute value can be removed using $e^{|x|}\leq e^x + e^{-x}$ ($x\in\R$). 
It is thus enough to separately bound, uniformly in $N,a$, 
the exponential moments of $\pm a [\theta^N-\zeta^N_{\alpha_a}]$. 
Feynman-Kac inequality (see Lemma~\ref{lemm_derivative_entropy}) and the bound of Lemma~\ref{lemm_bound_adjoint} on the adjoint give, e.g. for $+a[\theta^N-\zeta^N_{\alpha_a}]$:
\begin{align}
\log\, &\E^{\nu^N_{\tilde h}}_{\tilde h,\tilde\lambda}\bigg[\exp\Big[a\int_0^T\big[\theta^N(\eta_t)-\zeta^N_{\alpha_a}(\eta_t)\big]\, dt 
- \frac{1}{2}\int_0^T\mathcal R_{\tilde h,\tilde\lambda}(\eta_t)\, dt \Big]\bigg]
\nonumber\\
&\quad \leq
T\sup_{f\geq 0:\nu^N_{\tilde h}(f)=1}\Big\{ 
a\, \nu^N_{\tilde h}\big(f\big[\theta^N-\zeta^N_{\alpha_a}\big]\big)-\frac{1}{2}\nu^N_{\tilde h}\big(f \mathcal R_{\tilde h,\tilde\lambda}\big) 
\nonumber\\
&\hspace{5cm}
- \frac{N^2}{2}\nu^N_{\tilde h}\big(\Gamma_{\tilde h,\tilde\lambda}(\sqrt{f})\big) + \frac{N^2}{2}\nu^N_{\tilde h}\big(N^2L^*_{\tilde h,\tilde h,\tilde\lambda}{\bf 1}f\big)
\Big\}
\nonumber\\
&\quad
\leq
T\sup_{f\geq 0:\nu^N_{\tilde h}(f)=1}\Big\{ 
a\, \nu^N_{\tilde h}\big(f\big[\theta^N-\zeta^N_{\alpha_a}\big]\big)
- \frac{N^2}{4}\nu^N_{\tilde h}\big(\Gamma_{\tilde h,\tilde\lambda}(\sqrt{f})\big) \Big\}
.
\end{align}
Recalling the definition~\eqref{eq_def_zeta_N_corollary} of $\zeta^N_{\alpha_a}$ and the choice $\alpha_a = (4a)^{-1}$, 
the last supremum is bounded by $0$. 
This bound and~\eqref{eq_bound_zeta_N_R_h} conclude the proof.
\end{proof}
\section{A candidate for the approximate driven process}\label{sec_driven_process_candidate}
In this section, we start the proof of Theorem \ref{theo_main_result} by carrying out the first step in the program expounded in Section \ref{sec_structure_proof}. \\
This involves computing the Radon-Nikodym derivative $D^N_{\tilde h,\lambda,E}:= \mathrm{d}\Prob^N_E/\mathrm{d}\Prob^N_{\tilde h,\lambda+E}\big|_T$ up to time $T$ for $\lambda,E\in\R$ and some $\tilde h:\T^2\rightarrow\R$ 
(these two dynamics are defined in~\eqref{eq_def_generateur_SSEP}--\eqref{eq_def_jump_rates_h_and_lambda}),  
then finding conditions on $\tilde h$ that make $\Prob^N_{\tilde h,\lambda+E}$ a good candidate for the approximate driven process as defined in~\eqref{eq_def_driven_process}. 
This is carried out in Section \ref{sec_Radon_Nikodym_derivative}. \\
In Section \ref{sec_domination_current}, 
we prove domination results between the current-biased dynamics and the dynamics $\Prob^N_{\lambda+E}$, $\Prob^N_{h,\lambda+E}$, 
with $h$ the bias identified in Section~\ref{sec_Radon_Nikodym_derivative}. 
These domination results ensure that the current-biased dynamics and $\Prob^N_{h,\lambda+E}$ are indeed suitably close and reduce the proof of Theorem~\ref{theo_main_result} to long-time decorrelation estimates at the macroscopic level, 
established in Section \ref{sec_time_decorrelation}.
\subsection{The Radon-Nikodym derivatives}\label{sec_Radon_Nikodym_derivative}
Let $\tilde h\in C^0(\T)\cap C^\infty([0,1])$ be fixed throughout. 
The function $\tilde h$ is identified with $(x,y)\mapsto \tilde h(x-y)\in C^\infty_D(\T^2)$ 
so that $\Pi^N(\tilde h)$ is a meaningful object. 
We always assume that $\sigma\tilde h$ has leading eigenvalue strictly smaller than $1$, 
so that the results of Section~\ref{sec_fluc_correl_prob_h_lambda} on $\Prob^N_{\tilde h,\lambda+E}$ apply.

Here, we compute the Radon-Nikodym derivative $D^N_{\tilde h,\lambda+E} := \mathrm{d}\Prob^N/\mathrm{d}\Prob^N_{\tilde h,\lambda+E}\big|_T$ for trajectories up to time $T>0$, 
starting with $D^N_{\lambda+E} =: D^N_{0,\lambda+E}$. 
\begin{lemm}\label{lemm_RD_Prob_lambda}
Let $\mu^N$ be a probability measure on $\Omega_N$ ($N\in\N_{\geq 1}$). 
Let $\lambda,E\in\R$, $T>0$ and let $A_{\lambda,E}$ be given by:
\begin{equation}
A_{\lambda,E}(\eta)
= 
- \lambda(\lambda+2E)\sum_{i\in\T_N}\bar\eta_i\bar\eta_{i+1}
,\qquad 
\eta\in\Omega_N
.
\label{eq_def_A}
\end{equation}
There is then $\epsilon_{\lambda,E}^N:\Omega_N\rightarrow\R$ with $\sup_{\eta}|\epsilon_{\lambda,E}^N| \leq C(\lambda,E)/N$ for each $N\in\N_{\geq 1}$, 
such that:
\begin{equation}
\Prob^{curr,\mu^N}_{\lambda,E,T}(\cdot) 
= 
\frac{\E^{\mu^N}_{\lambda+E}\Big[{\bf 1}_{\cdot}\exp\Big(\int_0^T A_{\lambda,E}(\eta_t)\, dt + \int_0^T\epsilon^N_{\lambda,E}(\eta_t)\, dt\Big)\Big]}{\E^{\mu^N}_{\lambda+E}\Big[\exp\Big(\int_0^T A_{\lambda,E}(\eta_t)\, dt + \int_0^T\epsilon^N_{\lambda,E}(\eta_t)\, dt\Big)\Big]}
.
\label{eq_dyn_courant_en_terme_dyn_lambda}
\end{equation}
\end{lemm}
\begin{proof}
Let $T>0$. 
The two dynamics $\Prob^N_E$ and $\Prob^N_{\lambda+E}$ are absolutely continuous with respect to one another on the space of trajectories on $[0,T]$. 
Fix $\eta(\cdot)\in \mathcal D(\R_+,\Omega_N)$. 
We write $\eta_t$ rather than $\eta(t)$ for the configuration at time $t$ when there is no ambiguity with the occupation number $\eta_i$ at a site $i\in\T_N$. 
The Radon-Nikodym derivative $D_{\lambda,E}^{N}$ until time $T$ reads (see the proof of Proposition A.2.6 in~\cite{Kipnis1999}):
\begin{align}
D^{N}_{\lambda,E}\big((\eta_t)_{t\leq T}\big)
&=
\frac{\mathrm{d}\Prob^{N}_E}{\mathrm{d}\Prob^N_{\lambda+E}}\Big|_T\big((\eta_t)_{t\leq T}\big) 
\nonumber\\
&= \exp\Big[-\lambda NQ_T+ N^2\int_0^T\Big( \sum_{i\in\T_N}\eta_i(t)(1-\eta_{i+1}(t))e^{E/N}\big(e^{\lambda/N}-1\big) \nonumber\\
&\hspace{3cm}+ \eta_{i+1}(t)(1-\eta_i(t))e^{-E/N}\big(e^{-\lambda/N}-1\big)\Big)\, dt\Big]
.
\label{eq_der_RD_Prob_lambda_vs_Prob}
\end{align}
Developing the exponential in the second member, 
we find that there is $C>0$ and a function $\epsilon^N_{\lambda,E}:\Omega_N\rightarrow\R$ with $\sup_{\eta\in\Omega_N}|\epsilon^N_{\lambda,E}|\leq C/N$ such that:
\begin{align}
D^{N}_{\lambda,E}\big((\eta_t)_{t\leq T}\big)
&= 
\exp\Big[-\lambda NQ_T+ \frac{\lambda(\lambda+2E)}{2}\int_0^T \sum_{i\in\T_N}c(\eta_t,\eta_t^{i,i+1})\, dt+\int_0^t\epsilon^N_{\lambda,E}(\eta_t)\, dt\Big]
.
\label{eq_der_RD_Prob_lambda_Prob_h_dvlp}
\end{align}
The jump rates $c(\eta,\eta^{i,i+1})$ can be rewritten in terms of the variables $\bar\eta_i=\eta_i-1/2$ ($i\in\T_N$):
\begin{equation}
c(\eta,\eta^{i,i+1}) 
= 
2\sigma -2\bar\eta_i\bar\eta_{i+1}
.
\label{eq_dvplt_jump_rate_en_bar_eta}
\end{equation}
Note that the total number of particles $\sum_{i}\eta_i(t)$ is constant in time. 
The sum in \eqref{eq_der_RD_Prob_lambda_Prob_h_dvlp} thus reads:
\begin{align}
\int_0^T\sum_{i\in\T_N}c(\eta_t,\eta_t^{i,i+1})\, dt 
&=
2 NT\sigma 
-2\int_0^T\sum_{i\in\T_N} \bar\eta_i(t)\bar\eta_{i+1}(t)\, dt
.
\label{eq_expansion_c_as_sigma_plus_sigma'_integre_en_temps}
\end{align}
Injecting the expression~\eqref{eq_expansion_c_as_sigma_plus_sigma'_integre_en_temps} in both numerator and denominator in the definition~\eqref{eq_def_dynamique_conditionnee_courant} of $\Prob^{curr,\mu^N}_{\lambda,E,T}$ concludes the proof.
\end{proof}
Lemma~\ref{lemm_RD_Prob_lambda} expresses the current-biased dynamics $\Prob^{curr,\mu^N}_{\lambda,E,T}$ in terms of the dynamics $\Prob^{\mu^N}_{\lambda+E}$ which has the same macroscopic current. 
We now tune the long-range two-point correlation structure by considering the dynamics $\Prob^{\mu^N}_{\tilde h,\lambda+E}$ and optimising on $\tilde h$ so that this dynamics is, 
loosely speaking, as close to $\Prob^{curr,\mu^N}_{\lambda,E,T}$ as possible when $N$, then $T$ are large. 
%
\begin{prop}\label{prop_choice_of_h_sec4}
Let $\mu^N$ be a probability measure on $\Omega_N$ ($N\in\N_{\geq 1}$). 
Let $\lambda,E\in\R$ be sub-critical in the sense of Proposition~\ref{prop_solving_PDE_h} and $h=h_{\lambda,E}$ be the associated solution of~\eqref{eq_ODE_sur_h}.  
Then, for each $T>0$:
\begin{align}
\Prob^{curr,\mu^N}_{\lambda,E,T}(\cdot) 
= 
\frac{\E^{\mu^N}_{h,\lambda+E}\Big[{\bf 1}_{\cdot} \exp\Big(-\Pi^N_T(h)+\Pi^N_0(h) + \int_0^T \epsilon^N_{h,\lambda,E}(\eta_t)\, dt\Big)\Big]}
{\E^{\mu^N}_{h,\lambda+E}\Big[ \exp\Big(-\Pi^N_T(h)+\Pi^N_0(h) + \int_0^T \epsilon^N_{h,\lambda,E}(\eta_t)\, dt\Big)\Big]}
,
\label{eq_change_dynamics_lambda_to_h_lambda}
\end{align}
where $\epsilon^N_{h,\lambda,E}$ is a function that has small time integral under both $\Prob^N_{\lambda+E},\Prob^N_{h,\lambda+E}$ (started from suitable initial conditions):
\begin{align}
\forall\delta>0,\qquad 
&\limsup_{N\rightarrow\infty}\E^{\nu^N_{1/2}}_{\lambda+E}\Big[\Big|\int_0^T\epsilon^N_{h,\lambda,E}(\eta_t)\, dt\Big|\Big] 
= 
0,
\nonumber\\ 
&\limsup_{N\rightarrow\infty}\E^{\nu^N_h}_{h,\lambda+E}\Big[\Big|\int_0^T\epsilon^N_{h,\lambda,E}(\eta_t)\, dt\Big|\Big]
=
0
.
\label{eq_size_epsilon_N_h_sous_prob_lambda_et_prob_h}
\end{align}
\end{prop}
\begin{proof}
Let $T>0$. 
By Feynman-Kac formula, see Appendix A.7 in~\cite{Kipnis1999}, 
the Radon-Nikodym derivative $D^{N}_{\tilde h,\lambda+E} = \mathrm{d}\Prob^N_{\lambda+E}/\mathrm{d}\Prob^N_{\tilde h,\lambda+E}|_T$ up to time $T$ reads, 
on each trajectory $\eta_\cdot=(\eta_t)_{t\leq T}$:
\begin{equation}
\log D^N_{\tilde h,\lambda+E} (\eta_\cdot) 
= 
-\Pi^N_T(\tilde h)+ \Pi^N_0(\tilde h) + \int_0^Te^{-\Pi^N_t(\tilde h)}N^2L_{\lambda+E}\, e^{\Pi^N_t(\tilde h)}\, dt
.
\label{eq_der_RD_Prob_lambda_Prob_h}
\end{equation}
The computation of $e^{-\Pi^N(\tilde h)}N^2 L_{\lambda+E}\,  e^{\Pi^N(\tilde h)}$ is very similar to that of $N^2L^*_{h,h,\lambda+E}{\bf 1}$ 
performed in Section \ref{sec_entropy_estimate}, 
so we only give the result. 
For each $\eta\in\Omega_N$,
\begin{align}
&e^{-\Pi^ N(\tilde h)}N^2L_{\lambda+E}\, e^{\Pi^N(\tilde h)}(\eta) 
= 
\frac{1}{2}\sum_{i\in\T_N}\bar\eta_i\bar\eta_{i+1}\big(\partial^N_1 \tilde h_{i+1,i}-\partial^N_1 \tilde h_{i-1,i+1}\big)\label{eq_Neumann_choice_of_h}
\\
&\qquad + 
\frac{1}{2N}\sum_{j\neq \ell\in\T_N}\bar\eta_j\bar\eta_\ell \bigg[{\bf 1}_{|j-\ell|>1}\Delta^N_1 \tilde h_{j,\ell} + \frac{\sigma}{2N} \sum_{i\notin\{j-1,j,\ell-1,\ell\}} \partial^N_1 \tilde h_{i,j}\partial^N_1 \tilde h_{i,\ell}\bigg] 
\label{eq_EDP_sur_h_discrete}
\\
&\qquad + 
\frac{\sigma^2}{4N}\sum_{i\in\T_N}\sum_{j\notin\{i,i+1\}}\big(\partial^{N}_1 \tilde h_{i,j}\big)^2 
+ \delta^N_{\tilde h,\lambda+E}(\eta) + O(N^{-1})
,
\end{align}
with the error term $O(N^{-1})$ uniform on the configuration. 
The quantity $\delta^N_{\tilde h,\lambda+E}(\eta)$ reads:
\begin{align}
\delta^N_{\tilde h,\lambda+E}(\eta) 
&:= 
\frac{(\lambda+E) \sigma}{N}\sum_{i\in\T_N}\bar\eta_i N\big[\tilde h_{i-1,i} -\tilde h_{i+1,i}\big] 
\nonumber\\
&\quad 
-\frac{1}{4N^2}\sum_{i\in\T_N}\bar\eta_i\bar\eta_{i+1}\sum_{j,\ell\notin\{i,i+1\}}\bar\eta_j\bar\eta_\ell\partial^{N}_1\tilde h_{i,j}\partial^{N}_1 \tilde h_{i,\ell} 
\nonumber\\
&\quad
- \frac{(\lambda+E)}{N}\sum_{i\in\T_N}\bar\eta_i\bar\eta_{i+1}\sum_{j\notin\{i,i+1\}}\bar\eta_j\partial^N_1\tilde h_{i,j} 
.
\label{eq_def_error_term_RD}
\end{align}
Let us show that $\delta^N_{\tilde h,\lambda+E}$ satisfies~\eqref{eq_size_epsilon_N_h_sous_prob_lambda_et_prob_h}. 
The term $\delta^N_{\tilde h,\lambda+E}$ is similar to~\eqref{eq_def_epsilon_N} estimated in Section~\ref{sec_entropy_estimate}. 
There we showed that, for any $\alpha\in(0,1/2]$ and 
any $g\in C^\infty_D(\T^2)$ with $\sigma g$ having leading eigenvalue strictly below $1$, 
there is a function $\zeta^N_{g,\tilde h,\lambda+E,\alpha}$ such that, 
for any density $f$ for $\nu^N_{g}$:
\begin{align}
\nu^N_{g}\big(f\delta^N_{\tilde h,\lambda+E}\big)
&\leq 
\alpha N^2 \nu^N_{g}\big(\Gamma_{\tilde h,\tilde \lambda}(\sqrt{f})\big)
+ \nu^N_g\big(f\zeta^N_{g,\tilde h,\lambda+E,\alpha}\big)
\end{align}
with, for some $\gamma_\alpha,C_\alpha>0$:
\begin{equation}
\frac{1}{\gamma_\alpha}\log \nu^N_g\Big(\exp\big[\gamma_\alpha\big|\zeta^N_{g,\tilde h,\lambda+E,\alpha}\big|\big]\Big)
\leq 
\frac{C_\alpha}{N^{1/2}}
,\qquad 
N\in\N_{\geq 1}
.
\end{equation}
Taking $g=0$ and $g=\tilde h$
this implies~\eqref{eq_size_epsilon_N_h_sous_prob_lambda_et_prob_h} for $\Prob^{\nu^N_{1/2}}_{\lambda+E},\Prob^{\nu^N_{\tilde h}}_{\tilde h,\lambda+E}$ respectively by Corollary~\ref{coro_bound_error_terms}.

We now choose $\tilde h$. 
The bracket in \eqref{eq_EDP_sur_h_discrete} is a discretised differential operator acting on $\tilde h$, 
while the first line \eqref{eq_Neumann_choice_of_h} plays the role of a condition on the normal derivative of $\tilde h$ on the diagonal $\{(x,x):x\in\T\}$ of $\T^2$ 
(recall that, as $\tilde h$ is symmetric, 
$\partial_1 \tilde h(x,y) = \partial_2\tilde h(y,x)$ for each $x\neq y\in\T^2$), 
see~\eqref{eq_normal_derivative_sec4} below. 
To determine the continuous analogue of the discrete differential operator, 
let us replace discrete derivatives by their continuous counterparts using the following error estimates 
that follow from $\tilde h\in C^\infty_D(\T^2)$:
\begin{equation}
\sup_N\sup_{|i-j|>1} N\big|\Delta^N_1 \tilde h_{i,j}- \Delta_1\tilde h_{i,j}\big|<\infty,
\qquad
\sup_N\sup_{\substack{i,j \in \T_N \\ j\notin\{i,i+1\}}} N\big|\partial^{N}_1 \tilde h_{i,j}- \partial_1\tilde h_{i,j}\big|
<
\infty
.
\end{equation}
Let $\epsilon^N_{\tilde h,\lambda,E}$ be given by:
\begin{align}
\epsilon^N_{\tilde h,\lambda,E}
&=
\delta^N_{\tilde h,\lambda,E} 
+ \epsilon^N_{disc,\tilde h}
,
\label{eq_def_epsilon_h_lambda_E}
\end{align}
with $\epsilon^N_{disc,\tilde h}$ the discretisation error from switching to continuous partial derivatives:
\begin{align}
&\epsilon^N_{disc,\tilde h}(\eta)
=
\frac{1}{2N}\sum_{i\in\T_N} \bar\eta_i\bar\eta_{i+1}N \big[(\partial^N_1\tilde h_{i+1,i}-\partial^N_1\tilde h_{i-1,i+1}) - (\partial_1\tilde h_{i_+,i} - \partial_1\tilde h_{i_-,i})\big]
\nonumber\\
&\qquad 
+\frac{1}{2N^2}\sum_{j,\ell:|j-\ell|>1}\bar\eta_j\bar\eta_\ell N\big[\Delta^N_1\tilde h_{j,\ell} - \Delta_1\tilde h_{j,\ell}\big]
\label{eq_def_epsilon_disc}\\
&\qquad 
+ \frac{\sigma}{4N^2}\sum_{j\neq \ell}\bar\eta_j\bar\eta_\ell N\Big[\frac{1}{N}\sum_{i\notin\{j-1,j,\ell-1,\ell\}}\partial^{N}_1\tilde h_{i,j}\partial^{N}_1 \tilde h_{i,\ell} - \int_{\T}\partial_1\tilde h(x-j/N)\partial_1 \tilde h(x-\ell/N)\, dx\Big]
.
\nonumber
\end{align}
The exponential moment under $\nu^N_{\tilde h},\nu^N_{1/2}$ of each term above is controlled through Proposition~\ref{prop_concentration_exponentielle}. 
Corollary~\ref{coro_bound_error_terms} thus shows that $\epsilon^N_{\tilde h,\lambda,E}$, 
also satisfies~\eqref{eq_size_epsilon_N_h_sous_prob_lambda_et_prob_h} 
(replacing $h$ there by $\tilde h$) and: 
\begin{align}
&e^{-\Pi^ N(\tilde h)}N^2L_{\lambda+E}\, e^{\Pi^N(\tilde h)}(\eta) 
= 
\frac{\sigma^2}{4N}\sum_{i\in\T_N}\sum_{j\notin\{i,i+1\}}\big(\partial^N_1 \tilde h_{i,j}\big)^2 + \frac{1}{2}\sum_{i\in\T_N}\bar\eta_i\bar\eta_{i+1}\big(\partial_1 \tilde h_{i_+,i}-\partial^N_1 \tilde h_{i_-,i}\big) 
\nonumber\\
&\hspace{2cm}
+ 2\Pi^N\Big(\Delta_1 \tilde h + (x,y)\mapsto \frac{\sigma}{2} \int_{\T}\partial_1 \tilde h(z,x)\partial_1 \tilde h(z,y)\, dz\Big)
+ \epsilon^N_{\tilde h,\lambda,E}(\eta)
.
\label{eq_error_term_last_line_der_RD_before_epsilon_N}
\end{align}
Our goal is to now choose $\tilde h$ such that the two-point correlation term in~\eqref{eq_error_term_last_line_der_RD_before_epsilon_N} 
precisely cancels the term $A_{\lambda,E}=-\lambda(\lambda+2E)\sum_i\bar\eta_i\bar\eta_{i+1}$ appearing in Lemma~\ref{lemm_RD_Prob_lambda}.  
This fixes the partial derivative of $\tilde h$ across the diagonal:
\begin{equation}
\forall i\in\T_N,\qquad
\partial_1 \tilde h_{i_+,i}-\partial_1 \tilde h_{i_-,i}
=
h'(0_+)-h'(1_-)
=
2\lambda(\lambda+2E)
.
\label{eq_normal_derivative_sec4}
\end{equation}
The remaining two-point correlation terms in~\eqref{eq_error_term_last_line_der_RD_before_epsilon_N} also need to vanish for the approximation of the current-biased dynamics by $\Prob^N_{\tilde h,\lambda+E}$ to be better than that by $\Prob^N_{\lambda+E}$. 
Thus:
\begin{equation}
\Delta_1 \tilde h + (x,y)\mapsto \frac{\sigma}{2} \int_{\T}\partial_1 \tilde h(z,x)\partial_1 \tilde h(z,y)
=
0
,\qquad 
x\neq y\in\T
.
\label{eq_PDE_tildeh}
\end{equation}
As $\tilde h(x,y)=\tilde h(x-y)$ by assumption, 
$\partial_1\tilde h(z,x) = \tilde h'(z-x)=-\tilde h'(x-z)$ ($x\neq z$) and we see that~\eqref{eq_PDE_tildeh} is precisely the ordinary differential equation appearing in Proposition~\ref{prop_solving_PDE_h}:
\begin{equation}
\tilde h''(x) - \sigma\int_{\T}\tilde h'(x-y)\tilde h'(y)\, dy
=0
,\qquad 
x\in(0,1)
.
\end{equation}
From this observation and~\eqref{eq_normal_derivative_sec4} it follows that $\tilde h$ is the function $h$ of Proposition~\ref{prop_solving_PDE_h}. 
Note that $\sigma h$ indeed has leading eigenvalue strictly below $1$ as $\lambda,E$ are sub-critical by assumption. 
As the last term in the first line of~
\eqref{eq_error_term_last_line_der_RD_before_epsilon_N} is configuration-independent, 
the proof of Proposition~\ref{prop_choice_of_h_sec4} is concluded. 
\end{proof}
\subsection{A key domination result}\label{sec_domination_current}
The expression~\eqref{eq_change_dynamics_lambda_to_h_lambda} involves the exponential of terms that we know to have average bounded with $N$ under $\Prob^N_{h,\lambda+E}$ when $N$, then $T$ are large. 
Since these terms appear inside exponentials, 
we need to make sure that their exponential moments are also bounded uniformly in $N,T$. 
Unbounded exponential moments would mean that that $\Prob^N_{h,\lambda+E}$ is in fact not a good candidate for an approximate driven process. 
Proposition~\ref{prop_prob_curr_as_prob_driven_with_bounded_terms} shows that the exponential moments are indeed bounded, 
at least for parameters $\lambda,E$ in a small enough subset of the sub-critical region of Proposition~\ref{prop_solving_PDE_h}. 
This corresponds to Step 2 in the program outlined in Section~\ref{sec_structure_proof}.

The main ingredient to establish Proposition~\ref{prop_prob_curr_as_prob_driven_with_bounded_terms} is the following domination result. 
It is the only result of the present paper for which we need to restrict the range of sub-critical $\lambda,E$, 
see the discussion in Section~\ref{sec_perspectives}. 
\begin{prop}\label{prop_domination_by_P_h}
Let $\lambda,E\in\R$ be sub-critical parameters as in Proposition~\ref{prop_solving_PDE_h} and let $T>0$.  
If $\lambda(\lambda+2E)\geq 0$ and is small enough (independently of $T$), 
there are constants $\kappa\in(0,1),C>0$ depending only on $h,\lambda,E$ such that, 
for any sequence $(\mathcal O_N)_N$ of events involving the dynamics up to time $T$:
\begin{equation}
\limsup_{N\rightarrow\infty}\Prob^{curr,\nu^N_{1/2}}_{\lambda,T}(\mathcal O_N) 
\leq 
C \limsup_{N\rightarrow\infty}\Prob_{h,\lambda,E}^{\nu^N_h}\big(\mathcal O_N\big)^{\kappa}
.
\end{equation}
\end{prop}
Proposition~\ref{prop_domination_by_P_h} allows for an expression of the current-biased dynamics in terms of expectations of bounded observables under $\Prob^N_{h,\lambda+E}$ as follows. 
Let $M>0$ and define $F_M$ by:
\begin{equation}
F_M(x) = \begin{cases}
e^x &\text{if }x\leq M,\\
0 &\text{if }x\geq 2M,\\
\text{smooth and bounded by }e^{M} &\text{otherwise},
\end{cases}
\qquad 
x\in\R
.
\label{eq_def_F_M}
\end{equation}
\begin{prop}\label{prop_prob_curr_as_prob_driven_with_bounded_terms}
Let $\lambda,E\in\R$ be sub-critical so that $\sigma h$ has leading eigenvalue strictly below $1$. 
Assume that $\lambda,E$ are such that the domination bound of Proposition~\ref{prop_domination_by_P_h} holds with exponent $\kappa>0$.

There are then $C,C'>0$ and $\kappa'\in(0,1)$ depending only on $h,\lambda,E$ such that,
for each $T>0$, 
each $M>C'$ and each continuous and bounded $G: \mathcal D(\R_+,\s'(\T))\rightarrow\R$:
\begin{align}
\limsup_{N\rightarrow\infty}
\bigg|\E^{curr,\nu^N_{1/2}}_{\lambda,E,T}\Big[ G\big((Y^N_s)_{s\leq T}\big)\Big]
&- \frac{\E^{\nu^N_h}_{h,\lambda+E}\Big[G\big((Y^N_s)_{s\leq T}\big)F_M\big(-\Pi^N_0(h)\big)F_M\big(-\Pi^N_T(h)\big)\Big]}{\E^{\nu^N_h}_{h,\lambda+E}\Big[F_M\big(-\Pi^N_0(h)\big)F_M\big(-\Pi^N_T(h)\big)\Big]}
\bigg| 
\nonumber\\
&\hspace{5cm}
\leq 
\frac{C\|G\|_\infty}{M^{\kappa'}}
.
\label{eq_prob_curr_as_prob_driven_with_bounded_terms}
\end{align}
\end{prop}
\subsubsection{Proof of Proposition~\ref{prop_prob_curr_as_prob_driven_with_bounded_terms} assuming Proposition~\ref{prop_domination_by_P_h}}
\label{sec_proof_comparison_assuming_domination}
\begin{proof}
Let $T,M,\delta>0$, $t\geq 0$, and let  
$G$ be as in Proposition~\ref{prop_prob_curr_as_prob_driven_with_bounded_terms}. 
Introduce the sets:
\begin{equation}
S_t(M) 
:= 
\Big\{ \big|\Pi^N_t(h)\big|\leq M\Big\},
\qquad
U_T(\delta) 
:=
\Big\{ \Big|\int_0^T \epsilon^N_{h,\lambda,E}(\eta_t)\, dt\Big|\leq \delta\Big\}
,
\label{eq_def_ensemble_error_term_small}
\end{equation}
with $\epsilon^N_{h,\lambda,E}$ given in Proposition~\ref{prop_choice_of_h_sec4}. 
Theorem~\ref{theo_size_entropy}, 
the entropy inequality and the concentration bounds of Proposition~\ref{prop_concentration_exponentielle} give a control on $S_t(M)^c$: 
for some $\gamma,C>0$ depending only on $h,\lambda,E$,
\begin{align}
\Prob^{\nu^N_h}_{h,\lambda+E}\big(\big|\Pi^N_t(h)\big|\geq M\big) 
&\leq 
\frac{1}{M}\E^{\nu^N_h}_{h,\lambda+E}\big[\big|\Pi^N_t(h)\big|\big]
\leq
\frac{{\bf H}^N(f_t)}{M\gamma} +\frac{1}{M\gamma}\log \nu^N_h\Big(e^{\gamma^{-1}|\Pi^N(h)|}\Big)
\nonumber\\
&\leq
\frac{Ce^{Ct}}{N^{1/2}} + \frac{C}{M}
.
\end{align}
Proposition~\ref{prop_choice_of_h_sec4} shows that $\epsilon^N_{h,\lambda,E}$ has small time integral, 
thus for a different constant $C>0$:
\begin{equation}
\limsup_{N\rightarrow\infty}\Prob^{\nu^N_h}_{h,\lambda+E}\Big((V_T(M,\delta))^c\Big) 
\leq 
\frac{C}{M},
\quad 
V_T(M,\delta) 
:= 
S_0(M)\cap S_T(M)\cap U_T(\delta)
.
\label{eq_decay_complement_W_T}
\end{equation}
The domination bound of Proposition~\ref{prop_domination_by_P_h} then implies:
\begin{equation}
\limsup_{N\rightarrow\infty}
\Big|\E^{curr,\nu^N_{1/2}}_{\lambda,E,T}\big[G\big] - \E^{curr,\nu^N_{1/2}}_{\lambda,E,T}\big[G \big| V_T(M,\delta)\big]\Big| 
\leq
\frac{C\|G\|_\infty }{M^{\kappa}}
.
\label{eq_current_biased_as_current_biased_conditioned} 
\end{equation}
It is therefore enough to prove Proposition~\ref{prop_prob_curr_as_prob_driven_with_bounded_terms} for the current-biased dynamics conditioned on $V_T(M,\delta)$. 
The proof at this point is elementary but tedious. 
For short, write:
\begin{equation}
\mathcal R^N_T
=
\mathcal R^N_T \big((\eta_t)_{t\leq T}\big)
:=
-\Pi^N_T(h)-\Pi^N_0(h) + \int_0^T\epsilon^N_{h,\lambda,E}(\eta_t)\, dt
.
\end{equation}
From the expression~\eqref{eq_change_dynamics_lambda_to_h_lambda} of the current-biased dynamics in terms of $\Prob^N_{h,\lambda+E}$, 
we can write:
\begin{equation}
\E^{curr,\nu^N_{1/2}}_{\lambda,E,T}\big[G \big| V_T(M,\delta)\big]
=
\Big(\E^{\nu^N_h}_{h,\lambda+E}\big[{\bf 1}_{ V_T(M,\delta)}e^{\mathcal R^N_T}\big]\Big)^{-1}\E^{\nu^N_h}_{h,\lambda+E}\big[G {\bf 1}_{ V_T(M,\delta)}e^{\mathcal R^N_T}\big]
.
\end{equation}
Recall definition~\eqref{eq_def_ensemble_error_term_small} of $U_T(\delta)\subset V_T(M,\delta)$. 
On $U_T(\delta)$, the time integral of $\epsilon^N_{h,\lambda,E}$ is bounded by $\delta$.  
Moreover, on $S_t(M)$, 
$F_M(-\Pi^N_t(h)) = e^{-\Pi^N_t(h)}$, 
$t\in\{0,T\}$. 
Rewriting also the exponential containing $\epsilon^N_{h,\lambda,E}$ using $e^x = 1+ e^x-1$, 
we find after elementary computations:
\begin{align}
\Bigg|\E^{curr,\nu^N_{1/2}}_{\lambda,E,T}\big[G\big| V_T(M,\delta)\big] 
- 
\frac{\E^{\nu^N_h}_{h,\lambda+E}\Big[ G{\bf 1}_{V_T(M,\delta)}F_M\big(-\Pi^N_0(h)\big)F_M\big(-\Pi^N_T(h)\big)\Big]}{\E^{\nu^N_h}_{h,\lambda+E}\Big[{\bf 1}_{V_T(M,\delta)}F_M\big(-\Pi^N_0(h)\big)F_M\big(-\Pi^N_T(h)\big)\Big]}\Bigg| 
\leq 
\|G\|_\infty c(\delta)
,
\label{eq_proof_prop45_0}
\end{align}
with $c(\delta) = 2(e^\delta-1)$. 
To conclude the proof, 
it remains to remove the indicator function in the last equation. 
We start with the numerator, the denominator being similar.

Recall from Proposition~\ref{prop_solving_PDE_h} that $h$ is either a negative kernel when $\lambda(\lambda+2E)\geq 0$
or a positive kernel if $\lambda(\lambda+2E)\leq 0$. 
Recall also that we always work with sub-critical $\lambda,E$ or equivalently that the leading eigenvalue of $\sigma h$ strictly below $1$ (see Proposition~\ref{prop_solving_PDE_h}), 
say smaller than $1-r$ for $r=r(\lambda,E)>0$. 
Recall finally that the domination bound of Proposition~\ref{prop_domination_by_P_h} is assumed to hold for $\lambda,E$.   
Although this domination bound is established only for $\lambda(\lambda+2E)\geq 0$ corresponding to a negative kernel $h$, 
we conjecture it should hold throughout the sub-critical regime (see the discussion in Section~\ref{sec_perspectives}) and therefore carry out the proof of~\eqref{prop_prob_curr_as_prob_driven_with_bounded_terms} also in the case $\lambda(\lambda+2E)<0$.

Assume first that $h$ is a positive kernel. 
This implies (recall $\bar\eta_i^2=\sigma=1/4$ for each $i\in\T_N$):
\begin{equation}
-\Pi^N(h) 
=
-\frac{1}{4N}\sum_{i\neq j}\bar\eta_i\bar\eta_j h_{i-j}
=
-\frac{1}{4N}\sum_{i,j}\bar\eta_i\bar\eta_j h_{i-j}
+\frac{h_0\sigma}{4}
\leq 
\frac{h_0\sigma}{4}
\leq
\frac{\|h\|_\infty}{16}
.
\end{equation}
From $F_M(x)\leq e^x$ for each $x\in\R$, we get:
\begin{equation}
F_M(-\Pi_T(h))
\leq 
e^{-\Pi_T(h)}
\leq 
e^{\|h\|_\infty/16}
.
\label{eq_bound_Pi_T}
\end{equation}
The same holds for $\Pi^N_0$. Thus:
\begin{align}
&\bigg|\E^{\nu^N_h}_{h,\lambda+E}\Big[ G{\bf 1}_{V_T(M,\delta)}F_M\big(-\Pi^N_0(h)\big)F_M\big(-\Pi^N_T(h)\big)\Big]
\nonumber\\
&\hspace{2cm}
-\E^{\nu^N_h}_{h,\lambda+E}\Big[ GF_M\big(-\Pi^N_0(h)\big)F_M\big(-\Pi^N_T(h)\big)\Big]\bigg|
\nonumber\\
&\hspace{2.7cm}=
\Big|\E^{\nu^N_h}_{h,\lambda+E}\Big[ G{\bf 1}_{V_T(M,\delta)^c}F_M\big(-\Pi^N_0(h)\big)F_M\big(-\Pi^N_T(h)\big)\Big]\Big|
\nonumber\\
&\hspace{2.7cm}\leq
\|G\|_\infty e^{\|h\|_\infty/8} 
\Prob^{\nu^N_h}_{h,\lambda+E}\big(V_T(M,\delta)^c\big)
\leq
\frac{C(h,\lambda,E)\|G\|_\infty(1+o_N(1))}{M} 
.
\end{align}
Assume now that $h$ is a negative kernel. 
By the explicit formula for $h$ in Proposition~\ref{prop_solving_PDE_h}, 
we know that $-\sigma h$ has eigenvalues bounded from above. 
Let $r_h$ be the spectral radius of $-\sigma h$. 
Hölder inequality with exponents $1+(2r_h)^{-1}, 2r_h+1$ and Cauchy-Schwarz inequality give:
\begin{align}
&\E^{\nu^N_h}_{h,\lambda+E}\Big[ G{\bf 1}_{V_T(M,\delta)^c}F_M\big(-\Pi^N_0(h)\big)F_M\big(-\Pi^N_T(h)\big)\Big]
\label{eq_bound_with_spectral_radius}\\
&\
\leq 
\|G\|_\infty  \E^{\nu^N_h}_{h,\lambda+E}\Big[F_M\big(-\Pi^N_0(h)\big)^{1+\frac{1}{2r_h}}F_M\big(-\Pi^N_T(h)\big)^{1+\frac{1}{2r_h}}\Big]^{\frac{2r_h}{2r_h+1}}
\Prob^{\nu^N_h}_{h,\lambda+E}\big(V_T(M,\delta)^c\big)^{\frac{1}{2r_h+1}}
\nonumber\\
&\
\leq 
\frac{C(h,\lambda,E)\|G\|_\infty}{M^{\frac{1}{2r_h+1}}} 
\E^{\nu^N_h}_{h,\lambda+E}\Big[F_M\big(-\Pi^N_0(h)\big)^{2+\frac{1}{r_h}}\Big]^{\frac{r_h}{2r_h+1}}
\E^{\nu^N_h}_{h,\lambda+E}\Big[F_M\big(-\Pi^N_T(h)\big)^{2+\frac{1}{r_h}}\Big]^{\frac{r_h}{2r_h+1}}
.
\nonumber
\end{align}
We claim that the last two expectations are bounded independently of $N,M,T$. 
Indeed, as $F_M(x)\leq e^x$, the exponential moment involving $\Pi_0(h)$ satisfies:
\begin{align}
\E^{\nu^N_h}_{h,\lambda+E}\Big[F_M\big(-\Pi^N_0(h)\big)^{2+\frac{1}{r_h}}\Big]
&\leq 
\nu^N_h\Big[\exp\Big[-\Big(2+\frac{1}{r_h}\Big)\Pi^N_0(h)\Big]\Big]
\nonumber\\
&\leq 
\frac{1}{\mathcal Z^N_h}\nu^N_{1/2}\Big[\exp\Big[-\frac{1}{r_h}\Pi^N_0(h)\Big]\Big]
.
\label{eq_bound_exp_moment_Pi0}
\end{align}
By definition of $r_h$, 
$-(\sigma/2r_h) h$ has leading eigenvalue at most $1/2$. 
By Lemma~\ref{lemm_bound_correlations}, 
the exponential moment in~\eqref{eq_bound_exp_moment_Pi0} is therefore bounded with $N$. Lemma~\ref{lemm_bound_correlations} also gives $\mathcal Z^N_h\geq 1$. 

Consider now $\Pi^N_T$. 
As $F_M$ is a bounded function, 
Pinsker's inequality in Proposition~\ref{prop_entropy_estimate_general} gives:
\begin{equation}
\limsup_{N\rightarrow\infty}\Big|\E^{\nu^N_h}_{h,\lambda+E}\Big[F_M\big(-\Pi^N_T(h)\big)^{2+\frac{1}{2r_h}}\Big] -\nu^N_h\Big[F_M\big(-\Pi^N(h)\big)^{2+\frac{1}{2r_h}}\Big]\Big|
=
0
.
\label{eq_bound_exp_moment_PiT}
\end{equation}
The bound is then the same as in~\eqref{eq_bound_exp_moment_Pi0}. 

Consider now the denominator. 
Recalling that $F_M(-\Pi_t(h)) = e^{-\Pi_t(h)}$ on $S_t(M)$, 
it is enough to prove that, for $M$ larger than some $C(h,\lambda,E)>0$:
\begin{equation}
\inf_{M>C(h,\lambda,E)}\inf_{T>0}\liminf_{N\to\infty}\E^{\nu^N_h}_{h,\lambda+E}\Big[{\bf 1}_{V_T(M,\delta)}\exp\big[-\Pi^N_0(h)-\Pi^N_T(h)\big]\Big]
>
0
.
\end{equation}
Write ${\bf 1}_{V_T(\delta)} = 1-{\bf 1}_{V_T(\delta)^c}$. 
The moment bound in Proposition~\ref{prop_entropy_estimate_general}, 
Jensen inequality and 
the convergence/stationarity results of Theorem~\ref{theo_fluct_correl_driven_process} give the existence of $c>0$ independent of $N,T,M$ such that:
\begin{align}
&\liminf_{N\to\infty}\E^{\nu^N_h}_{h,\lambda+E}\Big[\exp\big[-\Pi^N_0(h)-\Pi^N_T(h)\big]\Big] 
\nonumber\\
&\hspace{4cm}\geq 
2\liminf_{N\to\infty}\exp\Big[-\E^{\nu^N_h}_{h,\lambda+E}\big[\Pi^N_0(h)+\Pi^N_T(h)\big]\Bigg]
\geq 
c>0
.
\label{eq_lower_bound_exp_moments_Pi0_Pi_T}
\end{align}
On the other hand, we know by~\eqref{eq_current_biased_as_current_biased_conditioned} and~\eqref{eq_proof_prop45_0} that: 
\begin{align}
&\sup_{T>0}\limsup_{N\to\infty}\E^{\nu^N_h}_{h,\lambda+E}\Big[{\bf 1}_{V_T(M,\delta)^c}\exp\big[-\Pi^N_0(h)-\Pi^N_T(h)\big]\Big]
\nonumber\\
&\hspace{5cm}
\leq 
\frac{C}{M^{\frac{1}{2r_h+1}}}{\bf 1}_{\lambda(\lambda+2E)>0}  +\frac{C}{M}{\bf 1}_{\lambda(\lambda+2E)\leq 0}
.
\end{align}
Putting together this bound and~\eqref{eq_current_biased_as_current_biased_conditioned} and~\eqref{eq_proof_prop45_0}, 
we conclude that for all $M$ larger than some $C=C(h,\lambda,E)>0$:
\begin{align}
&\limsup_{N\rightarrow\infty}\Bigg|\E^{curr,\nu^N_{1/2}}_{\lambda,E,T}\big[G\big] 
- 
\frac{\E^{\nu^N_h}_{h,\lambda+E}\Big[ GF_M\big(-\Pi^N_0(h)\big)F_M\big(-\Pi^N_T(h)\big)\Big]}{\E^{\nu^N_h}_{h,\lambda+E}\Big[F_M\big(-\Pi^N_0(h)\big)F_M\big(-\Pi^N_T(h)\big)\Big]}\Bigg| 
\nonumber\\
&\hspace{3.5cm}
\leq 
\|G\|_\infty c(\delta)
+\frac{C\|G\|_\infty}{M^{\kappa'}}
,
\end{align}
where $\kappa' = \min\{(2r_h+1)^{-1},\kappa\}$ if $h$ is a negative kernel,  
$\kappa'=\min\{1,\kappa\}$ if $h$ is a positive kernel. 
Since $c(\delta) = 2(e^{\delta}-1)$, 
taking $\delta= M^{-\kappa'}$ concludes the proof of Proposition~\ref{prop_prob_curr_as_prob_driven_with_bounded_terms} assuming the domination bound of Proposition~\ref{prop_domination_by_P_h}.
\end{proof}
\subsubsection{Proof of Proposition~\ref{prop_domination_by_P_h}: the domination bound}\label{sec_proof_domination}
The time-independent domination bound of Proposition~\ref{prop_domination_by_P_h} is proven by first establishing 
a weaker domination result, Proposition~\ref{prop_domination_by_P_lambda}, 
in terms of the dynamics $\Prob^N_{\lambda+E}$ that has the correct macroscopic current but the wrong correlation structure. 
This first domination bound provides a control of the current-biased dynamics to order $1$ in $N$, 
but not yet uniform in time, 
reflecting the fact that the correlation structure is not yet properly tuned. 
\begin{prop}\label{prop_domination_by_P_lambda}
Let $\lambda,E\in\R$ be sub-critical so that $\sigma h$ has leading eigenvalue strictly below $1$ (recall Proposition~\ref{prop_solving_PDE_h}). 
Recall the definition of $A_{\lambda,E}$ in~\eqref{eq_def_A}.

Then, for each $s>0$, if $\lambda(\lambda+2E)\geq0$ and is small enough depending on $s$, 
there is $C(s,\lambda,E)>0$ satisfying:
\begin{equation}
\forall T>0,\qquad 
\sup_{N\geq 1}\E^{\nu^N_{1/2}}_{\lambda+E}\Big[\exp\Big[s\int_0^TA_{\lambda,E}(\eta_t)\, dt\Big]\Big] 
\leq 
e^{C(s,\lambda,E)T}
.
\label{eq_bound_normalisation_P_lambda}
\end{equation}
This implies that there is $C(h,\lambda,E)>0$ such that, 
for any event $\mathcal O$ depending on the dynamics up to time $T$:
\begin{equation}
\Prob^{curr,\nu^N_{1/2}}_{\lambda,E,T}(\mathcal O) 
\leq 
e^{C(h,\lambda,E)T}\Prob_{\lambda+E}^{\nu^N_{1/2}}(\mathcal O)^{1/2}
,\qquad
N\in\N_{\geq 1}
.
\label{eq_domination_par_P_lambda}
\end{equation}
\end{prop}
\begin{proof}
Let us first prove~\eqref{eq_domination_par_P_lambda} assuming \eqref{eq_bound_normalisation_P_lambda}. 
Fix $\mathcal O$ as in the proposition. 
Lemma~\ref{lemm_RD_Prob_lambda} relates the current-biased dynamics starting from $\nu^N_{1/2}$ to $\Prob^{\nu^N_{1/2}}_{\lambda+E}$:
\begin{equation}
\Prob^{curr,\nu^N_{1/2}}_{\lambda,E,T}(\mathcal O) 
=
\frac{\E^{\nu^N_{1/2}}_{\lambda+E}\Big[{\bf 1}_{\mathcal O}\exp\Big[\int_0^TA_{\lambda,E}(\eta_t)\, dt\Big]\Big]}{\E^{\nu^N_{1/2}}_{\lambda+E}\Big[\exp\Big[\int_0^TA_{\lambda,E}(\eta_t)\, dt\Big]\Big]}
.
\end{equation}
Applying Cauchy-Schwarz inequality to the numerator and the bound~\eqref{eq_bound_normalisation_P_lambda}, 
then Jensen inequality to the denominator, 
one finds:
\begin{align}
\Prob^{curr,\nu^N_{1/2}}_{\lambda,E,T}(\mathcal O) 
\leq 
\frac{e^{C(s=2,\lambda,E)T/2}\, \nu_{1/2}^N\big(\exp[2\Pi^N(h)]\big)^{1/2}}{\exp\E^{\nu^N_{1/2}}_{\lambda+E}\Big[\int_0^TA_{\lambda,E}(\eta_t)\, dt\Big]}\Prob_{\lambda+E}^{N,\nu^N_{1/2}}(\mathcal O)^{1/2}
.
\label{eq_proof_domination_0}
\end{align}
Since $\sigma h$ has leading eigenvalue strictly below $1$, 
Lemma~\ref{lemm_bound_correlations} gives:
\begin{equation}
\sup_N\nu^N_{1/2}(\exp[2\Pi^N(h)])
<
\infty
.
\end{equation}
The numerator of~\eqref{eq_proof_domination_0} thus has the desired form~\eqref{eq_domination_par_P_lambda}. 
Moreover, 
recalling definition~\eqref{eq_def_A} of $A_{\lambda,E}$ 
and the fact that $\nu^N_{1/2}$ is invariant for $\Prob^N_{\lambda+E}$ and product,
\begin{equation}
\E^{\nu^N_{1/2}}_{\lambda+E}\Big[\int_0^TA_{\lambda,E}(\eta_t)\, dt\Big] 
= 
T\nu^N_{1/2}\big[A_{\lambda,E}\big]
=0
.
\end{equation}
This concludes the proof of \eqref{eq_domination_par_P_lambda} assuming \eqref{eq_bound_normalisation_P_lambda}, which we now prove. 
For each $0\leq m\leq N$, let $U_m$ denote the uniform measure on the set $\Omega_{N,m}$ of configurations with $m$ particles. 
One can check that the measure $U_m$ is invariant for the dynamics $\Prob^N_{\lambda+E}$. 
Let us rewrite $A_{\lambda,E}$, defined in~\eqref{eq_def_A}, 
in terms of quantities with vanishing average under $U_m$. 
Define, for each $m$ and $\eta\in\Omega_{N,m}$:
\begin{equation}
A^m_{\lambda,E}(\eta)
:=
-\lambda(\lambda+2E)\sum_{i\in\T_N}\bar\eta_i^m\bar\eta^m_{i+1},
\qquad
\bar\eta^m_\cdot 
= 
\eta_\cdot - \frac{m}{N}
.
\end{equation}
Then, on $\Omega_{N,m}$:
\begin{equation}
A_{\lambda,E}
=
A^m_{\lambda,E}
-\lambda(\lambda+2E)N\Big(\frac{m}{N}-\frac{1}{2}\Big)^2
\leq 
A^m_{\lambda,E}
,
\end{equation}
where the inequality follows from the assumption $\lambda(\lambda+2E)\geq 0$. 
As a result, the exponential moment to estimate satisfies, for each $s>0$:
\begin{align}
\E^{\nu^N_{1/2}}_{\lambda+E}\Big(\exp\Big[s\int_0^TA_{\lambda,E}(\eta_t)\, dt\Big]\Big) 
&\leq 
\max_{0\leq m\leq N}\E^{U_m}_{\lambda+E}\Big(\exp\Big[s\int_0^TA^m_{\lambda,E}(\eta_t)\, dt\Big]\Big) 
.
\end{align}
Fix $0\leq m\leq N$ henceforth, and let us estimate the above exponential moment. 
Let $\Gamma_{\lambda+E}$ denote the carré du champ operator associated with $L_{\lambda+E}$:
\begin{equation}
\Gamma_{\lambda+E}(g)(\eta)
:=
\frac{1}{2}\sum_{i\in\T_N} c_{\lambda+E}(\eta,\eta^{i,i+1})\big[u(\eta^{i,i+1})-u(\eta)\big]^2,
\qquad \eta\in\Omega_N,\quad
u:\Omega_N\rightarrow\R
.
\end{equation}
Feynman-Kac inequality (see Lemma~\ref{lemm_derivative_entropy}) and the invariance of $U_m$ give:
\begin{align}
\frac{1}{T}&\log\E^{U_m}_{\lambda+E}\Big[\exp\Big[s\int_0^TA^m_{\lambda,E}(\eta_t)\, dt\Big]\Big] 
\nonumber\\
&\qquad \leq 
\sup_{f\geq 0:U_m(f)=1}\Big\{U_m\big(fsA_{\lambda,E}\big) - \frac{N^2}{2}\, U_m\big(\Gamma_{\lambda+E}(\sqrt{f})\big)\Big\}
.
\label{eq_bound_normalisation_E_lambda_0}
\end{align}
An elementary computation shows that, for each density $f$ for $U_m$:
\begin{equation}
U_m\big(\Gamma_{\lambda+E}(\sqrt{f})\big) 
= 
\cosh\big((\lambda+E)/N\big) D^m_{ex}(\sqrt{f}),
\qquad 
D^m_{ex}(\sqrt{f}) 
:= 
- U_m(\sqrt{f}L_0\sqrt{f})
.
\end{equation}
The supremum in \eqref{eq_bound_normalisation_E_lambda_0} therefore reduces to:
\begin{align}
\sup_{f\geq 0:U_m(f)=1}\Big\{\nu^N_{1/2}\big(f s A^m_{\lambda,E}\big) - \frac{N^2}{2}\cosh\big((\lambda+E)/N\big)D^m_{ex}(\sqrt{f}) \Big\}.\label{eq_bound_normalisation_E_lambda_1}
\end{align}
Fix a density $f\geq 0$ for $U_m$. 
We aim to prove that this supremum is bounded by some $C(s,\lambda,E)>0$ provided $\lambda(\lambda+2E)\geq 0$ is small enough. 
To do so,  
we smooth out $A^m_{\lambda,E}$ using the carre du champ in~\eqref{eq_bound_normalisation_E_lambda_1} as was done in the proof of Lemma~\ref{lemm_estimate_threepoint}. 
Notice first that an integration by parts gives, for each $\eta\in\Omega_N$ and $i\in\T_N$:
\begin{equation}
\bar\eta^m_{i+1}
=
\frac{1}{N-1}\sum_{j\neq i}\bar\eta^m_{j}
+
\sum_{j\neq i,i-1} (\bar\eta^m_{j}-\bar\eta^m_{j+1})\phi_{N-1}(j-i)
,
\label{eq_IPP_sec3}
\end{equation}
with:
\begin{equation}
\phi_{N-1}(a) 
:= 
\frac{N-1-(a\mod N)}{N-1},
\qquad 
a\in\Z
.
\end{equation}
Fix a density $f$ for $\nu^N_{1/2}$. 
Using~\eqref{eq_IPP_sec3} on $A^m_{\lambda,E}$ for each $i\in\T_N$ and the integration by parts Lemma~\ref{lemm_IPP} 
(which applies to $U_m$ by Remark~\ref{rmk_IPP_uniform_measure}) 
with $\alpha=1/2$, $\tilde h=0$ and $\tilde\lambda=\lambda+E$, 
we find that there is a numerical constant $c>0$ such that:
\begin{align}
U_m\big(fsA^m_{\lambda,E}\big) 
&\leq 
\frac{N^2}{4}D^m_{ex}(\sqrt{f}) 
- \frac{s\lambda(\lambda+2E)}{N-1} U_m\Big(f\sum_{i\neq j}\bar\eta^m_i\bar\eta^m_j\Big)
\nonumber\\
&\quad
+ \frac{cs^2\lambda^2(\lambda+2E)^2}{N}U_m\bigg[f\sum_{j\in\T_N} \Big(\frac{1}{\sqrt{N}}\sum_{i\neq j}\phi_{N-1}(j-i)\bar\eta^m_i\Big)^2\bigg]
.
\label{eq_to_prove_control_A_lambda_line_1}
\end{align}
Adding and removing the (deterministic) $i=j$ term in the first sum and using $\lambda(\lambda+2E)\geq 0$, we get:
\begin{equation}
- \frac{s\lambda(\lambda+2E)}{N-1} U_m\Big(f\sum_{i\neq j}\bar\eta^m_i\bar\eta^m_j\Big)
\leq 
s\lambda(\lambda+2E)\frac{N}{N-1} \frac{m}{N}\Big(1-\frac{m}{N}\Big)
.
\end{equation}
Since $\rho(1-\rho)\leq 1/4$ for any $\rho\in[0,1]$,  
the bound~\eqref{eq_bound_normalisation_P_lambda} will be proven, 
recalling~\eqref{eq_bound_normalisation_E_lambda_1}, 
if we can show that, for $\lambda(\lambda+2E)$ small enough:
\begin{equation}
\frac{cs^2\lambda^2(\lambda+2E)^2}{N}U_m\bigg(f\sum_{j\in\T_N} \Big(\frac{1}{\sqrt{N}}\sum_{i\neq j}\phi_{N-1}(j-i)\bar\eta^m_i\Big)^2\bigg)
\leq 
C(s,\lambda,E) + \frac{N^2}{4}D^m_{ex}(\sqrt{f})
,
\label{eq_to_prove_control_A_lambda_line_2}
\end{equation}
where $C(s,\lambda,E)>0$ is independent of $f,m,N$. 

To prove~\eqref{eq_to_prove_control_A_lambda_line_2}, 
we use the entropy inequality for each $j\in\T_N$ to get, 
for any $\gamma>0$:
\begin{align}
&\frac{cs^2\lambda^2(\lambda+2E)^2}{N}U_m\bigg(f\sum_{j\in\T_N} \Big(\frac{1}{\sqrt{N}}\sum_{i\neq j}\phi_{N-1}(j-i)\bar\eta^m_i\Big)^2\bigg)
\nonumber\\
&\quad \leq 
\frac{U_m(f\log f)}{\gamma} 
\nonumber\\
&\qquad 
+\frac{1}{\gamma N}\sum_{j\in\T_N}\log U_m\bigg(\exp\Big(\gamma cs^2\lambda^2(\lambda+2E)^2\Big(\frac{1}{\sqrt{N}}\sum_{i\neq j}\phi_{N-1}(j-i)\bar\eta^m_i\Big)^2\Big]\bigg)
.
\end{align}
The entropic term can be estimated with the log-Sobolev inequality of Lemma~\ref{lemm_LSI_SSEP}: 
for $\gamma$ larger than some $\gamma_0>0$ independent of $m$,
\begin{equation}
\frac{U_m(f\log f)}{\gamma}
\leq 
\frac{N^2}{4} D^m_{ex}(\sqrt{f})
.
\label{eq_use_LSI_sec3}
\end{equation}
On the other hand, the function $\phi_{N-1}(j-\cdot)$ is Lipschitz uniformly in $j\in\T_N$, 
in the sense:
\begin{equation}
\sup_{N\geq 2}\max_{i,j\in\T_N}N\big|\phi_{N-1}(j-i-1)-\phi_{N-1}(j-i)\big|
=
2
.
\end{equation}
Lemma~\ref{lemm_exp_mom_uniform} then implies that, 
for $\lambda(\lambda+2E)$ smaller than a constant depending only on $s,\gamma$:
\begin{equation}
\frac{1}{\gamma N}\sum_{j\in\T_N}\log U_m\bigg[\exp\Big(\gamma cs^2\lambda^2(\lambda+2E)^2\Big(\frac{1}{\sqrt{N}}\sum_{i\neq j}\phi_{N-1}(j-i)\bar\eta^m_i\Big)^2\Big]\bigg)
\leq 
\frac{3}{\gamma}
.
\end{equation}
Putting the last estimate together with~\eqref{eq_to_prove_control_A_lambda_line_1}--\eqref{eq_to_prove_control_A_lambda_line_2}--\eqref{eq_use_LSI_sec3} inside~\eqref{eq_bound_normalisation_E_lambda_1} concludes the proof of Proposition~\ref{prop_domination_by_P_lambda}.
\end{proof}
We can now use Proposition~\ref{prop_domination_by_P_lambda}, 
together with the finely tuned correlation structure of $\Prob^N_{h,\lambda+E}$ to prove the time-independent domination bound of Proposition~\ref{prop_domination_by_P_h}. 
\begin{proof}[Proof of Proposition~\ref{prop_domination_by_P_h}]
Let $T>0$ and $(\mathcal O_N)_N$ be a sequence of measurable sets involving the dynamics up to time $T$ only. 
Recall the expression of $\Prob^{curr,\nu^N_{1/2}}_{\lambda,E,h}$ in terms of $\Prob^{\nu^N_h}_{h,\lambda+E}$ from Proposition \ref{prop_choice_of_h_sec4}:
\begin{align}
\Prob^{curr,\nu^N_{1/2}}_{\lambda,E,h}\big(\mathcal O_N\big) 
= 
\frac{\E^{\nu^N_h}_{h,\lambda+E}\Big[{\bf 1}_{\mathcal O_N}e^{-\Pi^N_T(h)-\Pi^N_0(h) + \int_0^T\epsilon^N_{h,\lambda,E}(\eta_t)\, dt}\Big]}{\E^{\nu^N_h}_{h,\lambda+E}\Big[e^{-\Pi^N_T(h)-\Pi^N_0(h)+ \int_0^T\epsilon^N_{h,\lambda,E}(\eta_t)\,dt}\Big]} 
.
\label{eq_interm_proof_prop_real_domination}
\end{align}
The idea is to use the (time-dependent) domination bound of Proposition~\ref{prop_domination_by_P_lambda} to reduce $\mathcal O_N$ to a set $\mathcal O_N\cap V^N_T(\delta)$ 
(see~\eqref{eq_def_V_N_T}) on which the exponential moment in the numerator in~\eqref{eq_interm_proof_prop_real_domination} cannot be too large when $N\gg 1$. 
This will allow us to use entropy estimate of Theorem~\ref{theo_size_entropy} to bound these exponential moments uniformly in time and $N$. 
The arguments are similar to those used to prove Proposition~\ref{prop_prob_curr_as_prob_driven_with_bounded_terms}.

Recall from~\eqref{eq_def_ensemble_error_term_small} the definition of the sets $S_t(M),U_T(\delta)$ for $M,\delta>0$ and $t\geq 0$:
\begin{equation}
S_t(M) 
= 
\big\{|\Pi_t(h)|\leq M\big\},
\qquad 
U_T(\delta)  
= 
\Big\{ \Big|\int_0^T\epsilon^N_{h,\lambda,E}(\eta_t)\, dt\Big|\leq \delta\Big\}
.
\end{equation}
Choosing $M=\log\log N$, define:
\begin{equation}
V^N_T(\delta) 
= 
S_T(\log\log N)\cap U_T(\delta)
.
\label{eq_def_V_N_T}
\end{equation}
Proposition~\ref{prop_domination_by_P_lambda} shows that $V^N_T(\delta)^c$ is unlikely under the current biased dynamics:
\begin{equation}
\limsup_{N\rightarrow\infty}
\Prob^{\nu^N_{1/2}}_{\lambda+E}\big(V^N_T(\delta)^c\big)
= 0
\quad \Rightarrow\quad 
\limsup_{N\rightarrow\infty}\Prob^{curr,\nu^N_{1/2}}_{\lambda,E,T}\big(V^N_T(\delta)^c\big)
=
0
.
\end{equation}
It is therefore enough to prove the domination bound for $(\mathcal O_N\cap V^N_T(\delta))_N$ rather than $(\mathcal O_N)_N$. 

On $V^N_T(\delta)$, we can estimate the exponential moments of $\Pi^N_T(h),\epsilon^N_{h,\lambda,E}$ in the numerator as we show below.  
Let us bound numerator and denominator in~\eqref{eq_interm_proof_prop_real_domination} separately. 
For the denominator, Jensen inequality gives the lower bound:
\begin{equation}
\exp\E^{\nu^N_h}_{h,\lambda+E}\Big[-\Pi^N_T(h)-\Pi^N_0(h) + \int_0^T\epsilon^N_{h,\lambda,E}(\eta_t)\, dt\Big]
.
\end{equation}
The expectation of $\Pi^N_0(h),\Pi^N_T(h)$ is bounded uniformly in $N,T$ as in~\eqref{eq_lower_bound_exp_moments_Pi0_Pi_T}. 
Together with the estimate of $\epsilon^N_{h,\lambda,E}$ in Proposition~\ref{prop_choice_of_h_sec4}, 
we obtain the following bound on the denominator of \eqref{eq_interm_proof_prop_real_domination}:
\begin{equation}
\inf_{T>0}\liminf_{N\to\infty}\E^{\nu^N_h}_{h,\lambda+E}\Big[e^{-\Pi^N_T(h)-\Pi^N_0(h)+ \int_0^T\epsilon^N_{h,\lambda,E}(\eta_t)\,dt}\Big]
>
0
.
\label{eq_final_bound_denom_prop_44}
\end{equation}
Consider now the numerator in~\eqref{eq_interm_proof_prop_real_domination} for the set $\mathcal O_N\cap V^N_T(\delta)$. 
On $V^N_T(\delta)$, the time integral of $\epsilon^N_{h,\lambda,E}$ is bounded by $\delta$, 
so only the exponential moments of $-\Pi^N_0(h),-\Pi^N_T(h)$ remain.  
As in~\eqref{eq_bound_with_spectral_radius}, 
one can use Hölder inequality to estimate exponential moments of $-\Pi^N_0(h),-\Pi^N_T(h)$ separately. 
The key observation is that, on the set $S_T(\log \log N)$, 
$\exp[-\Pi_T(h)]$ is bounded by $\log N=o(N^{1/4})$, 
therefore one can still use Pinsker's inequality in Proposition~\ref{prop_entropy_estimate_general} to reduce the computation of this exponential moment to an estimate under $\nu^N_h$. 
This is done exactly as in~\eqref{eq_bound_with_spectral_radius} and following equations, 
so we conclude the proof of Proposition~\ref{prop_domination_by_P_h} as in~\eqref{eq_bound_exp_moment_Pi0}--\eqref{eq_bound_exp_moment_PiT}:
\begin{align}
\limsup_{N\to\infty}\Prob^{curr,\nu^N_{1/2}}_{\lambda,E,h}\big(\mathcal O_N
\big)
&=
\limsup_{N\to\infty}\Prob^{curr,\nu^N_{1/2}}_{\lambda,E,h}\big(\mathcal O_N\cap V^N_T(\delta)\big)
\nonumber\\
&\leq 
C(h,\lambda,E)\limsup_{N\to\infty}\Prob^{\nu^N_h}_{h,\lambda+E}\big(\mathcal O_N\big)^{\kappa}
,
\end{align}
with $\kappa = (2r_h+1)^{-1}$ and $r_h$ the spectral radius of $-\sigma h$ ($r_h=0$ if $h$ is a positive kernel).
\end{proof}
\section{Time decorrelation and fluctuations for the current-biased dynamics}\label{sec_time_decorrelation}
In this section, 
we carry out the second part of Step $3$ as outlined in Section~\ref{sec_structure_proof}, 
establishing time decorrelation estimates. 
This concludes the proof of the main result,~Theorem \ref{theo_main_result}. 

Throughout, let $\lambda,E$ be sub-critical in the sense of Proposition~\ref{prop_solving_PDE_h} and let $h=h_{\lambda,E}$ be the associated solution of \eqref{eq_ODE_sur_h}.\\
Let $\tau\geq 0$ and let $G: \mathcal D(\R_+,\s'(\T))\rightarrow\R$ be a bounded continuous function. 
We need to prove:
\begin{equation}
\lim_{t\rightarrow\infty}\lim_{T\rightarrow\infty}\lim_{N\rightarrow\infty} 
\int G\big((Y^N_{s})_{t\leq s\leq t+\tau}\big)\, \mathrm{d}\Prob^{curr,\nu^N_{1/2}}_{\lambda,E,T} 
=  
{\bf E}_{h}^{\nu^\infty_{h}}\big[G\big((Y_{s})_{0\leq s\leq \tau}\big)\big]
,
\label{eq_main_result}
\end{equation}
where $\nu^\infty_{h}$ is the centred Gaussian field on $\s'(\T)$ with covariance $C_h = (\sigma^{-1}\text{id}-h)^{-1}$ ($\sigma:=1/4$)
and ${\bf E}^{\nu^\infty_{h}}_{h}$ denotes the expectation associated with the couple $(Y_\cdot,\Pi_\cdot)$ of limiting fields as given by Theorem~\ref{theo_fluct_correl_driven_process}. 

The proof of~\eqref{eq_main_result} is carried out in several steps. 
We first start from the expression of the current-biased dynamics obtained in Proposition~\ref{prop_prob_curr_as_prob_driven_with_bounded_terms} to take the large $N$ limit and reduce to quantities involving only the limiting processes $Y_\cdot,\Pi_\cdot$; this is done below.  
We then study the large time behaviour of the limiting fluctuation process. 
This is done through a classical approximation in terms of finite-dimensional diffusions following~\cite{Holley1978}, 
in Section~\ref{sec_finite_dimensional_approx}--\ref{sec_convergence_to_Y}, 
together with a large time analysis of these approximations in Section~\ref{sec_long_time_asymptotics}. \\

Let $M>0$. 
Since $F_M$ (recall~\eqref{eq_def_F_M}) is continuous and bounded, 
Proposition~\ref{prop_prob_curr_as_prob_driven_with_bounded_terms} and
Theorem~\ref{theo_fluct_correl_driven_process} give:
\begin{equation}
\limsup_{N\rightarrow\infty}
\Bigg|
\E^{curr,\nu^N_{1/2}}_{\lambda,E,T}\big[G\big]
- 
\frac{{\bf E}^{\nu^\infty_h}_{h}\Big[GF_M\big(-\Pi_T(h)\big)F_M\big(-\Pi_0(h)\big)\Big]}{{\bf E}^{\nu^\infty_h}_{h}\Big[F_M\big(-\Pi_T(h)\big)F_M\big(-\Pi_0(h)\big)\Big]}\Bigg| 
\leq 
\frac{C(h)\|G\|_\infty}{M^{\kappa'}}
,
\label{eq_start_to_prove_main_result}
\end{equation}
It is convenient to rewrite the ratio of expectations in~\eqref{eq_start_to_prove_main_result} in terms of the fluctuation field only. 
For any $\delta>0$, 
let $h_\delta$ denote the truncation of the Fourier series of $h$ given in Proposition~\ref{prop_solving_PDE_h} at some $n_\delta\in\N$ satisfying:
\begin{equation}
\|h_\delta - h\|_\infty
\leq \delta
,\qquad 
h_\delta 
:=
\sum_{j=0}^{n_\delta}  \big<e_j,h\big>e_j
.
\end{equation}
As usual, we identify $h_\delta$ with the function $(x,y)\in\T^2\mapsto h_\delta(x-y)$. 
Equation~\eqref{eq_continuity_Pi} in Proposition~\ref{prop_lien_Y_Pi} quantifies the error made when replacing $h$ by $h_\delta$ in~\eqref{eq_start_to_prove_main_result}: 
writing $1 = {\bf 1}_{B} + {\bf 1}_{B^c}$ with $B^c=\{|\Pi_0(h-h_\delta)|\leq \delta^{1/2},|\Pi_T(h-h_\delta)|\leq \delta^{1/2}\}$ and using that $F_M$ is Lipschitz,
\begin{align}
\bigg|&{\bf E}^{\nu^\infty_{h}}_{h}\Big[GF_M\big(-\Pi_T(h)\big)F_M\big(-\Pi_0(h)\big)\Big] 
- 
{\bf E}^{\nu^\infty_{h}}_{h}\Big[GF_M\big(-\Pi_T(h_\delta)\big)F_M\big(-\Pi_0(h_\delta)\big)\Big]
\bigg|
\nonumber\\
&\qquad \leq 
2\|G\|_\infty\|F_M\|^2_\infty {\bf P}^{\nu^\infty_{h}}_{h}\Big(|\Pi_0(h-h_\delta)|\geq \delta^{1/2}\text{ or }|\Pi_T(h-h_\delta)|\geq \delta^{1/2}\Big)
\nonumber\\
&\qquad \quad 
+ 2\delta^{1/2}\|G\|_\infty\|F_M\|_\infty\|F'_M\|_\infty
\nonumber\\
&\quad \ \,
\overset{\eqref{eq_continuity_Pi}}{\leq} 
C(\|G\|_\infty,M)\delta^{1/2}
.
\end{align}
The function $F_M(-\Pi_T(h_\delta))$ is now a function of $Y_T$ in view of the link~\eqref{eq_lien_Pi_et_Y_dans_theo} between $\Pi_T,Y_T$. 
Since the right-hand side above does not depend on time, 
vanishes when $\delta$ is small and $\delta$ does not depend on $M$, 
the proof of the main result, 
Theorem~\ref{theo_main_result}, 
is reduced to the proof of the following proposition.
\begin{prop}\label{prop_to_prove_main_result}
Let $G: \mathcal D(\R_+,\s'(\T))\rightarrow \R$ be a continuous and bounded function, 
let $\tau\geq 0$, 
and let $H_1,H_2 : \s'(\T)\rightarrow\R$ be continuous and bounded. 
Then:
\begin{align}
\lim_{t\rightarrow\infty}\lim_{T\rightarrow\infty}
\frac{{\bf E}_{h}^{\nu^\infty_{h}}\big[G\big((Y_{s})_{t\leq s\leq t+\tau}\big)H_1(Y_0)H_2(Y_T)\big]}{{\bf E}_{h}^{\nu^\infty_{h}}\big[H_1(Y_0)H_2(Y_T)\big]} 
= 
{\bf E}_{h}^{\nu^\infty_{h}}\big[G\big((Y_{s})_{0\leq s\leq \tau}\big)\big]
.
\end{align}
The result also holds if $t,\tau$ are diverging functions of $T$ provided $T-\tau$ diverges with $T$.
\end{prop}
In Proposition~\ref{prop_to_prove_main_result}, 
the microscopic model does not appear any more. 
All that is required is an analysis of the long time behaviour of the limiting fluctuation process $(Y_t)_{t\geq0}$, 
carried out in the next sections via an approximation procedure.  
\subsection{The finite-dimensional approximation}\label{sec_finite_dimensional_approx}
In this section, we follow~\cite{Holley1978} to construct, 
for each $n\in\N_{\geq 1}$, 
a diffusion process $Y^{(n)}_\cdot\in C(\R_+,\s'(\T))$. 
This process is finite-dimensional in a sense made clear below, relaxes to its invariant measure exponentially fast, 
and $(Y^{(n)}_\cdot)_n$ will be shown to converge weakly to $Y_\cdot$ in the next sections. \\

Recall from \eqref{eq_eigenbasis_laplacien} the definition of the eigenvectors $(e_\ell)_{\ell\in\N}$ of the Laplacian on the torus, 
and notice that the fact that $Y_t$ is a bounded linear form for each $t\geq 0$ implies:
\begin{equation}
\forall t\geq 0,\forall f\in C^\infty(\T),\qquad 
Y_t(f) 
= 
\sum_{\ell\in\N} \big<e_\ell,f\big>Y_t(e_\ell)
.
\end{equation}
To construct an approximation $Y^{(n)}_\cdot$ ($n\in\N_{\geq 1}$), 
let us make some preliminary remarks and introduce notations. 
Recall the definition of the self-adjoint operator $\mathcal L_{h}$ appearing in the martingale problem of Theorem \ref{theo_fluct_correl_driven_process} when $\tilde h$ there is equal to $h$.

For $n\in\N_{\geq 1}$, let $\pi_n$ be the projection operator on $\text{span}(e_1,...e_n)$. 
In the following, if $A$ is a bounded linear operator on $C^\infty(\T)$, its projection $\pi_n A\pi_n$ is assimilated when useful with the matrix $(\big<e_k,Ae_\ell\big>)_{1\leq k,\ell\leq n}$. 
Recall that the covariance operator $C_h$ defining the centred Gaussian measure $\nu^\infty_h$ is given by:
\begin{equation}
C_h 
=
(\sigma^{-1}\text{id}-h)^{-1}
,\qquad
\sigma
=
1/4
.
\end{equation}
For each $n\in\N_{\geq 1}$, let $C_n := \pi_n C_h\pi_n$ and let $\nu^{(n)}$ be the centred normal distribution on $\R^n$ with covariance matrix $C_n$:
\begin{equation}
\nu^{(n)}(dx) 
= 
p^{(n)}(x)\, dx,\qquad 
p^{(n)}(x) 
= 
\frac{1}{(2\pi)^{n/2}\det(C_n)^{1/2}}\exp\Big[-\frac{x^TC_n^{-1}x}{2}\Big],
\quad x\in\R^n
.
\label{eq_def_nu_(n)}
\end{equation}
Let us now define the approximation process $(Y^{(n)}_t)_{t\geq 0}$. 
First, by the martingale characterisation of fluctuations in Theorem~\ref{theo_fluct_correl_driven_process}, 
observe that $Y_t(1) = Y_0(1)$ for all $t\geq 0$ with probability $1$, with $1$ denoting here the function constant equal to $1$. 
We therefore only need to approximate $Y_\cdot$ for functions with vanishing average.

For $Y^{(n)}_\cdot$ to be a good approximation of $Y_\cdot$, 
it will have to satisfy the same martingale problem in the large $n$ limit. 
We build $Y^{(n)}_\cdot$ by appropriate finite-dimensional truncations of the drift and diffusion operator of $Y_\cdot$. 
It turns out to be convenient to do so by approximating the drift, on the one hand, 
and the covariance $C_h$ of the invariant measure on the other hand. 
The diffusion coefficient is then fixed by the other two.
Introduce thus the diffusion $y^{(n)}_\cdot$ on $\R^n$
given by:
\begin{equation}
dy^{(n)}_t
=
B_ny^{(n)}_tdt + \sqrt{2}D_n^{1/2}dW^{(n)}_t
,
\end{equation}
with $(W^{(n)}_t)_{t\geq 0}$ a standard Brownian motion on $\R^n$, 
and where $B_n$, $D_n$ respectively approximate the drift and diffusion part:
\begin{align}
B_n 
&:=
\pi_n \mathcal L_{h}\pi_n 
=
\Big(\big<e_k,\mathcal L_{h}e_\ell\big>\Big)_{1\leq k,\ell\leq n}
,
\nonumber\\
\qquad
D_n 
&:= 
-\frac{1}{2}\big[B_nC_n +C_n B_n^T\big]
=
-\frac{1}{2}\big[B_nC_n +C_n B_n\big]
.
\label{eq_def_D_n_B_n_appendix}
\end{align}
Note that $B_n,C_n,D_n$ all commute since $h$ commutes with the Laplacian ($h$ is translation invariant, recall Proposition~\ref{prop_solving_PDE_h}). 
Note also that $y^n_\cdot$ is a reversible diffusion, 
admitting $\nu^{(n)}$ as its unique invariant measure as can be checked from direct computations. 
The process $Y^{(n)}_\cdot$ ($n\in\N_{\geq 1}$)is then defined as follows:
\begin{equation}
Y^{(n)}_t(f) 
= 
\big<1,f\big>Y_0(1) + \sum_{j=1}^n \big<e_j,f\big>y^{(n)}_t(j)
,\qquad
f\in C^\infty(\T),t\geq 0
.
\label{eq_def_Y_(n)_appendix}
\end{equation}
Let ${\bf P}^{(n)},{\bf E}^{(n)}$ denote the probability/expectation associated with $Y^{(n)}_\cdot$. 
Notice that there is a continuous function $\Phi^{(n)}$ such that $Y^{(n)}_\cdot = \Phi^{(n)}(y^{(n)}_\cdot)$. 
Therefore, abusing notations, we will also write ${\bf P}^{(n)},{\bf E}^{(n)}$ for the probability and expectation associated with $y^{(n)}_\cdot$.  
Unless an initial condition is specified, the processes are assumed to start from the invariant measure $\nu^{(n)}$ of $y^{(n)}_\cdot$ (or the corresponding measure on distributions for $Y^{(n)}_\cdot$). 
\subsection{Convergence to $Y_\cdot$}\label{sec_convergence_to_Y}
Here we prove:
\begin{prop}\label{prop_convergence_approximation}
Assume $y_0^{(n)}$ has law $\nu^{(n)}$ and is independent from $Y_0(1)$. 
Then $Y^{(n)}_\cdot$ converges weakly to $Y_\cdot$ in $C(\R_+,\s'(\T))$.
\end{prop}
For deterministic initial conditions, 
the convergence in Proposition~\ref{prop_convergence_approximation} is proven in~\cite{Holley1978} 
(on $\R^d$, adapting the proof to the torus $\T$ is straightforward) as stated next.
\begin{prop}{\cite[Theorem (1.23)]{Holley1978}.}
Assume $Y^{(n)}_0$ starts at a deterministic $X^{(n)}_0\in \s'(\T)$ 
with $\lim_n X^{(n)}_0 = X_0\in\s'(\T)$.  
Then $Y^{(n)}_\cdot$ converges weakly in $C(\R_+,\s'(\T))$ to the process $Y^{X_0}_\cdot$ starting at $X_0$ and satisfying the same martingale problem as $Y_\cdot$, 
written out in Theorem~\ref{theo_fluct_correl_driven_process}.  
\end{prop}
To prove Proposition~\ref{prop_convergence_approximation}, 
it is therefore enough to prove that, 
if $y^{(n)}$ is distributed according to $\nu^{(n)}$ and independent from $Y_0(1)$, 
then $Y^{n}_0$ converges weakly to the centred Gaussian field $\nu^\infty_{h}$ with covariance $C_h=(\sigma^{-1}\text{id}-h)^{-1}$. 
This is an elementary Gaussian computation, given next. 
Let $n\in\N$. By definition, 
\begin{equation}
Y^{(n)}_0(f) 
= 
\big<1,f\big>Y_0(1) + \sum_{j=1}^n\big<e_j,f\big>y^{(n)}_0(k)
,\qquad
f\in C^\infty(\T)
.
\end{equation}
with $y^{(n)}_0$ a normal distribution with covariance $C_n$. As $Y_0(1)$ and $y^{(n)}_0$ are independent, 
$Y^{(n)}_0$ is a centred Gaussian field with covariance $\sigma\text{id} + C_n$. 
To prove that $Y_0^{(n)}$ converges to the Gaussian field $Y_0$ with covariance $C_h$, 
it is enough to check convergence of the covariances. 
Take thus $f_1,f_2 \in C^\infty(\T)$, and write, as $y^{(n)}_0$ has vanishing average: 
\begin{align}
{\bf E}^{(n)}\Big[Y_0^{(n)}(f_1)Y_0^{(n)}(f_2)\Big] 
&= 
\sigma\big<1,f_1\big>\big<1,f_2\big>+\sum_{j,\ell=1}^n\big<e_j,f_1\big>\big<e_\ell,f_2\big>\big<e_j,C_he_\ell\big> 
\nonumber\\
&\underset{n\rightarrow\infty}{\longrightarrow} 
\sigma\big<1,f_1\big>\big<1,f_2\big>+ \sum_{j,\ell=1}^\infty\big<e_j,f_1\big>\big<e_\ell,f_2\big>\big<e_j,C_he_\ell\big>
. 
\end{align}
To prove convergence of the covariances to those of $Y_0$, we need the last line to be equal to:
\begin{equation}
\sum_{j,\ell =0}^\infty\big<e_j,f_1\big>\big<e_\ell,f_2\big>\big<e_j,C_he_\ell\big>
= 
\int_{\T}f_1(x)(C_h f_2)(x)\, dx
.
\end{equation}
For this to be true, it is enough to prove that $\big<e_j,C_h1\big> = \big<1,C_he_j\big>= 0$ for each $j\in\N_{\geq 1}$. 
We show in Proposition~\ref{prop_existence_regularity_h} that $\int_{\T}h(x)\, dx=0$. 
It follows that the symmetric operator $C_h$ admits $1$ has an eigenvector, 
thus the orthogonal decomposition $\mathbb L^2(\T)  = 1\oplus 1^\perp$ 
(for the usual inner product in $\mathbb L^2(\T)$) is stable under $C_h$. 
This concludes the proof of Proposition~\ref{prop_convergence_approximation}.
\subsection{Long time asymptotics and conclusion of the proof}\label{sec_long_time_asymptotics}
To conclude the proof of Proposition \ref{prop_to_prove_main_result}, 
it is enough to prove that $Y_\cdot$ decorrelates in time in the following sense.
\begin{prop}\label{prop_time_decorrel}
Let $\lambda,E\in\R$ be such that $\sigma h$ has leading eigenvalue strictly below $1$. 
Let $\tau\geq 0$ and $G:\mathcal D(\T_+,\s'(\T))\to\R$, $H_1,H_2:\s'(\T)\to\R$ be bounded continuous functions. 
There is $C>0$ and $c_{h}>0$ such that, for any $t\geq 0$ with $t+\tau<T$, 
\begin{align}
\Big|{\bf E}^{\nu^\infty_{h}}_{h,\lambda+E}\Big[H_1(Y_0)H_2(Y_T)G\big((Y_s)_{t\leq s\leq t+\tau}\big)\Big] 
&- \nu^\infty\big[H_1(Y)\big]\nu^\infty\big[H_2(Y)\big] {\bf E}^{\nu^\infty_{h}}_{h,\lambda+E}\big[G\big((Y_s)_{0\leq s\leq \tau}\big)\big]\Big|
\nonumber\\
&\qquad\leq 
C\exp\big[-c_{h}\min\{t,T-\tau-t\}\big]
.
\end{align}
\end{prop}
As $Y^{(n)}_\cdot$ converges weakly to $Y_\cdot$, 
it is enough to prove Proposition \ref{prop_time_decorrel} for $Y^{(n)}_\cdot,\nu^{(n)}$ instead of $Y_\cdot,\nu^\infty$ 
with the same constants $C,c_{h}$ for each $n\in\N_{\geq 1}$. 
In particular, Proposition~\ref{prop_time_decorrel} is a direct consequence of the following lemma (recall the definition of ${\bf E}^{(n)}$ from below~\eqref{eq_def_Y_(n)_appendix}). 

\begin{lemm}\label{lemm_long_time}
With the notations of Proposition \ref{prop_time_decorrel} and $c_{h}$ given by~\eqref{eq_gap_section_long_time} below, 
for any $\delta>0$, there are constants $C_1,C_2$ depending only on $H_1,H_2,h,\delta$ such that, 
for any $t\geq \delta$:
\begin{equation}
\sup_{n\in\N_{\geq 1}} \Big|{\bf E}^{(n)}\Big[H_1(Y^{(n)}_0) G\big((Y^{(n)})_{s\geq t}\big)\Big] - \nu^{(n)}(H_1) {\bf E}^{(n)}\Big[G\big((Y^{(n)})_{s\geq 0}\big)\Big] \Big| 
\leq 
C_1e^{-c_{h}t}
,
\label{eq_time_decorrel_at_0}
\end{equation}
and for $t+\tau<T$:
\begin{equation}
\sup_{n\in\N_{\geq 1}} \Big|{\bf E}^{(n)}\Big[H_2(Y^{(n)}_T) G\big((Y^{(n)})_{s\leq t+\tau}\big)\Big] 
- \nu^{(n)}(H_2) {\bf E}^{(n)}\Big[G\big((Y^{(n)})_{s\leq t+\tau}\big)\Big] \Big| 
\leq 
C_2e^{-c_{h}(T-t-\tau)}
.
\label{eq_time_decorrel_at_T}
\end{equation}
\end{lemm}
To prove Lemma~\ref{lemm_long_time}, the key element is the fact that the generator $\mathcal L_{h}$ is for sub-critical $\lambda,E$. 
\begin{lemm}\label{lemm_gap}
Let $\lambda,E\in\R$ be sub-critical in the sense of Proposition~\ref{prop_solving_PDE_h}. 
Then the generator $\mathcal L_{h}$ is gapped: 
	there is $c_{h}>0$ such that, 
	for any $f\in C^\infty(\T)$ with $\int_{\T}f(x)\, dx = 0$:
	\begin{equation}
	\int_{\T} f(x)\big( -\mathcal L_{h}f\big)(x)\, dx
	\geq 
	c_{h}\int_{\T} f(x)^2\, dx
	.
	\label{eq_gap_section_long_time}
	\end{equation}
	Let $u\in\R^n$ with $\|u\|_2=1$ and write $f = \sum_{\ell=1}^n u_\ell e_\ell$. 
Then the gap~\eqref{eq_gap_section_long_time} carries over to $B_n$:
\begin{align}
u^T(-B_n) u 
&= 
-\int_{\T}f(x)\big(\mathcal L_{h}f\big)(x)dx 
\geq 
c_{h}
.
\label{eq_lower_bound_spectre_-L_Y_dans_preuve}
\end{align}
\end{lemm}
\begin{proof}
Recall that $\mathcal L_{h}$ acts on $f\in C^\infty(\T)$ according to:
\begin{equation}
\mathcal L_{h} f(x)
=
f''(x) 
-\sigma\int_{\T} h'(x-y)f'(y)\, dy
=
(\text{id} - \sigma h)\Delta f(x)
.
\end{equation}
As $h$ is translation invariant, $\Delta$ and $h$ commute and are thus diagonalisable in the real orthonormal basis $(e_\ell)_{\ell\in\N}$ of eigenvectors of the Laplacian. 
The assumption that $\lambda,E$ are sub-critical implies by Proposition~\ref{prop_solving_PDE_h} that $\sigma h$ has leading eigenvalue $\rho_{\sigma h}$ strictly smaller than $1$. 
As $-\Delta$ has gap $4\pi^2$ on $\T$, this implies:
\begin{equation}
-\int_{\T} f(x)\mathcal L_{h}f(x)\, dx
\geq 
c_{h}\int_{\T}f(x)^2\, dx
,
\quad
c(h)
\geq 
4\pi^2(1-\rho_{\sigma h})
.
\end{equation}
\end{proof}
\begin{proof}[Proof of Lemma~\ref{lemm_long_time}]
The proof is an explicit Gaussian computation for finite dimensional Ornstein-Uhlenbeck processes.

Let us start with some notations. 
Let $\nu^{(n)}_t(x,dy) = p^{(n)}_t(x,y)\, dy$ denote the law of $y^{(n)}_t$ starting from $x$, 
and recall that the Gaussian measure $\nu^{(n)}(dy) = p^{(n)}(y)\, dy$ is the invariant measure of the diffusion $y^{(n)}_\cdot$.  
For two probability measures $\pi_1(dx) = \rho_1(x)\, dx$, 
$\pi_2(dx)=\rho_2(x)\, dx $ on $\R^n$, 
let $\|\pi_1-\pi_2\|_{TV}$ denote their total variation distance:
\begin{equation}
\|\pi_1-\pi_2\|_{TV}
=
\frac{1}{2}\int_{\R^n}|\rho_1(x)-\rho_2(x)|\, dx
.
\end{equation}
Let $\Psi^{(n)}$ be the continuous function such that $Y^{(n)}_t = \Psi^{(n)}(y^{(n)}_t)$ for each $t\geq 0$. 
We first prove decorrelation at time $0$, i.e.~\eqref{eq_time_decorrel_at_0}. 
By the Markov property for $y^{(n)}_\cdot$,
\begin{align}
&\Big|{\bf E}^{(n)}\Big[H_1(Y^{(n)}_0) G\big((Y^{(n)})_{s\geq t}\big)\Big] - \nu^{(n)}(H_1) {\bf E}^{(n)}\Big[G\big((Y^{(n)})_{s\geq 0}\big)\Big] \Big| 
\nonumber\\
&\quad
= 
\bigg|\int \nu^{(n)}(dx) H_1(\Psi^{(n)}(x)) \Big({\bf E}^{\nu^{(n)}_t(x)}\big[G\big(\big(\Psi^{(n)}(y^{(n)}_s)\big)_{s\geq 0}\big)\big]
 - {\bf E}^{(n)}\big[G\big(\big(\Psi^{(n)}(y^{(n)}_s)\big)_{s\geq 0}\big)\big]\Big)\bigg|
\nonumber \\
&\quad
\leq 
2\|G\|_\infty\|H_1\|_\infty\int \nu^{(n)}(dx) \big\|\nu^{(n)}_t(x)-\nu^{(n)}\|_{TV}
.
\label{eq_TV_distance_to_compute}
\end{align}
Let us compute this total variation distance. 
Writing $\nu^{(n)}_t(x) = p^{(n)}_t(x,y)\, dy$, 
an elementary elementary computation gives $p^{(n)}(y)=\lim_{t\rightarrow\infty}p^{(n)}_t(x,y)$ 
and:
\begin{equation}
p^n_t(x,y) 
=
\frac{1}{(2\pi)^{n/2} \det(C_n(t))^{1/2}}\exp\Big[-\frac{1}{2}(y- e^{tB_n}x)^TC_n(t)^{-1}(y- e^{tB_n}x)\Big]
,
\end{equation}
with, recalling~\eqref{eq_def_D_n_B_n_appendix} and the fact that $B_n,D_n,C_n$ commute:
\begin{equation}
C_n(t) 
:=
\int_0^t D_n^{1/2} e^{\frac{s}{2}(B_n + B_n^T)}D_n^{1/2}\, ds
=
C_n 
- 
\int_t^\infty D_n^{1/2} e^{sB_n}D_n^{1/2}\, ds
.
\label{eq_def_C_n_t}
\end{equation}
The last integral is well defined as $B_n$ is gapped, 
recall~\eqref{eq_lower_bound_spectre_-L_Y_dans_preuve}. 
Let $x\in\R^n$ and $t\geq 1$. 
Pinsker's inequality (see Proposition~\ref{prop_entropy_estimate_general}, the inequality is valid on general measurable spaces) 
bounds the total variation distance in terms of the relative entropy:
\begin{equation}
\|\nu^{(n)}_t(x)-\nu^{(n)}\|_{TV}
\leq 
\sqrt{2H\big(\nu^{(n)}_t(x)|\nu^{(n)}\big)}
,
\end{equation}
with:
\begin{equation}
H\big(\nu^{(n)}_t(x)|\nu^{(n)}\big)
:=
\int \log\Big(\frac{p^{(n)}_t(x,y)}{p^{(n)}(y)}\Big)p^{(n)}_t(x,y)\, dy
.
\end{equation}
The relative entropy between Gaussian measures can be computed explicitly:
\begin{equation}
H\big(\nu^{(n)}_t(x)|\nu^{(n)}\big)
=
\frac{1}{2}\text{tr}\big(C_n^{-1}C_n(t)\big) - \frac{n}{2}
+
\frac{1}{2}(e^{tB_n}x)^TC_n^{-1}(e^{tB_n}x) - \frac{1}{2}\log\det\big(C_n^{-1}C_n(t)\big)
.
\label{eq_relative_entropy_n_t}
\end{equation}
Let us bound each term above. 
Consider first the trace. 
The explicit formula~\eqref{eq_def_C_n_t} for $C_n(t),C_n$ gives, as $B_n$ is invertible:
\begin{align}
\frac{1}{2}\text{tr}\big(C_n^{-1}C_n(t)\big) - \frac{n}{2}
&=
-\frac{1}{2}\text{tr}\Big(-2B_n\int_t^\infty e^{2sB_n}\, ds\Big)
=
-\frac{1}{2}\text{tr}\big(e^{2tB_n}\big)
\nonumber\\
&=
-\frac{1}{2}e^{-2c_h t}\, \text{tr}\big(e^{2t(B_n+c_h\text{id})}\big)
\label{eq_trace_exp_B_n}
.
\end{align}
Since $\Delta$ and $h$ commute by translation invariance of $h$, 
the spectrum of $B_n$ is known explicitly. 
In particular, 
if $0\leq \lambda^n_1\leq ...\leq\lambda^n_n(B_n)$ denote the eigenvalues of $-B_n$ then: 
\begin{align}
\lambda^n_{2\ell},
\lambda^n_{2\ell+1}
\geq 
\ell^2 c_h
,\qquad 
0\leq \ell\leq \lfloor n/2\rfloor
.
\end{align}
As a result,
\begin{equation}
\text{tr}\big(e^{2t(B_n+c_\ell\text{id})}\big)
\leq 
\sum_{\ell=1}^{\lfloor n/2\rfloor} 2e^{-(\ell^2-1)c_h t}
,
\label{eq_bound_trace_exp_B_n}
\end{equation}
which is bounded uniformly in $n$ and $t\geq 1$. 
Recalling~\eqref{eq_trace_exp_B_n}, 
we have thus proven:
\begin{equation}
\forall t\geq 1,\qquad
\sup_n\Big|\frac{1}{2}\text{tr}\big(C_n^{-1}C_n(t)\big) - \frac{n}{2}\Big|
\leq 
C e^{-2c_{h} t}
,\qquad 
C = C(h)>0
.
\label{eq_bound_trace_C_ninverse_C_n}
\end{equation}
The determinant in~\eqref{eq_relative_entropy_n_t} is treated similarly:
\begin{equation}
\frac{1}{2}\log\det\big(C_n^{-1}C_n(t)\big)
=
\frac{1}{2}\log\det\big(1-e^{2tB_n}\big) 
\leq 
Ce^{-2c_{h} t},
\qquad 
C=C(h)>0
.
\end{equation}
All in all, recalling~\eqref{eq_relative_entropy_n_t}, 
we find that~\eqref{eq_TV_distance_to_compute} is bounded
by:
\begin{equation}
C\|G\|_\infty\|H_1\|_\infty \int \nu^{(n)}(dx) \Big[e^{-2c_{h} t}+ (e^{tB_n}x)^T C_n^{-1}(e^{tB_n}x) \Big]^{1/2}
.
\label{eq_interm_bound_decorrel_at_0}
\end{equation}
Equation~\eqref{eq_interm_bound_decorrel_at_0} can be bounded
using Cauchy-Schwarz inequality and the following inequality, valid for any matrices $A,B$:
\begin{equation}
\text{tr}(A^TB)
\leq 
\text{tr}(A^TA)^{1/2}\text{tr}(B^TB)^{1/2}
.
\end{equation}
The resulting bound on~\eqref{eq_interm_bound_decorrel_at_0} reads:
\begin{align}
&C\|G\|_\infty \|H_1\|_\infty \Big[e^{-2c_{h} t} + \int \nu^{(n)}(dx) (e^{tB_n}x)^T C_n^{-1}e^{tB_n}x \Big]^{1/2}
\nonumber\\
&\hspace{3cm}
=
C\|G\|_\infty \|H_1\|_\infty \Big[e^{-2c_{h} t} + \text{tr}\Big(C_ne^{tB_n} C_n^{-1}e^{tB_n}\Big)\Big]^{1/2}
\nonumber\\
&\hspace{3cm}
\leq
C\|G\|_\infty \|H_1\|_\infty \Big[e^{-2c_{h} t} + \text{tr}\big(e^{2tB_n}\big)\Big]^{1/2}
\nonumber\\
&\hspace{3cm}
\leq 
C'\|G\|_\infty\|H_1\|_\infty e^{-c_{h} t}
.
\end{align}
This proves the first inequality~\eqref{eq_time_decorrel_at_0} in the lemma. 
This and the reversibility of $y^{(n)}_\cdot$ immediately give the second inequality~\eqref{eq_time_decorrel_at_T}.
\end{proof}

\section{Correlations under the current-biased dynamics}\label{sec_comparison}
In this section, we prove Proposition~\ref{prop_correl_sec_results} on the correlation structure of the current-biased dynamics at times $1\ll t \ll T$. 
We then show heuristically that the result agrees with the prediction of~\cite{Bodineau2008}, 
where correlations were computed at fixed density of particles.
\begin{prop}\label{prop_correlations_current_app}
Let $\lambda,E\in\R$ be sub-critical as in Proposition~\ref{prop_solving_PDE_h}. 
Let $k_{curr}$ denote the kernel encoding long-range correlations under the current-biased dynamics at times far away from $0$ and $T$, in the sense that:
\begin{equation}
\lim_{T\rightarrow\infty}\lim_{N\rightarrow\infty}\E^{curr,\nu^N_{1/2}}_{\lambda,E,T}\big[Y^N_{T/2}(f_1)Y^N_{T/2}(f_2)\big] 
=
\int f_1(x) (\sigma\text{id} + k_{curr})f_2 (x)\, dx
.
\label{eq_def_k_curr_sec6}
\end{equation}
Then $k_{curr}$ satisfies $(\sigma^{-1}\text{id} - h)^{-1}=\sigma\text{id}+k_{curr}$ and is explicitly given by:
\begin{equation}
k_{curr}(x)
=
\sqrt{2}\sigma\sum_{\ell\geq 1}\bigg[\frac{1}{\sqrt{1+\frac{\lambda(\lambda+2E)\sigma}{\pi^2\ell^2}}} -1\bigg]\sqrt{2}\cos(2\pi\ell x)
,\qquad 
x\in\T
.
\label{eq_k_curr_mine}
\end{equation}
\end{prop}
\begin{proof}
Theorem~\ref{theo_main_result} shows that the current-biased dynamics has the same fluctuations as the dynamics $\Prob^{\nu^N_h}_{h,\lambda+E}$ in the sense of convergence of bounded observables. 
To determine the correlation structure of the current-biased dynamics (and in particular prove that the limits in~\eqref{eq_def_k_curr_sec6} make sense), 
we check that Theorem~\ref{theo_main_result} also holds for observables of the form $Y^N_t(f_1)Y^N_t(f_2)$ ($t\geq 0,f_1,f_2\in C^\infty(\T)$), 
then compute this covariance using the characterisation fo limiting fluctuations in Theorem~\ref{theo_fluct_correl_driven_process}. 

Let $f_1,f_2\in C^\infty(\T)$. 
The moment bound~\eqref{eq_moment_bound} on $\Pi^N_t$ at each time $t$ and 
the relationship~\eqref{eq_lien_Pi_et_Y_dans_intro} between $\Pi^N_t$ and $Y^N_t$ give uniform integrability of $(Y^N_t(f_1)Y^N_t(f_2))_N$ under $\Prob^{\nu^N_h}_{h,\lambda+E}$. 
Combining with the time decorrelation of Proposition~\ref{prop_time_decorrel}, 
valid with the same proof for unbounded test functions $G(Y_{T/2})= Y_{T/2}(f_1)Y_{T/2}(f_2)$ that have bounded moments, 
we get:
\begin{equation}
\lim_{T\rightarrow\infty}\lim_{N\rightarrow\infty}\E^{curr,\nu^N_{1/2}}_{\lambda,E,T}\big[Y^N_{\frac{T}{2}}(f_1)Y^N_{\frac{T}{2}}(f_2)\big]
=
{\bf E}^{\nu^{\infty}}_{h,\lambda+E}\big[Y_{0}(f_1)Y_0(f_2)\big] 
,
\end{equation}
where we used the stationarity of $Y_\cdot$ under ${\bf E}^{\nu^{\infty}}_{h,\lambda+E}$. 
Since $Y_0$ is a Gaussian field with covariance $C_h = (\sigma^{-1}\text{id} - h)^{-1}$, 
there is $k_{curr}\in C^\infty([0,1])\cap C^0(\T)$ such that $C_h = \sigma \text{id} + k_{curr}$ (because $C_h-\sigma\text{id} =  \sigma h C_h$ can be checked to be a kernel operator on $\mathbb L^2(\T)$; the regularity of the associated kernel is inherited from $h$). 

It remains to find an expression for this kernel.
We do so using Ito formula on $Y_t(f_1)Y_t(f_2)$ ($t\geq 0$) and the stationarity of $(Y_s)_{s\geq 0}$. 
Recall the martingale problem characterising $(Y_s)_{s\geq 0}$ in Theorem~\ref{theo_fluct_correl_driven_process}. 
Taking expectations, we find:
\begin{equation}
\int_{\T} f_1(x)(\sigma \text{id} + k_{curr})\mathcal L_h f_2(x)\, dx + \sigma\int_{\T} f'_1(x)f'_2(x)\, dx
=
0
,
\end{equation}
where $\mathcal L_h f_2 = f_2'' -\sigma \int h'(\cdot-y) f_2'(y)\, dy$. 
Taking $f_1 = f_2=f$ with $f'(x)=e^{-2i\pi\ell x}$ ($x\in\T$) for $\ell\in\Z\setminus\{0\}$ and evaluating at $x=0$ yields, 
with $c_\ell(u) = \int_{\T} u(x)e^{-2\iota\pi\ell x}\, dx$ the $\ell^{\text{th}}$ Fourier coefficient of a function $u\in\mathbb{L}^2(\T)$:
\begin{equation}
-2\iota\pi\ell c_\ell(k_{curr}) +\sigma c_\ell(k_{curr})c_\ell(h') +\sigma^2c_\ell(h')
=
0
.
\end{equation}
Since $c_{\ell}(h')=2\iota\pi\ell c_\ell(h)$ is known by Proposition~\ref{prop_solving_PDE_h}, we obtain:
\begin{align}
c_\ell(k_{curr}) 
&= 
\frac{\sigma^2c_\ell(h')}{2\iota\pi\ell -\sigma c_\ell(h')}
=
\sigma \bigg[\frac{1}{\sqrt{1+\frac{\lambda(\lambda+2E)\sigma}{\pi^2\ell^2}}} -1\bigg]
.
\end{align}
\end{proof}
{\bf Comparison with~\cite{Bodineau2008}}. 
The conjecture in~\cite{Bodineau2008} 
(see Equation 61 there with the values $D=1$, $\sigma(\rho) = 2\rho(1-\rho)$, $j_0 =\sigma(\rho) (\lambda+E)$, $\bar j_0 = \sigma(\rho) E$) is the following.  
Let $\rho\in[0,1]$,
and consider the current-biased dynamics~\eqref{eq_current_biased_dynamics_intro} starting from the uniform measure with $p=\lfloor N\rho\rfloor$. 
Then, in the large $N$ limit, 
the two-point correlation kernel $k^\rho_{curr}$ satisfies, 
for each $f\in C^{\infty}(\T)$ and $x\in\T$:
\begin{equation}
k^\rho_{curr}f(x)
=
\rho(1-\rho)\int_{\T}\sum_{\ell\geq 1}2\cos\big(2\pi\ell (x-y)\big)\Bigg(1+\Bigg[\frac{1}{\sqrt{1+\frac{\rho(1-\rho)\lambda(\lambda+2E)}{\pi^2\ell^2}}}-1\Bigg]\Bigg)f(y) \, dy
.
\label{eq_C_curr_canonical}
\end{equation}
To be precise, the conjecture in~\cite{Bodineau2008} is formulated for the dynamics conditioned to having current $j_0$, rather than for our current-biased dynamics tilted by $\lambda$ times the current. 
In the large $N$, large $T$ limit, the two should be equivalent with $j_0 = \rho(1-\rho)(\lambda+E)$ at least in the absence of dynamical phase transition~\cite{TouchetteEquivalenceNonequivalenceEnsembles2015}. 

Using the identity $\delta(\cdot)-1 = \sum_{\ell\geq 1}2\cos(2\pi\ell \cdot)$, 
with $\delta(\cdot)$ the Dirac distribution,~\eqref{eq_C_curr_canonical} becomes:
\begin{equation}
C^\rho_{curr}f(x) 
=
\rho(1-\rho)f(x)
+k^\rho_{curr}f(x)
,
\end{equation}
with the translation-invariant kernel $k^\rho_{curr}$ given for $x\in\T$ by:
\begin{equation}
k^\rho_{curr}(x) 
=
-\rho(1-\rho) + \rho(1-\rho)\sum_{\ell\geq 1}2\cos(2\pi\ell x)\Bigg[\frac{1}{\sqrt{1+\frac{\rho(1-\rho)\lambda(\lambda+2E)}{\pi^2\ell^2}}}-1\Bigg]
.
\label{eq_k_curr_canonical}
\end{equation}
%

The next proposition states that the expression~\eqref{eq_k_curr_mine} of $k_{curr}$ and the $k_{curr}^\rho$ of~\cite{Bodineau2008} are compatible, 
in the sense that one recovers $k_{curr}$ upon averaging $k_{curr}^\rho$ on all densities. 
\begin{prop}
Let $\mathcal B_N$ denote the Binomial law with parameters $N,1/2$. Then:
\begin{equation}
k_{curr}(x)
=
\lim_{N\rightarrow\infty} 
\bigg[\E_{\mathcal B_N}\big[ k^{p/N}_{curr}(x)\big] + N \text{Var}_{\mathcal B_N}\big[ p/N\big]\bigg]
=
\lim_{N\rightarrow\infty} \E_{\mathcal B_N}\big[ k^{p/N}_{curr}(x)\big] + \frac{1}{4}
,
\label{eq_proof_comparison_k_curr}
\end{equation}
\end{prop}
\begin{rmk}
The decomposition of $k_{curr}$ as the sum of two terms is not artificial. 
Indeed, on the one hand $k_{curr}$ is the limit of $N$ times a certain covariance, 
and the right-hand side of~\eqref{eq_proof_comparison_k_curr} similarly involves two rescaled covariances.  
On the other hand $k_{curr}$ can be seen as the limiting covariance under the convolution of a binomial measure determining the density of particles, and a measure associated with the steady state of the current biased simple exclusion at fixed density. 
The decomposition of $k_{curr}$ then follows from the general fact that for any three probability measures $\pi,\mu,\nu$ such that $\pi = \mu\star\nu$ and any test functions $F,G$ for which the following quantities make sense:
\begin{equation}
\text{Cov}_{\pi}(F,G)
=
\text{Cov}_{\mu}\big(\E_{\nu}[F],\E_{\nu}[G]\big)
+\E_{\mu}\big[\text{Cov}_{\nu}(F,G)\big]
.
\end{equation}
%
\demo
\end{rmk}
\begin{proof}
Note that Fourier coefficients of $k^{\rho}_{curr}$ in~\eqref{eq_k_curr_canonical} scale like $o(\ell^{-2})$ uniformly in $\rho$,
thus $(\rho,x)\mapsto k_{curr}^\rho(x)$ is continuous on $[0,1]\times \T$ by Dirichlet's normal convergence theorem. 
Let $x\in\T$. 
As $\rho\in[0,1]\mapsto k^\rho_{curr}(x)$ is continuous and bounded, 
the law of large numbers gives, 
recalling the expression~\eqref{eq_k_curr_mine} of $k_{curr}$:
\begin{equation}
\lim_{N\rightarrow\infty}\E_{\mathcal B_N}\big[ k^{p/N}_{curr}(x)\big] 
=
-\frac{1}{4} + k_{curr}(x)
.
\end{equation}
On the other hand, elementary computations give 
$N\text{Var}_{\mathcal B_N}\big[ p/N\big]=N/4$. 
\end{proof}

\section*{Acknowledgements}
The author thanks Thierry Bodineau for many discussions and comments at all stages of this work. 
This work was partially carried out at the University of Cambridge with support from the European Research Council under the European Union's Horizon 2020 research and innovation programme
(grant agreement No.~851682 SPINRG).

\appendix
\section{Concentration of measure}\label{app_concentration}
\subsection{Concentration for discrete Gaussian measures}\label{sec_concentration_product_measure}
Let $g:\T^2\to\R$ be a bounded function. 
Recall that $\nu^N_{1/2}$ is the product Bernoulli measure with density $1/2$ and that $\nu^N_g$ is the following discrete measure:
\begin{equation}
\forall \eta\in\Omega_N,\qquad 
\nu_{g}^N 
:= (\mathcal Z^N_g)^{-1}e^{2\Pi^N(g)}\nu^N_{1/2},\quad \mathcal Z^N_g \text{ a normalisation factor}.\label{eq_def_mesure_gaussienne_appendice}
\end{equation}
We often decompose $g$ as $g = g_-+g_+$ with $g_-$ ($g_+$) a negative (positive) kernel on $\mathbb L^2(\T)$:
\begin{equation}
\forall f\in\mathbb L^2(\T),\qquad 
\pm\int_{\T^2}f(x) g_\pm(x,y)f(y)\, dx\, dy
\geq 
0
.
\label{eq_g_pm_neg_kernel_appendix}
\end{equation}
When $g_{\pm}\in C^0(\T^2)$, 
it is known that the definition~\eqref{eq_g_pm_neg_kernel_appendix} is equivalent to asking for the matrix $\pm(g_{\pm}(x_i,x_j))_{1\leq i,j\leq n}$ to be positive for any $n\in\N_{\geq 1}$ and any $x_1,...,x_n\in\T$, see e.g.~\cite{FerreiraEigenvaluesIntegralOperators2009}. 
In the following we will often assume $g_{\pm}\in C^1_D(\T^2)\subset C^0(\T^2)$, 
the set of continuous functions with restrictions to $\{x\leq y\}$, $\{x\geq y\}$ that are $C^1$ (the derivatives of the restrictions on the diagonal need not coincide). 
Recall also the notations:
\begin{equation}
\sigma 
:= 
1/4
,\qquad
\bar\eta_i := \eta_i-\frac{1}{2}
,\quad i\in\T_N
. 
\label{eq_def_sigma_app_concentration}
\end{equation}
\begin{lemm}\label{lemm_bound_correlations}
Let $n\in\N$ and let $I\subset \T_N$ with $|I| = 2n+1$. 
Then:
\begin{equation}\label{eq_odd_powers_vanish}
\nu^N_{g}\Big(\prod_{i\in I}\bar\eta_i\Big) = 0.
\end{equation}
Moreover, assume that the $g$ of \eqref{eq_def_mesure_gaussienne_appendice} satisfies $g_{\pm}\in C^1_D(\T^2)$ and that, for some $\delta>0$:
\begin{equation}
\forall f\in\mathbb L^2(\T),
\qquad
\sigma \int_{\T^2}f(x)g_+(x,y)f(y)\,dx\,dy 
\leq 
(1-\delta)\int_{\T}f(x)^2\, dx
.
\label{eq_assumption_gap_g}
\end{equation}
Then:
\begin{equation}
1
\leq
\sup_{N\in\N_{\geq 1}}\mathcal Z^N_g
<
\infty
\quad 
\text{and}\quad 
\sup_{J\subset\T_N : |J|=2n}\nu^N_{g}\Big(\prod_{j\in J}\bar\eta_j\Big) 
= 
O(N^{-n})
,\qquad
n\in\N_{\geq 1}
.
\label{eq_bound_correlations}
\end{equation}
\end{lemm}
A similar result was obtained in~\cite[Lemma A.1]{Correlations2022} under the stronger assumption that $\|\sigma g_+\|_2<1$. 
Here the assumption~\eqref{eq_assumption_gap_g} is optimal.
\begin{proof}
The mapping $\eta\in\Omega_N \mapsto (1-\eta_i)_{i\in\T_N}\in\Omega_N$ leaves $\nu^N_{g}$ invariant due to the $1/2$ density, 
but it changes the sign of $\prod_{i\in I}\bar\eta_i$ (recall $\bar\eta_\cdot :=\eta_\cdot-1/2$). 
This proves~\eqref{eq_odd_powers_vanish}. 
Moreover, Jensen inequality implies:
\begin{equation}
\mathcal Z^N_g
\geq 
\exp\Big[\nu^N_{1/2}\big[2\Pi^N(g)\big]\Big]
=
1
.
\end{equation}
The estimate on even correlations is proven in Lemma A.2 in \cite{Correlations2022} provided $\sup_{N}\mathcal Z^N_g<\infty$. 
We thus only prove this bound.

Recall that $g_-\subset C^0(\T^2)$ being a negative kernel implies:
\begin{equation}
2\Pi^N(g_-)
:=
\frac{1}{2N}\sum_{i\neq j}(g_-)_{i,j}\bar\eta_i\bar\eta_j
=
\frac{1}{2N}\sum_{i,j}(g_-)_{i,j}\bar\eta_i\bar\eta_j 
-\frac{\sigma}{2N}\sum_{i}(g_-)_{i,i}
\leq 
\frac{\sigma}{2}\|g_-\|_\infty
.
\end{equation}
Thus:
\begin{align}
\mathcal Z^N_g 
&\leq
\nu^N_{1/2}\Big(\exp\Big[\frac{1}{2N}\sum_{i\neq j\in\T_N}(g_+)_{i,j}\bar\eta_i\bar\eta_j + \frac{1}{2N}\sum_{i,j\in\T_N}(g_-)_{i,j}\bar\eta_i\bar\eta_j\Big]\Big)e^{\sigma \|g_-\|_\infty/2}
\nonumber\\
&\leq 
\nu^N_{1/2}\Big(\exp\Big[\frac{1}{2N}\sum_{i\neq j\in\T_N}(g_+)_{i,j}\bar\eta_i\bar\eta_j\Big]\Big)e^{\sigma\|g_-\|_\infty/2}
.
\label{eq_getting_rid_of_g_-_mathcal_Z}
\end{align}
It is thus enough to bound $\mathcal Z^N_g$ when $g = g_+$, i.e. $g_-=0$, 
which we now assume. 
To make use of Assumption~\eqref{eq_assumption_gap_g} on $g$, 
it is convenient to replace $(g_{i,j})_{(i,j)\in\T_N^2}$ by the matrix $\hat g^N$:
\begin{equation}
\hat g^N(i,j)
=
N^2\int_{(\frac{i}{N},\frac{j}{N})+\frac{1}{N}[0,1)^2}g(x,y)\,dx\, dy 
,\qquad
(i,j)\in \T_N^2
.
\label{eq_def_tilde_g_N}
\end{equation}
The advantage of $\hat g^N$ over $(g_{i,j})_{i,j}$ is that the spectrum of $\sigma \hat g^N/N$ is included in the spectrum of $\sigma g$. 
Indeed, 
if $u\in\R^N$ is an eigenvector of $\sigma\hat g^N$, 
one can check that the following $f\in\mathbb L^2(\T)$ is an eigenvector for $\sigma g$ associated with the same eigenvalue:
\begin{equation}
f(x)
=
\sqrt{N}{\bf 1}_{\frac{1}{N}[i,i+1)}(x) u_i
,\qquad 
x\in\T
.
\end{equation}
As $g=g_+\in C^1_D(\T^2)\subset C^0(\T^2)$, 
replacing $g$ by $\hat g^N$ is not restrictive as far as bounding $\mathcal Z^N_g$ is concerned: 
\begin{equation}
\mathcal Z^N_g 
\leq 
C(g)\, \tilde{\mathcal{Z}}^N_g
,
\qquad
\tilde{\mathcal{Z}}^N_g
:=
\nu^N_{1/2}\Big(\exp\Big[\frac{1}{2N}\sum_{i,j\in\T_N} \bar\eta_i\bar\eta_j \hat g^N(i,j)\Big]\Big)
.
\label{eq_def_tilde_mathcal_Z_N_g}
\end{equation}
Let us estimate $\tilde{\mathcal{Z}^N_g}$. 
Expanding the exponential in~\eqref{eq_def_tilde_mathcal_Z_N_g} yields:
\begin{align}
\nu^N_{1/2}&\Big(\exp\Big[\frac{1}{2N}\sum_{i,j\in\T_N}\hat g^N(i,j)\bar\eta_i\bar\eta_j\Big]\Big)
\nonumber\\
&\qquad
=
\sum_{n=0}^\infty \frac{1}{2^n n! N^n}\sum_{\substack{i_1,...,i_{2n}}}\nu^N_{1/2}\Big( \hat g^N(i_1,i_2)...\hat g^N(i_{2n-1},i_{2n})\prod_{j=1}^{2n}\bar\eta_{i_j}\Big)
.
\label{eq_average_estimate_Z}
\end{align}
Let $P_{2n}$ denote the set of pairings of $\{1,...,2n\}$ elements. 
As $\nu^N_{1/2}$ is a product measure, 
the average in~\eqref{eq_average_estimate_Z} vanishes whenever an index $i_1,...,i_{2n}$ appears only once, 
i.e. vanishes unless there is a pairing $p\in P_{2n}$ such that $i_j=i_{p(j)}$ for $j\in\{1,...,2n\}$. 
Thus, using $\bar\eta_i^2=\sigma$ for each $i\in\T_N$:
\begin{align}
&\frac{1}{N^n}\sum_{\substack{i_1,...,i_{2n}}}\nu^N_{1/2}\Big[ \hat g^N(i_1,i_2)... \hat g^N(i_{2n-1},i_{2n})\prod_{j=1}^{2n}\bar\eta_{i_j}\Big]
\nonumber\\
&\qquad =
\frac{\sigma^n}{N^n}\sum_{p\in P_{2n}}\sum_{i_1,...,i_{2n}} \hat g^N(i_1,i_2)... \hat g^N(i_{2n-1},i_{2n}){\bf 1}_{(i_j)_{j\leq 2n}\text{ paired by }p}
\end{align}
To use the assumption on the spectrum of $\sigma g$, 
we now express the right-hand side in terms of the trace of $\hat g^N$. 
To do so, we need some notations. 
For $p\in P_{2n}$, 
let $b(p)$ be the bijection on $\{1,...,n\}$ obtained from $p$ by identifying integers $2i-1,2i$ ($1\leq i\leq n$), 
so that $b(p)(i)=j$ if either $2i-1$ or $2i$ is paired to either $2j-1$ or $2j$ ($1\leq i,j\leq n$), 
see Figure~\ref{fig_bijection}. 
\begin{figure}
\begin{center}
\includegraphics[width=9cm]{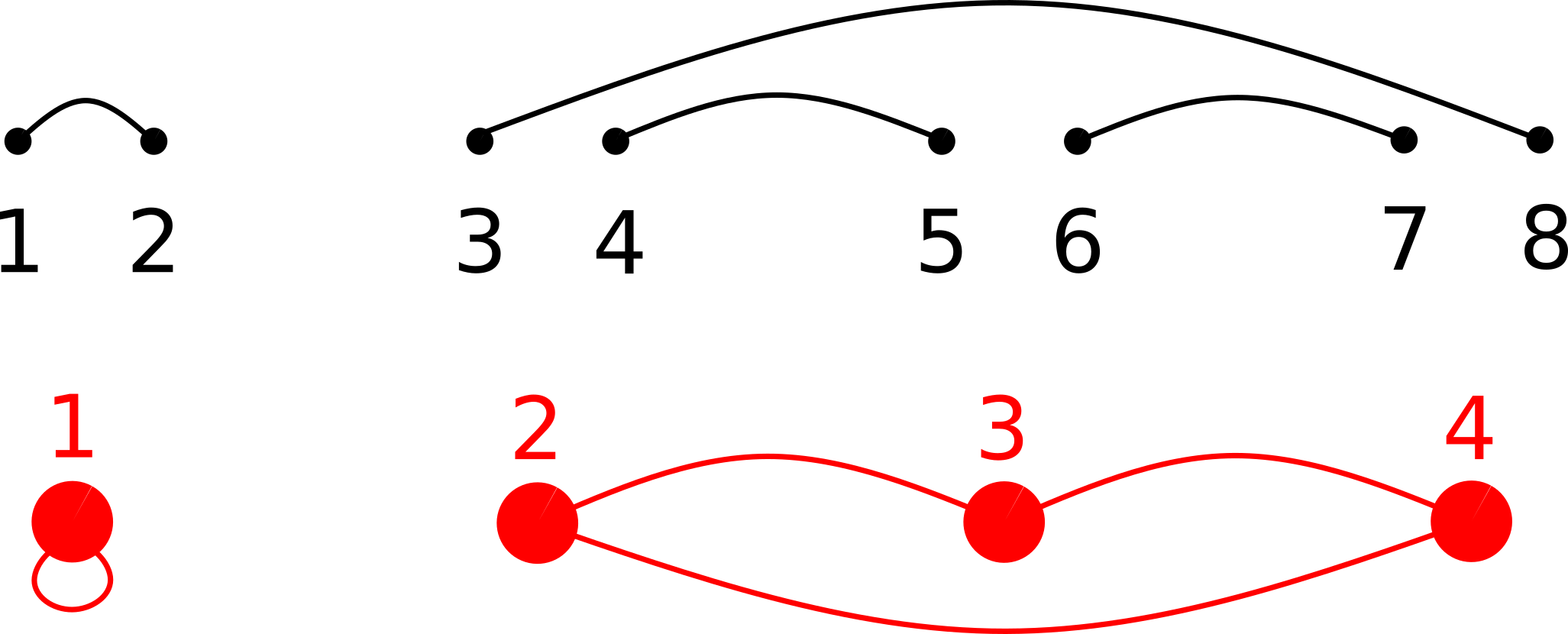} 
\caption{Pairing of eight points (black dots) and associated bijection on four points (red dots) obtained by fusing points $2i-1,2i$ into a single point $i$ ($i\in\{1,2,3,4\}$). 
Black lines link paired points. Lines are conserved when the points are fused (in red) and determine the cycles of the resulting bijection (here (1)(234)).\label{fig_bijection}}
\end{center}
\end{figure}
This is not a bijective mapping as discussed further along the proof. 
Let also $c_\ell$ ($1\leq \ell\leq n$) denote the number of cycles with $\ell$ points in $b(p)$, 
and note that $\sum_{\ell} \ell c_\ell = n$ by definition. 
For each $p\in P_{2n}$, one has:
%
\begin{align}
\sum_{\substack{i_1,...,i_{2n}}} \prod_{j=1}^n \hat g^N(i_{2j-1},i_{2j}){\bf 1}_{(i_j)_{j\leq 2n}\text{ paired by }p}
&=
\prod_{\ell=1}^{n}\Big[ \sum_{i\in\T_N}\big(\hat g^N\big)^\ell (i,i)\Big]^{c_\ell}
.
\end{align}
As a result,
\begin{align}
&\frac{1}{N^n}\sum_{\substack{i_1,...,i_{2n}}}\nu^N_{1/2}\Big[ \hat g^N(i_1,i_2)... \hat g^N(i_{2n-1},i_{2n})\prod_{j=1}^{2n}\bar\eta_{i_j}\Big]
\nonumber\\
&\qquad =
\sum_{p\in P_n}\prod_{\ell=1}^{n}\Big[\frac{\sigma^\ell}{N^\ell}\sum_{i\in \T_N} \big( \hat g^N\big)^\ell (i,i)\Big]^{c_\ell}
=
 \sum_{p\in P_n}\prod_{\ell=1}^n\Big[\text{tr}\Big(\Big(\frac{\sigma}{N}\hat g^N\Big)^\ell\Big) \Big]^{c_\ell}
.
\label{eq_average_estimate_Z_bis}
\end{align}
Letting $\E_{P_{2n}}$ denote the average with respect to the uniform measure on pairings, 
we obtain:
\begin{equation}
\mathcal{Z}^N_g
\leq 
\sum_{n=0}^\infty \frac{|P_{2n}|}{2^n n!}\E_{P_{2n}}\bigg[\prod_{\ell=1}^n\Big[\text{tr}\Big(\Big(\frac{\sigma}{N}\hat g^N\Big)^\ell\Big) \Big]^{c_\ell}\bigg]
.
\end{equation}
The fact that $|P_{2n}|=2^{-n} (2n)!/n!$ and the Stirling equivalent give:
\begin{align}
&\frac{(2n)!}{2^{2n}(n!)^2}
=
\frac{(1+o_n(1))}{\sqrt{\pi n}}
\leq 
\frac{C}{\sqrt{\pi (n+1)}}
,\qquad 
n\geq 0
\nonumber\\
&\hspace{3cm}\Rightarrow\quad
\tilde{\mathcal{Z}}^N_g
\leq  
\frac{C}{\sqrt{\pi}}\sum_{n=0}^\infty
\E_{P_{2n}}\bigg[\prod_{\ell=1}^n\Big[\text{tr}\Big(\Big(\frac{\sigma}{N}\hat g^N\Big)^\ell\Big) \Big]^{c_\ell}\bigg]
.
\label{eq_to_bound_tilde_Z}
\end{align}
Let us estimate this expectation. 
For small values of $\ell$, the boundedness of $g$ implies:
\begin{equation}
\text{tr}\Big[\frac{\sigma}{N} \hat g^N\Big]
\leq 
\|\sigma g\|_\infty
.
\end{equation}
Since all eigenvalues of $\sigma \hat g^N/N$ are in $(0,1-\epsilon)$ by~\eqref{eq_assumption_gap_g} and the discussion below~\eqref{eq_def_tilde_g_N}, 
the trace will be dominated by the leading eigenvalue for large values of $\ell$ and we prove next the existence of $\ell_0\in\N_{\geq 1}$, independent of $N$, such that:
\begin{equation}
\forall\ell\geq \ell_0,
\qquad
\text{tr}\Big[\big(\frac{\sigma}{N} \hat g^N\big)^{\ell}\Big]
\leq 
(1-\epsilon)^{\ell/2}
.
\label{eq_bound_trace_tilde_g}
\end{equation}
Write first, for any $\ell\in\N_{\geq 1}$:
\begin{equation}
\text{tr}\Big[\big(\frac{\sigma}{N} \hat g^N\big)^{\ell}\Big]
\leq 
(1-\epsilon)^{\ell-1}\text{tr}\Big[\frac{\sigma}{N} \hat g^N\Big]
\leq
(1-\epsilon)^{\ell-1}\|\sigma g\|_\infty
.
\label{eq_bound_trace_interm_appendix}
\end{equation}
Choose $\ell_1=\ell_1(g)\in\N_{\geq 1}$ such that the above is smaller than $1$ for $\ell\geq \ell_1$.  
The claim~\eqref{eq_bound_trace_tilde_g} follows for $\ell\geq \ell_0=2\ell_1$, as:
\begin{equation}
\text{tr}\Big[\big(\frac{\sigma}{N} \hat g^N\big)^{\ell}\Big]
\leq 
(1-\epsilon)^{\ell-\ell_1}\text{tr}\Big[\big(\frac{\sigma}{N} \hat g^N\big)^{\ell_1}\Big]
\leq 
(1-\epsilon)^{\ell-\ell_1}
\leq 
(1-\epsilon)^{\ell/2}
.
\end{equation}
It follows that the sum in~\eqref{eq_to_bound_tilde_Z} satisfies, recalling $\sum_\ell \ell c_\ell = n$:
\begin{align}
&\sum_{n=0}^\infty 
\E_{P_{2n}}\bigg[\prod_{\ell=1}^{n}\text{tr}\Big[\big(\frac{\sigma}{N} \hat g^N\big)^\ell \Big]^{c_\ell}\bigg]
\nonumber\\
&\quad 
\leq 
C(\ell_0,\|g\|_\infty)
+
\sum_{n\geq \ell_0}
\E_{P_{2n}}\bigg[\exp\Big[\sum_{\ell=1}^{\ell_0}\ell c_\ell\log (\|\sigma g\|_\infty) +\frac{1}{2}\log(1-\epsilon)\Big(n-\sum_{\ell=1}^{\ell_0}\ell c_\ell\Big)\Big]\bigg]
\nonumber\\
&\quad \leq 
C(\ell_0,\|g\|_\infty)
+
\sum_{n\geq \ell_0}
(1-\epsilon)^{n/2}\prod_{\ell=1}^{\ell_0}\E_{P_{2n}}\bigg[\exp\Big[\ell_0^2 c_\ell\big[\sigma \|g\|_\infty -\frac{1}{2}\log(1-\epsilon)\big]\Big]\bigg]
,
\end{align}
where we used Hölder inequality in the last line on the sum $\sum_{\ell=1}^{\ell_0}$. 
It remains to estimate this exponential moment, which now only involves small cycles.

This is done by rewriting it as an exponential moments of cycles of general permutations on $n$ elements, which is a well known quantity. 
To do so, notice that each bijection on $n$ elements is associated with $2^{\sum_{\ell} (\ell-1)c_\ell}=2^{n-\sum_\ell c_\ell}$ different pairings in $P_{2n}$, 
corresponding to the $2^{(\ell-1)c_\ell}$ ways of joining edges in each cycle. 
Thus, if $\E_{S_n}$ denote the uniform measure on permutations with $n$ elements, 
writing thanks to Jensen inequality:
\begin{equation}
\E_{P_{2n}}[\cdot ] 
= 
\E_{S_n}\big[2^{n-\sum_{\ell}c_\ell}\big]^{-1}\E_{S_n}\big[\cdot 2^{n-\sum_{\ell}c_\ell}\big]
\leq 
2^{\sum_{\ell}\E_{S_n}[c_\ell]}\E_{S_n}\big[\cdot \big]
,
\end{equation}
it is enough to prove that the following quantity has slower-than-geometric growth in $n$:
\begin{align}
2^{\sum_{\ell}\E_{S_n}[c_\ell]}\prod_{\ell=1}^{\ell_0}\E_{S_n}\bigg[\exp\Big[\ell_0^2 c_\ell\big[\sigma \|g\|_\infty -\frac{1}{2}\log(1-\epsilon)\big]\Big]\bigg]
.
\end{align}
The statistics of the number $c_\ell$ of cycles of length $\ell$ under $\E_{S_n}$ is well known~\cite[Theorem 2]{LengyelMomentsNumberCycles1997}. 
In particular, the moment of order $k\in \N$ is smaller than the same moment for a Poisson distribution of parameter $1/\ell$, which has moment generating function defined on the full real line. 
The above exponential moment is thus bounded for each $n>\ell_0$ in terms of $\ell_0,g$ only. 
Also, $\E_{S_n}[c_\ell]=1/\ell$ for each $n>\ell_0$. 
This concludes the proof:
\begin{equation}
2^{\sum_\ell\E_{S_n}[c_\ell]}\prod_{\ell=1}^{\ell_0}\E_{S_n}\bigg[\exp\Big[\ell_0^2 c_\ell\big[\sigma\|g\|_\infty -\frac{1}{2}\log(1-\epsilon)\big]\Big]\bigg]
\leq 
C(\ell_0,g)2^{\log n}
,\quad\ 
n\geq \ell_0
.
\end{equation}
\end{proof}
In the following, let $I_N$ denote the identity matrix on $\R^N$ and define $g^N$ as the matrix:
\begin{equation}
g^N(i,j)
=
g_{i,j},
\qquad
g^N(i,i)
=
0
,
\qquad 
i\neq j\in\T_N
.
\end{equation}
Write also $g^N_{\pm}$ for the matrices similarly defined in terms of $g_\pm$.
\begin{prop}[Approximate Wick theorem for $\nu^N_{g}$]\label{prop_Wick_theorem}
Assume that $g=g_+ + g_-$ satisfies the hypotheses of Lemma~\ref{lemm_bound_correlations}. 
For $n\in\N_{\geq 1}$, let $P_{2n}$ be the set of pairings of $2n$ elements. 
For each 
$I\subset\T_N$ with $|I|=2n$, 
there is then an error term $e^{I}_N$ such that:
\begin{align}
\nu^N_{g}\Big(\prod_{i\in I}\bar\eta_i\Big) 
= 
\sum_{p\in P_{2n}} \prod_{e=(i,p(i))\text{ pair in }p} \big(\sigma^{-1}I_N-N^{-1}g^N)^{-1}(i,p(i))+ e^{I}_N
,
\end{align}
where pairs are counted only once in the product, and with:
\begin{align}
\sup_{|J|=2n}\big|e^{J}_N\big| 
= 
O(N^{-n-1})
.
\end{align}
\end{prop}
\begin{proof}
We claim that the hypothesis on $g_+$ ensures that $\sigma^{-1}I_N-N^{-1}g^N$ is invertible for all large enough $N\in\N_{\geq 1}$. 
Indeed, the continuity of $g_-$ implies that $g^N_-$ has eigenvalues in $\R_-$, 
see e.g.~\cite[Theorem 2.3]{FerreiraEigenvaluesIntegralOperators2009}. 
On the other hand, recall the notation $\hat g^N_+$ from~\eqref{eq_def_tilde_g_N}. 
Since $g_+\in C^1_D(\T^2)$, one has:
\begin{equation}
\frac{\sigma}{N}\max_{i\neq j}\big|[\hat g^N_+-g^N_+](i,j)\big|
=
O(N^{-2})
,\qquad
\frac{\sigma}{N}\max_{i}\big|[\hat g^N_+-g^N_+](i,i)\big|
=
O(N^{-1})
.
\end{equation}
It follows that the matrix $\frac{\sigma}{N}(\hat g^N_+-g^N_+)$ has eigenvalues bounded by $C/N$ for some $C>0$. 
Therefore, for any vector $u\in\R^N$ with $\|u\|_2=1$, by definition of $\hat g^N_+$:
\begin{align}
\frac{\sigma}{N}\sum_{i,j}u_i u_j  g^N_+(i,j)
&\leq 
\frac{\sigma}{N}\sum_{i,j}u_i u_j  \hat g^N_+(i,j) + \frac{C}{N}
\leq 
1-\delta + \frac{C}{N}
.
\label{eq_matrix_inverse_exists}
\end{align}
If $N$ is large enough, the spectral radius of $\frac{\sigma}{N} g^N_+$ is thus smaller than $1-\delta/2$ and the matrix $I_N-\frac{\sigma}{N} g^N$ is invertible as claimed. 
We henceforth work with such an $N$. \\

For any $J\subset \T_N$, 
write $\bar\eta_J$ for the product $\prod_{j\in J}\bar\eta_j$. 
The key part of the proof is the observation \eqref{eq_to_prove_wick_theorem_0} below: 
the correlations $\nu^N_{g}(\bar\eta_a\bar\eta_J)$ can be expressed in terms of all other correlations of the form $\nu^N_{g}(\bar\eta_b\bar\eta_J)$ ($b\in\T_N$), 
up to a small error that depends on the cardinal $|J|$ of $J$. 
The claim of Proposition \ref{prop_Wick_theorem} is then established by recursion on $|J|$.

Let $n\in\N_{\geq 1}$. 
The set $I$ in the proposition will correspond here to sets of the form $\{a\}\cup J$ with $a\notin J$ and $J\subset\T_N$ with $|J|=2n-1$. 
Let $\bar\eta_\cdot \bar\eta_J$ denote the vector $(\bar\eta_b\bar\eta_J)_{b\in\T_N}$.  
Assume for the moment that we have proven the following averaging identity, 
for each $J$ with $|J|=2n-1$ and each $a\notin J$:
\begin{equation}
\nu^N_{g}(\bar\eta_a\bar\eta_J) 
= 
\Big[\frac{\sigma}{N}g^N\nu^N_{g}(\bar\eta_\cdot\bar\eta_J)\Big](a) + \tilde e^{J}_N(a)
,
\label{eq_to_prove_wick_theorem_0}
\end{equation}
with $\tilde e^{n,J}_N = (\tilde e^{J}_N(i))_{i\in\T_N}$ a vector satisfying:
\begin{equation}
\sup_{|J'|=2n-1}\sup_{i\notin J'}|\tilde e^{J'}_N(i)| 
= 
O(N^{-n-1})
,\qquad
\tilde e^{J'}_N(i) 
=0\quad (i\in J')
.
\label{eq_size_tilde_e_J_Wick_appendix}
\end{equation}
Let us then use~\eqref{eq_to_prove_wick_theorem_0} to prove the claim of Proposition~\ref{prop_Wick_theorem} by recursion. 
Fix $J\subset\T_N$ with $|J|=2n-1$.
Equation~\eqref{eq_to_prove_wick_theorem_0} is valid for $a\notin J$. 
Let us first extend it to general $a\in\T_N$. 
Using $(\bar\eta_\cdot)^2 = \sigma$, 
with the convention $\bar\eta_{\emptyset} := 1$:
\begin{align*}
\nu^N_{g}(\bar\eta_a\bar\eta_J) 
= 
{\bf 1}_{J}(a) \sigma\nu^N_{g}\big(\bar\eta_{J\setminus\{a\}}\big) + {\bf 1}_{J^c}(a)\sigma\big[N^{-1}g^N\nu^N_{g}(\bar\eta_\cdot\bar\eta_J)\big](a) + \tilde e^{J}_N(a)
.
\end{align*}
Rearranging terms, this becomes:
\begin{align}
\Big[\big(\sigma^{-1}I_N -  N^{-1}&g^N\big)\nu^N_{g}(\bar\eta_\cdot\bar\eta_J)\Big](a)
\nonumber \\
&= 
{\bf 1}_{J}(a)\Big(\nu^N_{g}(\bar\eta_{J\setminus\{a\}}) - \frac{1}{N}\sum_{i\neq a}g_{a,i}\nu^N_{g}(\bar\eta_i\bar\eta_J)\Big)+\sigma^{-1}\tilde e^{J}_N(a) 
\label{eq_recursion_to_compute_correlations_2}
\\
&= 
{\bf 1}_{J}(a)\big(\nu^N_{g}(\bar\eta_{J\setminus\{a\}})+O(N^{-n})\big) +  O(N^{-n-1}) 
,
\nonumber
\end{align}
where the $O(N^{-n}),O(N^{-n-1})$ are uniform on $a,J$ as given respectively by Lemma~\ref{lemm_bound_correlations} and~\eqref{eq_size_tilde_e_J_Wick_appendix}. 
We now proceed with the proof of the claim of Proposition \ref{prop_Wick_theorem} by recursion assuming~\eqref{eq_to_prove_wick_theorem_0}.

Take first $n=1$, $a\in\T_N$ and $J = \{b\}\subset\T_N$. 
We will make use of the following lemma, 
proven afterwards.
\begin{lemm}\label{lemm_LU}
Let $\mathcal M$ be the set of sequences of matrices on $\R^{N\times N}$ indexed by $N\in\N_{\geq 1}$ and defined as follows. 
$(A_N)_{N\geq N_0}\in\mathcal M$ for some $N_0\in\N_{\geq 1}$ if and only if there are sequences $(D_N)_N,(R_N)_N$ of matrices such that, 
for each $N\in \N_{\geq N_0}$, 
$D_N$ is diagonal, 
$R_N(i,i) = 0$ for $1\leq i \leq N$, 
and $A_N = D_N+N^{-1}R_N$. Moreover, $(R_N)_N$, $(D_N)_N$ are such that there is $c>0$ with:
\begin{equation}
\text{for each }N\geq N_0,\qquad \inf_{1\leq i \leq N}|D_N(i,i)|\geq c,\quad \sup_{1\leq i,j \leq N}|R_N(i,j)|\leq c^{-1}.
\end{equation}
It then holds that if $(A_N)_N\in\mathcal M$ and $A_N$ is invertible $(N\geq N_0)$, 
then $(A_N^{-1})_{N\geq N_0}\in\mathcal M$.
\end{lemm}
By Lemma~\ref{lemm_LU}, 
the matrix $\sigma^{-1} I_N -N^{-1}g^N$ has inverse satisfying:
\begin{equation}
\sup_{i\in\T_N} \big|\big(\sigma^{-1} I_N -N^{-1}g^N\big)^{-1}(i,i)\big| 
= 
O_N(1),\quad 
\sup_{i\neq j\in \T_N}\big|\big(\sigma^{-1} I_N -N^{-1}g^N\big)^{-1}(i,j)\big| 
= 
O(N^{-1})
.
\label{eq_consequence_LU}
\end{equation}
The last equation and~\eqref{eq_recursion_to_compute_correlations_2} with $J=\{b\}$ yield:
\begin{align}
\nu^N_{g}\big(\bar\eta_a\bar\eta_b\big) 
= 
\big(\sigma^{-1} I_N -N^{-1}g^N\big)^{-1}(a,b) + e^{\{b\}}_N(a)
,
\end{align}
with:
\begin{align}
e^{\{b\}}_N (a)
:= 
\big(\sigma^{-1} I_N -N^{-1}g^N\big)^{-1}(a,b)
\bigg[-\frac{1}{N}\sum_{i\neq b}g_{b,i}\nu^N_{g}(\bar\eta_i\bar\eta_b) + \sigma^{-1}\tilde e^{\{b\}}_N(b)\bigg]
.
\end{align}
The error term indeed satisfies $\sup_{a\neq b}|e^{\{b\}}_N(a)| = O(N^{-2})$ by~\eqref{eq_consequence_LU}. 
Assume now $n\geq 1$ and the claim of Proposition \ref{prop_Wick_theorem} to hold for each subset of $\T_N$ containing $2n$ elements. 
Let $J\subset\T_N$ with $|J|=2n+1$ and $a\in\T_N$. 
Then, by~\eqref{eq_recursion_to_compute_correlations_2},
\begin{align}
\nu^N_{g}(\bar\eta_a\bar\eta_J) = \sum_{j\in J} \big(\sigma^{-1}I_N - N^{-1}g^N\big)^{-1}(a,j)\nu^N_{g}(\bar\eta_{J\setminus\{j\}}) + e^{J}_{N}(a),\label{eq_to_prove_wick_theorem_5}
\end{align}
with $e^{J}_N = (e^{J}_N(i))_{i\in\T_N}$ the vector given by:
\begin{align}
e^{J\cup}_N (a)
&:= 
\sum_{j\in J}\big(\sigma^{-1} I_N -N^{-1}g^N\big)^{-1}(a,j)\bigg[-\frac{1}{N}\sum_{i\in\T_N}g_{j,i}\nu^N_{g}(\bar\eta_i\bar\eta_J) +\sigma^{-1}\tilde e^{J}_N(j)\bigg]
.
\end{align}
The recursion hypothesis and~\eqref{eq_consequence_LU} imply that $e^{J}_N(a) = O(N^{-n-2})$ uniformly on $J,a$. 
The recursion hypothesis applied to each $J\setminus\{j\}$ ($j\in J$) then turns~\eqref{eq_to_prove_wick_theorem_5} into:
\begin{align}
\nu^N_{g}(\bar\eta_a\bar\eta_J) 
&= 
\sum_{j\in J} \big(\sigma^{-1}I_N - N^{-1}g^N\big)(a,j)\sum_{p\in P_{|J\setminus\{j\}|}} \prod_{i\in J\setminus\{j\}} \big(\sigma^{-1} I_N-N^{-1}g^N\big)^{-1}(i,p(i)) + e^{J}_N(a)
.
\end{align}
Since the sum on $j\in J$ and $p\in P_{|J\setminus\{j\}|} = P_{2n}$ in the left-hand side exactly enumerates elements of the set $P_{2n+2}$ of pairings of $2n+2$ elements, 
Proposition~\ref{prop_Wick_theorem} is proven assuming~\eqref{eq_to_prove_wick_theorem_0}:
\begin{equation}
\nu^N_{g}(\bar\eta_a\bar\eta_J) 
= 
\sum_{p\in P_{2n+2}} \prod_{i\in J\cup\{a\}} \big(\sigma^{-1} I_N-N^{-1}g^N\big)^{-1}(i,p(i)) +e^{J}_N(a)
.
\end{equation}
We now prove the identity~\eqref{eq_to_prove_wick_theorem_0}. 
Let $J\subset\T_N$ with $|J|=2n-1$ ($n\geq 1$), 
and let $a\in\T_N\setminus J$. 
Recall that $\nu^N_{g}\propto \exp[2\Pi^N(g)]\nu^N_{1/2}$, 
and decompose $\Pi^N(g)$ as follows:
\begin{equation}
2\Pi^N(g) 
= 
\frac{\bar\eta_a}{N}\sum_{i\neq a}\bar\eta_i g_{a,i} + 2G_{\{a\}^c}(\eta),
\qquad 
G_{\{a\}^c}(\eta) 
= 
\frac{1}{4N}\sum_{i\neq j \in\T_N\setminus\{a\}}\bar\eta_i\bar\eta_j g_{i,j}
.
\label{eq_splitting_Pi_N_en_G_moins_a}
\end{equation}
Using $e^x = 1+x\int_0^1 e^{tx}dx$ with $x=\frac{\bar\eta_a}{N}\sum_{i\neq a}\bar\eta_ig_{a,i}$, 
we find:
\begin{align}
\mathcal Z^N_g\nu^N_{g}(\bar\eta_a\bar\eta_J) 
= 
\nu^N_{1/2}\bigg(e^{2G_{\{a\}^c}}\bar\eta_J \bar\eta_a\Big[1+ \frac{\bar\eta_a}{N}\sum_{i\neq a}\bar\eta_i g_{a,i} \int_0^1 \exp\Big[\frac{t\bar\eta_a}{N}\sum_{i\neq a}\bar\eta_ig_{a,i}\Big]dt\Big]\bigg)
.
\label{eq_expansion_recursion_Wick}
\end{align}
Under the product measure $\nu^N_{1/2}$, 
$\bar\eta_a$ is independent of $\bar\eta_J, G_{\{a\}^c}$ (and has average 0). 
It follows that $\nu^N_{1/2}\big(\bar\eta_a\bar\eta_J\exp[2G_{\{a\}^c}]\big)$ vanishes. 
Since also $(\bar\eta_a)^2 = \sigma$,~\eqref{eq_expansion_recursion_Wick} becomes, 
reconstructing $2\Pi^N(g)$ from $2G_{\{a\}^c}$ through~\eqref{eq_splitting_Pi_N_en_G_moins_a}:
\begin{align}
\nu^N_{g}(\bar\eta_a\bar\eta_J) 
= 
\sigma \nu^N_{g}\bigg(\frac{\bar\eta_J }{N}\sum_{i\neq a}\bar\eta_i g_{a,i} \int_0^1  \exp\Big[\frac{(t-1)\bar\eta_a}{N}\sum_{i\neq a}\bar\eta_ig_{a,i}\Big]\, dt \bigg)
.
\label{eq_to_prove_wick_theorem_1}
\end{align}
To obtain~\eqref{eq_to_prove_wick_theorem_0}, 
we further expand the exponential in~\eqref{eq_to_prove_wick_theorem_1}, 
using the following formula: 
for each $t\in[0,1]$,
\begin{align}
\exp&\Big[\frac{(t-1)\bar\eta_a}{N}\sum_{i\neq a}\bar\eta_ig_{a,i}\Big] 
= 
1+ T^1_t(\eta) + T^2_t(\eta),
\end{align}
with:
\begin{align}
T^1_t(\eta) 
&:= 
\sum_{p=1}^{2n} \frac{1}{p!}\big((t-1)\bar\eta_a\big)^p \Big(\frac{1}{N}\sum_{i\neq a}\bar\eta_i g_{a,i}\Big)^p \nonumber\\
T^2_t(\eta) 
&:= 
\int_0^1 \frac{s^{2n}}{(2n)!}\big((t-1)\bar\eta_a\big)^{2n+1} \Big(\frac{1}{N}\sum_{i\neq a}\bar\eta_i g_{a,i}\Big)^{2n+1} \exp\Big[\frac{s(t-1)\bar\eta_a}{N}\sum_{i\neq a}\bar\eta_ig_{a,i}\Big]ds.\label{eq_to_prove_wick_theorem_2}
\end{align}
Injecting this decomposition in~\eqref{eq_to_prove_wick_theorem_1} yields:
\begin{align}
\nu^N_{g}(\bar\eta_a\bar\eta_J) 
&= 
\Big[\frac{\sigma}{N}g^N\nu^N_{g}\big(\bar\eta_\cdot\bar\eta_J\big)\Big](a) 
+\sigma \nu^N_{g}\bigg(\int_0^1  \frac{\bar\eta_J }{N}\sum_{i\neq a}\bar\eta_i g_{a,i}\big(T^1_t(\eta) + T^2_t(\eta)\Big)\, dt\bigg)
.
\end{align}
The first term in the right-hand side above is exactly the first term in the right-hand side of~\eqref{eq_to_prove_wick_theorem_0}. 
We thus need to prove that the contributions of $T_t^1,T_t^2$ are of order $O(N^{-n-1})$ uniformly on $J,a$. 

Both $T^1_t$ and $T^2_t$ are readily estimated through Lemma~\ref{lemm_bound_correlations}. 
For $T^1_t$, 
one has for each $t\in[0,1]$ and each $1\leq p \leq 2n$:
\begin{equation}
\nu^N_{g}\bigg(\frac{\bar\eta_J}{N}\sum_{i\neq a}\bar\eta_i g_{a,i}\cdot \frac{1}{p!}\big((t-1)\bar\eta_a\big)^p \Big(\frac{1}{N}\sum_{i\neq a}\bar\eta_i g_{a,i}\Big)^p\bigg) = \begin{cases}
O(N^{-n-(p+1)/2})\quad &\text{if }p\text{ is odd},\\
O(N^{-n-p/2})\quad &\text{if }p\text{ is even,}
\end{cases}\label{eq_to_prove_wick_theorem_3}
\end{equation}
and the estimate above is uniform in $t\in[0,1]$, $a\in\T_N$ and sets $J$ with cardinal $2n-1$. 

For $T^2_t$, 
bounding $\bar\eta_a,\bar\eta_J$ by $1$ and the exponential by $e^{\|g\|_\infty}$, 
one similarly concludes on the existence of $C_{2n+1}>0$ such that:
\begin{align}
\bigg|\nu^N_{g}\bigg[\bar\eta_J&\frac{\bar\eta_a}{N}\sum_{i\neq a}\bar\eta_i g_{a,i}\int_0^1 T^2_t(\eta)\, dt\bigg]\bigg|
\nonumber\\
&\qquad \leq 
C_{2n+1} e^{\|g\|_\infty}\nu^N_{g}\bigg[\Big( \frac{1}{N}\sum_{i\neq a}\bar\eta_i g_{a,i}\Big)^{2n+2}\bigg] 
= 
O(N^{-n-1})
.
\end{align}
The error term is again uniform on $t,a,J$. 
This concludes the proof of \eqref{eq_to_prove_wick_theorem_0}:
\begin{equation}
\nu^N_{g}(\bar\eta_a\bar\eta_J) = \sigma\big[N^{-1}g^N\nu^N_{g}(\bar\eta_\cdot\bar\eta_J)\big](a) + \tilde e^{J}_N(a),\label{eq_to_prove_wick_theorem_4}
\end{equation}
where $\tilde e^{J}_N$ satisfies $\sup_{|J|=2n-1}\sup_{i\notin J}|\tilde e^J_N(i)| = O(N^{-n-1})$ and is defined as follows:
\begin{align}
\tilde e^J_N(a) 
&:= 
{\bf 1}_{J^c}(a)
\int_0^1 \nu^N_{g}\bigg[\bar\eta_J\frac{\bar\eta_a}{N}\sum_{i\neq a}\bar\eta_i g_{a,i} \Big(T^1_t(\eta) + T^2_t(\eta)\Big)\bigg] \,dt
.
\end{align}
\end{proof}
%
\begin{proof}[Proof of Lemma~\ref{lemm_LU}]
Let $N_0\in\N_{\geq 1}$ and $(A_N)_{N\geq N_0}\in\mathcal M$ be a family of invertible matrices. 
Then $(A_N(i,j))_{1\leq i,j\leq n}$ is also invertible for each $n\leq N$. 
This implies (see Theorem 4.3.1 in~\cite{Ciarlet1989}) the existence of a lower triangular matrix $L_N$ with diagonal equal to $1$, 
and an upper triangular matrix $U_N$ such that:
\begin{equation}
A_N 
= 
L_NU_N
,\qquad 
N\geq N_0
.
\end{equation}
$L_N$ is invertible by assumption, 
thus $U_N$ is invertible by invertibility of $A_N$. 
The claim of Lemma~\ref{lemm_LU} therefore boils down to proving that $(U_N^{-1}L_N^{-1})_N \in\mathcal M$. 
Let us first prove that $(L_N)_N,(U_N)_N\in\mathcal M$. 
The coefficients of these matrices can be computed explicitly: 
for each $N\geq N_0$ and each $1\leq i \leq 2N-1$,
\begin{align}
U_N(i,j) 
&= 
A_N(i,j) - \sum_{\ell=1}^{i-1}L_N(i,\ell) U_N(\ell,j)\quad \text{for }j\geq i,\\
L_N(j,i) 
&= 
\frac{1}{U_N(i,i)}\Big(A_N(j,i) - \sum_{\ell=1}^{i-1}L_N(j,\ell)U_N(\ell,i)\Big)\quad \text{for }j>i.
\end{align}
In particular, one can check that $(L_N)_N,(U_N)_N$ indeed belong to $\mathcal M$. 

Let us now check that $(L_N^{-1})_N, (U_N^{-1})_N$ also belong to $\mathcal M$. 
The inverse $L_N^{-1}$ of $L_N$ is also lower triangular, 
and can be computed straightforwardly: for $N\geq N_0$,
\begin{equation}
\forall i\geq j,\qquad [L_N^{-1}](i,j) 
= 
\frac{1}{A_N(i,i)}\Big(\delta_{i,j}-\sum_{\ell=1}^{i-1}L_N(i,\ell)[L_N^{-1}](\ell,j)\Big)
,
\end{equation}
where:
\begin{equation}
\delta_{i,j} 
:= 
\begin{cases}
1\quad &\text{if }i=j,\\
0\quad &\text{otherwise}.
\end{cases}
\end{equation}
The same result holds for the upper-triangular $U_N$. 
A recursion thus gives $(L_N^{-1})_N,(U_N^{-1})_N\in\mathcal M$. 
It remains to notice that $\mathcal M$ is stable under matrix product: 
$(B_N)_N,(C_N)_N\in\mathcal M$ implies $(B_NC_N)_N\in\mathcal M$. 
This concludes the proof of Lemma~\ref{lemm_LU}.
\end{proof}
The following proposition gives exponential concentration results under $\nu^N_{g}$. 
These are useful when applying the entropy inequality.
\begin{prop}\label{prop_concentration_exponentielle}
Assume that $g$ satisfies the hypotheses of Lemma \ref{lemm_bound_correlations}. 
Let $J\subset \Z$ contain $0$, $n\in\N$ and $\phi_n: \T^{n}_N\rightarrow\R$. 
We assume $N$ to be large enough that $J$ can be viewed as a subset of $\T_N$. 
Define the functions:
\begin{align}
X^{n,J}_{\phi_n}(\eta) 
:= 
\sum_{(i_1,...,i_n)\in (\T_N)^{n}}\phi_n(i_1,...,i_n)\Big(\prod_{j\in J}\bar\eta_{i_1+j}\Big) \bar\eta_{i_2}...\bar\eta_{i_n}
,\qquad
\eta\in\Omega_N
.
\label{eq_def_U_n_J}
\end{align}
Let $W^{n,J}_{\phi_n} := N^{-(n-1)}X^{n,J}_{\phi_n}$ if $n\geq 2$, and $W^{1,J}_{\phi_1} := N^{-1/2}X^{0,J}_{\phi_1}$.

Then, for any $n\geq 0$, there are constants $\gamma^{n,J},C^{n,J}>0$ depending on $n,J$ but independent of $\phi_n,N$, 
such that:
\begin{align}
\forall N\in\N_{\geq 1},\forall \gamma<\gamma^{n,J},
\qquad
\log \nu^N_{g}\Big[\exp\Big[\frac{\gamma |W^{n,J}_{\phi_n}|}{\sup_{N}\|\phi_n\|_{\infty}}\Big]\Big]
\leq 
C^{1,J}{\bf 1}_{n=1}+ {\bf 1}_{n\geq 2}\frac{C^{n,J}}{N^{\frac{n-2}{2}}}
.
\end{align}
\end{prop}
\begin{proof}
The proof can be found in Theorem A.3 and Corollary A.4 in~\cite{Correlations2022}. 
It is written for negative kernels only but this is just used to reduce to the case $g=0$. 
The assumption that $\sigma g_+$ has largest eigenvalue $\lambda_{\max}$ strictly below $1$ similarly enables one to reduce to this case. 
Indeed, for $\epsilon_g>0$ such that $(1+\epsilon_g)\lambda_{\max}<1$, 
$\sup_N\mathcal Z^N_{(1+\epsilon_g)g}<\infty$ by Lemma~\ref{lemm_bound_correlations}. 
The identity $\E[X]=\int_0^\infty\Prob(X>t)\, dt$ for a non-negative random variable $X$ and a change of variable thus give:
\begin{align}
\nu^N_g\Big[\exp\Big[\frac{\gamma |W^{n,J}_{\phi_n}|}{\sup_{N}\|\phi_n\|_{\infty}}\Big]\Big]
&\leq 
1+
\int_0^\infty e^{t}\nu^N_g\Big(\frac{\gamma |W^{n,J}_{\phi_n}|}{\sup_{N}\|\phi_n\|_{\infty}}>t\Big)
\, dt
\nonumber\\
&\leq 
1+
\big(\mathcal Z^N_{(1+\epsilon_g)g}\big)^{\frac{1}{1+\epsilon_g}}\int_0^\infty e^{t}\nu^N_{1/2}\Big(\frac{\gamma |W^{n,J}_{\phi_n}|}{\sup_{N}\|\phi_n\|_{\infty}}>t\Big)^{\frac{\epsilon_g}{1+\epsilon_g}}
\, dt
.
\end{align}
\end{proof}
\subsection{Concentration for the uniform measure}\label{app_concentration_uniform_measure}
Recall that $\Omega_{N,m}$ is the set of configurations with $m$ particles on the torus $\T_N$ ($0\leq m\leq N$). 
Recall also that $U_m$ denotes the uniform measure on $\Omega_{N,m}$. 
Concentration results under $U_{m}$ are given in Lemma~\ref{lemm_exp_mom_uniform}, 
relying on the following logarithmic Sobolev inequality.
\begin{lemm}[Theorem 4 in~\cite{LeeYau_LSI_RW1998}]\label{lemm_LSI_SSEP}
For $m\in\{0,...,N\}$, 
let $D^m_{ex}$ denote the Dirichlet form of symmetric simple exclusion with $m$ particles: 
\begin{align}
D^m_{ex}(g)
&=
\frac{1}{2}U_m\Big[\sum_{i\in\T_N} \big[g(\eta^{i,i+1})-g(\eta)]^2\Big]
,\qquad 
g:\Omega_{N,m}\rightarrow\R
.
\end{align}
There is $C_{LS}>0$ independent of $m,N$ such that, for any density $f$ for $U_m$:
\begin{equation}
U_m[f\log f]
\leq 
C_{LS}N^2 D^m_{ex}(\sqrt{f})
.
\end{equation}
\end{lemm}
For $\phi:\T_N\rightarrow\R$, 
define its discrete gradient at point $i\in\T_N$ by:
\begin{equation}
\partial^N\phi(i)
=
N[\phi(i+1)-\phi(i)]
.
\end{equation}
The log-Sobolev inequality provides concentration results for Lipschitz functions.
\begin{lemm}\label{lemm_exp_mom_uniform}
Let $\phi:\T_N\rightarrow\R$ be a Lipschitz function, 
in the sense that:
\begin{equation}
\sup_N\|\partial^N\phi\|_\infty
<
\infty
.
\end{equation}
Similarly to~\eqref{prop_concentration_exponentielle}, 
define:
\begin{equation}
W_m (\eta)
=
W^{1,\{0\}}_{\phi,m}(\eta)
:= 
\frac{1}{N^{1/2}}\sum_{i\in\T_N}\phi(i)\bar\eta^m_i
.
\end{equation}
There is then $\gamma_0>0$ independent of $m,\phi$ such that, for any $\gamma<\gamma_0$:
\begin{equation}
U_m\bigg(\exp \Big[\frac{\gamma (W_m)^2}{\|\partial^N\phi\|^2_\infty}\Big]\bigg)
\leq 
3
.
\label{eq_exp_moment_uniform}
\end{equation}
\end{lemm}
\begin{proof}
Note that, for any $i\in\T_N$, 
one has:
\begin{equation}
W_m(\eta^{i,i+1}) - W_m(\eta)
=
(\eta_i-\eta_{i+1})\frac{\partial^N\phi(i)}{N^{3/2}}
.
\end{equation}
The Herbst argument (see e.g. Proposition 2.3 in Ledoux's course~\cite{LedouxConcentrationMeasureLogarithmic1999}) then gives, 
for some $C>0$ independent of $m,\phi$:
\begin{equation}
U_m\bigg(\exp \Big[\frac{\theta W_m}{\|\partial^N\phi\|_\infty}\Big]\bigg)
\leq 
e^{C\theta^2/2}
,
\qquad 
\theta\in\R
.
\end{equation}
The identity $\E[e^X] \leq 1 + \int_0^\infty e^t\Prob(X>t)\, dt$ applied to $X=\gamma (W_m^2)/\|\partial^N\phi\|_\infty^2$ for $\gamma$ small enough depending only on $C$ and a Chernov bound conclude the proof.
\end{proof}
\section{Integration by parts formula}\label{sec_IPP}
Let $\tilde \lambda\in\R$ and $\tilde h\in C^\infty_D(\T^2)$ be symmetric. 
In this section, 
we provide an integration by parts formula under $\nu^N_{\tilde h}$ and $U_{m,\tilde h}$ for $0\leq m\leq N$. 
Let $f:\Omega_N\rightarrow\R$, 
and define $\Gamma^{i,i+1}_{\tilde h,\tilde\lambda}(f)$ for $i\in\T_N$ as the quantity:
\begin{equation}
\Gamma^{i,i+1}_{\tilde h,\tilde\lambda}(f)(\eta) 
:= 
\frac{1}{2}c_{\tilde h,\tilde\lambda}(\eta,\eta^{i,i+1})\big[\nabla_{i,i+1}f(\eta)\big]^2
,
\end{equation}
with:
\begin{equation}
\nabla_{i,i+1} f(\eta) = f(\eta^{i,i+1})-f(\eta),
\qquad
\eta\in\Omega_N
.
\end{equation}
\begin{lemm}{(Integration by parts under $\nu^N_{\tilde h}$)}\label{lemm_IPP}
Let $i\in\T_N$, let $f$ be a density for 
$\nu^N_{\tilde h}$ and let $u:\Omega_N\rightarrow\R$ satisfy $u(\eta^{i,i+1}) = u(\eta)$, 
i.e. $u$ is invariant under a jump between sites $i,i+1$. 
There is then a constant $C=C(\tilde h,\tilde \lambda)>0$ such that, 
for any $\alpha>0$:
\begin{align}
\nu^N_{\tilde h}\big( fu (\bar\eta_{i+1}-\bar\eta_i)\big)
&\leq
\alpha N^2\nu^N_{\tilde h}\big( \Gamma^{i,i+1}_{\tilde h,\tilde \lambda}(\sqrt{f})\big)  + \frac{C}{\alpha N^2} \nu^N_{\tilde h}\big(f|u|^2\big)
\nonumber\\
&\quad 
+ \nu^N_{\tilde h}\bigg( fu(\bar\eta_{i+1}-\bar\eta_i) \Big(\exp \Big[-\frac{2(\bar\eta_{i+1}-\bar\eta_i)}{N} B^{\tilde h}_i\Big]-1\Big)\bigg)
.
\end{align}
\end{lemm}
\begin{rmk}\label{rmk_IPP_uniform_measure}
The claim of the lemma also holds if $\nu^N_{\tilde h}$ is replaced by the uniform measure $U_{m}$ on configurations with $0\leq m\leq N$ particles, 
with the same proof.
\demo
\end{rmk}
\begin{proof}
A direct computation using the fact that $\eta\mapsto\eta^{i,i+1}$ is bijective gives:
\begin{equation}
\nu^N_{\tilde h}\big( fu (\bar\eta_{i+1}-\bar\eta_i)\big) 
=
\frac{1}{2}\sum_{\eta\in\Omega_N}\nu^N_{\tilde h}(\eta)u(\eta) (\bar\eta_{i+1}-\bar\eta_i)\Big[f(\eta) - \frac{\nu^N_{\tilde h}(\eta^{i,i+1})}{\nu^N_{\tilde h}(\eta)}f(\eta^{i,i+1})\Big]
\end{equation}
and, for each $\eta\in\Omega_N$:
\begin{equation}
\frac{\nu^N_{\tilde h}(\eta^{i,i+1})}{\nu^N_{\tilde h}(\eta)}
=
\exp \Big[-\frac{2(\bar\eta_{i+1}-\bar\eta_i)}{N} B^{\tilde h}_i\Big],
\qquad 
B^{\tilde h}_i 
= 
\frac{1}{2N}\sum_{j\notin\{i,i+1\} }\bar\eta_j \partial^{N}_1 \tilde h_{i,j}
.
\end{equation}
Thus, writing $e^x = 1 + e^x-1$ for $x = -2N^{-1}(\bar\eta_{i+1}-\bar\eta_i)B^{\tilde h}_i$ and again changing variables between $\eta$ and $\eta^{i,i+1}$:
\begin{align}
\nu^N_{\tilde h}\big( fu (\bar\eta_{i+1}-\bar\eta_i)\big)
\leq
\nu^N_{\tilde h}\big( |\nabla_{i,i+1} f||u| \big)
+ 
\nu^N_{\tilde h}\bigg( fu(\bar\eta_{i+1}-\bar\eta_i) \Big(e^{-\frac{2(\bar\eta_{i+1}-\bar\eta_i)}{N} B^{\tilde h}_i}-1\Big)\bigg)
.
\end{align}
Note that jump rates are bounded below:
\begin{equation}
c_{\tilde h,\tilde \lambda}(\eta,\eta_{i+1})
\geq 
c_0 c(\eta,\eta_{i+1}),
\qquad 
\eta\in\Omega_N,i\in\T_N
.
\end{equation}
The claim then follows from  
the lower bound on the jump rates, 
the fact that $\sup_{N,i}|B^{\tilde h}_i|<\infty$ 
and the inequality $|ab|\leq \frac{\kappa}{2}a^2 + \frac{1}{2\kappa}b^2$ applied to $\kappa=2\alpha N^2/c_0$, 
$a = \sqrt{f(\eta^{i,i+1})}-\sqrt{f(\eta)}$ and $b = u(\sqrt{f(\eta^{i,i+1})}+\sqrt{f(\eta)})$.
\end{proof}
\section{On some partial differential equations involving the bias $h$}\label{app_BFP}
\subsection{Existence and regularity of the bias $h$}
In this section, we prove Proposition \ref{prop_solving_PDE_h}, 
i.e. the existence and properties of $h\in C(\T)\cap C^\infty([0,1])$ solving:
\begin{equation}
\begin{cases}
&\displaystyle{h''(x) -\frac{\sigma}{2} \int_{\T} h'(x-y)h'(y)\, dy =0\quad} \text{for }x\in(0,1),
\\
&\displaystyle{h'(0_+)-h'(1_-) = 2\lambda(\lambda+2E).}
\end{cases}
\label{eq_ODE_sur_h_appendix}
\end{equation}
Define:
\begin{equation}
\mathcal H
:= 
\Big\{\psi \in \mathbb H^1(\T): \int_{\T} \psi(x)\, dx=0\Big\}
.
\end{equation}
Then $\mathcal H$ equipped with the norm $\psi\mapsto\|\psi'\|_2$ is a Hilbert space. 
Upon integrating~\eqref{eq_ODE_sur_h_appendix} against a test function $f\in C^\infty(\T)$ and integrating by parts,
~\eqref{eq_ODE_sur_h_appendix} becomes:
\begin{equation}
\int_{\T} h(x)f''(x)\, dx -2\lambda(\lambda+2E)f(0) -  \frac{\sigma}{2}\int_{\T^2}f''(x)h(x-y)h(y)\, dx\,  dy
=
0
.
\label{eq_weak_ODE_h_appendix}
\end{equation}
\begin{prop}\label{prop_existence_regularity_h}
Recall the definition of sub-criticality in Proposition~\ref{prop_solving_PDE_h}. 
There is a unique family $(h_{\lambda,E})_{\lambda,E\text{ sub-critical}}$ of solutions of~\eqref{eq_ODE_sur_h} in $\mathcal H$. 
These solutions are in $C^0(\T)\cap C^\infty([0,1])$ and:
\begin{equation}
(\lambda,E)\mapsto \|h_{\lambda,E}\|_{2}\text{ is a continuous function vanishing on the line }\lambda=0
.
\end{equation}
For each sub-critical $\lambda,E$, the Fourier series of $h_{\lambda,E}$ reads:
\begin{equation}
\forall x\in\T,\qquad 
h(x) 
= 
\frac{\sqrt{2}}{\sigma} \sum_{\ell\geq 1}\Big(1-\Big[1+\frac{\sigma\lambda(\lambda+2E)}{\pi^2\ell^2}\Big]^{1/2}\Big)\sqrt{2}\cos(2\pi\ell x) 
.
\label{eq_explicit_formula_h}
\end{equation}
Moreover, the kernel $(x,y)\mapsto h_{\lambda,E}(x-y)$ is positive if $\lambda(\lambda+2E)\leq 0$, 
negative otherwise, 
with eigenvalues given by its Fourier coefficients.
\end{prop}
\begin{proof}
We look for a solution $h$ to~\eqref{eq_weak_ODE_h_appendix} in $\mathbb L^2(\T)$ by computing its Fourier decomposition. 
For $\ell\in\Z$, let $c_\ell(h) = \int_{\T} h(x)e^{-2\iota\pi\ell x}\, dx$, with $\iota^2=-1$. 
We look for $h$ with $c_0(h)=0$ by assumption. 
Note that, for any $\psi,\tilde \psi\in\mathbb L^2(\T)$, 
one has:
\begin{equation}
c_\ell\Big(\int_{\T} \psi(x-y)\tilde\psi(y)\, dy\Big) 
=
c_{\ell}(\psi)c_\ell(\tilde \psi)
.
\end{equation}
Taking $f=e^{-2\iota\pi\ell\cdot}$ for $\ell\neq 0$,  
Equation~\eqref{eq_weak_ODE_h_appendix} becomes:
\begin{equation}
-(2\pi\ell)^2 c_\ell(h) - 2\lambda(\lambda+2E) + \frac{\sigma}{2}(2\pi\ell)^2 c_\ell(h)^2
= 
0
.
\end{equation}
This gives:
\begin{equation}
c_\ell(h)
=
\frac{1}{\sigma}\Big(1+\eta_\ell \sqrt{1+\frac{\sigma\lambda(\lambda+2E)}{\pi^2\ell^2}}\Big),
\qquad 
\eta_\ell\in\{-1,1\} 
.
\end{equation}
Note that $h\in\mathbb L^2(\T)\supset\mathcal H$ only if all $\eta_\ell$ are equal to $-1$ for $\ell$ large enough. 
In addition, we must have $\eta_\ell=-1$ for any $\ell\neq 1$ in order for $h$ to vanish on the line $\lambda=0$. 
This gives existence and uniqueness of a family $(h_{\lambda,E})$ of solutions to~\eqref{eq_weak_ODE_h_appendix} in $\mathcal H$ such that $(\lambda,E)\mapsto\|h_{\lambda,E}\|_2$ is continuous and vanishes on the line $\lambda=0$. 
Note further that $c_\ell(h_{\lambda,E})=o(\ell^{-2})$, 
thus $h_{\lambda,E}\in\mathbb H^1(\T)\subset C^0(\T)$ by Dirichlet's normal convergence theorem and~\eqref{eq_explicit_formula_h} holds pointwise. 
Moreover, the Fourier coefficients $c_\ell(h_{\lambda,E})$ are the eigenvalues of the kernel operator $h_{\lambda,E}$ all have the same sign, 
the sign of $-\lambda(\lambda+2E)$. 

It remains to observe that $h_{\lambda,E}\in C^\infty([0,1])$. 
Indeed, defining $h_{\lambda,E}''$ on $(0,1)$ through~\eqref{eq_ODE_sur_h_appendix} gives $h_{\lambda,E}''\in\mathbb H^1((0,1))$, 
thus $h_{\lambda,E}'\in\mathbb H^2((0,1))$, 
with $\mathbb H^n((0,1))$ the usual Sobolev space of functions with $n$ weak derivatives in $\mathbb L^2((0,1))$. 
Iterating we get $h_{\lambda,E}\in\mathbb H^n((0,1))$ for any $n\in\N$, 
thus $h_{\lambda,E}\in C^\infty([0,1])$ by Sobolev embedding (Theorem 4.12 in \cite{Adams2003}).
\end{proof}

\subsection{Forwards and backwards Fokker-Planck equations}
\label{app_sub_BFP}
Here, we establish well-posedness and regularity for solutions of certain Fokker-Planck equations. 
These are used to characterise limiting fluctuations in Section~\ref{sec_martingale_problem}. \\
For a symmetric $\tilde h\in C^\infty_D(\T^2)$, 
recall that $\mathcal L_{\tilde h}$ is the operator acting on $f\in C^\infty(\T)$ according to:
\begin{align}
\mathcal L_{\tilde h}f(x)
&= 
f''(x) + \sigma\int_{\T}\partial_2 \tilde h(x,y)f'(y)dy 
\nonumber\\
&=
f''(x)- \sigma m_{\tilde h} f(x) - \sigma \int_{\T}\partial^2_2 \tilde h(x,y)f(y)\, dy
,\qquad
x\in\T
,
\label{eq_def_L_y_h_lambda_macro_apppendix}
\end{align}
with $m_{\tilde h}\in C^\infty(\T)$ the function:
\begin{equation}
m_{\tilde h} (x)
:= 
\big(\partial_1 - \partial_2\big) \tilde h(x_-,x)
,\quad 
x\in\T
.
\end{equation}
For $f\in C^\infty(\T)$ and $t\geq 0$, consider the following forwards and backwards Fokker-Planck equations:
\begin{equation}
\begin{cases}
\partial_s u_s - \mathcal L_{\tilde h}u_s = 0\quad \text{for }s\leq t,\\
u_0 = f,
\end{cases}
\qquad 
\begin{cases}
\partial_s v_s + \mathcal L_{\tilde h}v_s = 0\quad \text{for }s\leq t,\\
v_t = f.
\end{cases}
\label{eq_BFP_equations}
\end{equation}
\begin{prop}\label{prop_existence_sol_BFP}
Let $f\in C^\infty(\T)$ and $t\geq 0$. 
Then~\eqref{eq_BFP_equations} admits solutions $u,v$ in $C^\infty([0,t]\times \T)$.
\end{prop}
\begin{proof}
Let $f\in C^\infty(\T)$. 
The Laplacian operator $\Delta$ generates a strongly continuous semi-group on each of the Sobolev spaces $\mathbb H^\ell(\T)$ of functions with $\ell$ weak derivatives in $\mathbb L^2(\T)$ ($\ell\in\N$). 
On the other hand, $\tilde h\in C^\infty_D(\T^2)$ means that
the operator $f\mapsto m_{\tilde h}f -\sigma \partial^2_2\tilde hf$ is bounded in $\mathbb H^\ell(\T)$. 
By Theorem III.1.3 in \cite{Engel2000}, 
it follows that $\mathcal L_{\tilde h}$ generates a strongly continuous semi-group $(\mathcal P^*_{\tilde h}(s))_{s\geq 0}$ on each $\mathbb H^\ell(\T)$ ($\ell\in\N$). 
In particular, $(u_s)_{s\leq t} := (\mathcal P^\infty_{\tilde h}(s)f)_{s\leq t}$ is in $C^\infty([0,t]\times \T)$ and solves the first equation in \eqref{eq_BFP_equations}. 
Setting $v_s := u_{t-s}$ for $s\leq t$ yields a solution of the second equation in \eqref{eq_BFP_equations} with the desired regularity.
\end{proof}

\bibliographystyle{abbrv}
\bibliography{all_refsbibtex}

\end{document}